\documentclass[11pt,reqno]{amsart}
\usepackage{amsmath,amssymb,mathrsfs}
\usepackage{graphicx,cite,enumitem}
\usepackage{subfigure}
\usepackage{epstopdf}
\usepackage{graphics} 
\usepackage{epsfig} 
\usepackage{tikz}
\usepackage{paralist}
\usepackage{booktabs}
\usepackage{enumerate}

\newtheorem{theorem}{Theorem}[section]

\newtheorem{problem}{Problem}

\numberwithin{equation}{section}

\begin{document}

\title[inverse obstacle scattering for elastic waves]{Inverse obstacle
scattering problem for elastic waves with phased or phaseless far-field data}

\author{Heping Dong}
\address{School of Mathematics, Jilin University, Changchun,  Jilin 130012, P. R. China}
\email{dhp@jlu.edu.cn}

\author{Jun Lai}
\address{School of Mathematical Sciences, Zhejiang University
	Hangzhou, Zhejiang 310027, China}
\email{laijun6@zju.edu.cn}

\author{Peijun Li}
\address{Department of Mathematics, Purdue University, West Lafayette, Indiana
	47907, USA}
\email{lipeijun@math.purdue.edu}

\thanks{The research of HD was supported in part by the NSFC grants 11801213 and
11771180. The research of JL was partially supported by the Funds for Creative
Research Groups of NSFC (No. 11621101), the Major Research Plan of NSFC (No.
91630309), NSFC grant No. 11871427 and The Fundamental Research Funds for the
Central Universities.}



\begin{abstract}
This paper concerns an inverse elastic scattering problem which is to
determine the location and the shape of a rigid obstacle from the phased or
phaseless far-field data for a single incident plane wave. By
introducing the Helmholtz decomposition, the model problem is reduced to a
coupled boundary value problem of the Helmholtz equations. The relation
is established between the compressional or shear far-field pattern for the
elastic wave equation and the corresponding far-field pattern for the coupled
Helmholtz equations. An efficient and accurate Nystr\"{o}m type discretization
for the boundary
integral equation is developed to solve the coupled system. The
translation invariance of the phaseless compressional and shear far-field
patterns are proved. A system of nonlinear integral equations is proposed and
two iterative reconstruction methods are developed for the inverse problem. In
particular, for the phaseless data, a reference ball technique is
introduced to the scattering system in order to break the translation
invariance. Numerical experiments are presented to demonstrate the effectiveness
and robustness of the proposed method.
\end{abstract}

\maketitle

\section{Introduction}

Scattering problems for elastic waves have significant applications in 
seismology and geophysics \cite{LL-86}. As an important research topic in
scattering theory, the inverse obstacle scattering problem (IOSP) is to identify
unknown objects that is not accessible by direct observation
through the use of waves. The IOSP for elastic waves have continuously
attracted much attention by many researchers. The recent development can be
found in \cite{ABG-15} on mathematical and numerical methods for solving the
IOSP in elasticity imaging. 

The phased IOSP refers to as the IOSP by using full scattering data
which contains both phase and amplitude information. The phased IOSP for
elastic waves has been extensively studied and a great deal of mathematical and
numerical results are available. In \cite{L-IP15, L-SIAP12, LWWZ-IP16}, the
domain derivatives were investigated by using either the boundary integral
equation method or the variational method. In \cite{YLLY}, based on the
Helmholtz decomposition, the boundary value problem of the Navier equation was
converted into a coupled boundary value problem of the Helmholtz equations. A
frequency continuation method was developed to reconstruct the shape of the
obstacles. We refer to  \cite{HH-IP93, EY-IP10} for the uniqueness results on
the inverse elastic obstacle scattering problem. Related work can be found in
\cite{LaiLi, BC-IP05, HKS-IP13, KS-JE15, K-IP96, NU-SIAP95, HLLS-SJIS14,
LWZ-IP15} on the general inverse scattering problems for elastic waves.

In many practical applications, the phase of a signal either can be very
difficult to be measured or can not be measured accurately compared with its
amplitude or intensity. Thus it is often desirable to solve the problems with
phaseless data, which are called phase retrieval problems in optics, or physical
and engineering sciences. Due to the translation invariance property of the
phaseless wave field, it is impossible to reconstruct uniquely the location of
the unknown objects, which makes the phaseless inverse scattering problems much
more difficult compared to the phased case. Various numerical methods have been
proposed to solve
the phaseless IOSP for acoustic waves governed by the scalar Helmholtz
equation. In \cite{RW1997}, Kress and Rundell proposed a Newton's iterative
method for imaging a two-dimensional sound-soft obstacle from the phaseless
far-field data with one incident wave. A nonlinear integral equation method was
developed in \cite{Ivanyshyn2007, OR2010} for the two- and three-dimensional
shape reconstruction from a given incident field and the modulus of the
far-field data, respectively. The nonlinear integral equation method proposed by
Johansson and Sleeman \cite{TB2007} was extended to reconstruct the shape of a
sound-soft crack by using phaseless far-field data for one incident plane wave
in \cite{GDM2018}. In addition, fundamental solution method
\cite{KarageorghisAPNUM} and a hybrid method \cite{Lee2016} were proposed to
detect the shape of a sound-soft obstacle by using of the modulus of the
far-field data for one incident field. To overcome the nonuniqueness issue,
Zhang et al. \cite{ZhangBo2017, ZhangBo2018} proposed to use superposition of
two plane waves with different incident directions as the illuminating field to
recover both of the location and the shape of an obstacle simultaneously by
using phaseless far-field data. Recently, a reference ball
technique was developed in \cite{DZhG2018} to break the translation invariance
and reconstruct both the location and shape of an obstacle from phaseless
far-field data for one incident plane wave. We refer to \cite{XZZ18, LZ2010,
ZG18} for the uniqueness results on the inverse scattering problems by using
phaseless data. For related phaseless inverse scattering problems as
well as numerical methods may be found in \cite{Ammari2016, Li2017,
CH2016, BLL2013, Bao2016, JiLiuZhang2018_1, ZGLL18, Klibanov2014,
Klibanov2017}. 

In this paper, we consider the inverse elastic scattering problem of determining
the location and shape of a rigid obstacle from phased or phaseless
far-field data for a single incident plane wave. Motivated by the recent work in
\cite{DZhG2018, GDM2018}, the reference ball technique in \cite{LiJingzhi2009,
ZG18}, and the Helmholtz decomposition in \cite{YLLY}, we propose a nonlinear
integral equation method combined with the reference ball technique to solve the
IOSP for elastic waves. In particular, for the phaseless IOSP, since the
location of the reference ball is known, the method has the capability of
calibrating the scattering system so that the translation invariance does not
hold anymore. Therefore, the location information of the obstacle can be
recovered with negligible additional computational costs. Moreover, we develop a
Nystr\"{o}m type discretization for the integral equation to efficiently and
accurately solve  the direct obstacle scattering problem for elastic waves. It
is worth mentioning
that the proposed method for phased and phaseless IOSP are extremely efficient
since we only need to solve the scalar Helmholtz equations and avoid solving the
vector Navier equations. The goal of this work is fivefold: establish the
relationship between the compressional or shear far-field pattern for the
Navier equation and the corresponding far-field pattern for the coupled
Helmholtz system; prove the translation invariance property of the phaseless
compressional and shear far-field pattern; develop a Nystr\"{o}m discretization
for the boundary integral equation to solve the direct obstacle scattering
problems for elastic waves; propose a nonlinear integral equation method  to
reconstruct the obstacle's location and shape by using far-field data for one
incident plane wave; develop a reference ball based nonlinear integral equation
method to reconstruct the obstacle's location and shape by using phaseless
far-field data for one incident plane wave.

The paper is organized as follows. In Section 2, we introduce the problem
formulation. Section 3  establishes the
relationship of the far-field patterns between the elastic wave equation and the
coupled Helmholtz system. In Section 4, a Nystr\"{o}m-type discretization is
developed to solve the coupled boundary value problem of the
Helmholtz equations. In Section 5, a nonlinear integral equation method and a
reference ball based iterative method are present to solve the phased and
phaseless IOSP, respectively. Numerical experiments are shown to demonstrate the
feasibility of the proposed methods in Section 6. The paper is concluded with
some general remarks and directions for future work in Section 7.

\section{Problem formulation}

Consider a two-dimensional elastically rigid obstacle,
which is described as a bounded domain $D\subset\mathbb R^2$ with
$\mathcal{C}^2$ boundary $\Gamma_D$. Denote by $\nu=(\nu_1, \nu_2)^\top$ and
$\tau=(\tau_1, \tau_2)^\top$ the unit normal and tangential vectors on
$\Gamma_D$, respectively, where $\tau_1=-\nu_2, \tau_2=\nu_1$. The exterior
domain $\mathbb{R}^2\setminus \overline{D}$ is assumed to be filled with a
homogeneous and isotropic elastic medium with a unit mass density.

Let the obstacle be illuminated by a time-harmonic plane wave
$\boldsymbol{u}^{\rm inc}$, which satisfies the two-dimensional Navier equation:
\[
\mu\Delta\boldsymbol{u}^{\rm inc}
+(\lambda+\mu)\nabla\nabla\cdot\boldsymbol{ u}^{\rm inc}
+\omega^2\boldsymbol{u}^{\rm inc}=0\quad {\rm in} ~
\mathbb{R}^2\setminus \overline{D},
\]
where $\omega>0$ is the angular frequency and $\lambda, \mu$ are the
Lam\'{e} constants satisfying $\mu>0, \lambda+\mu>0$. Explicitly, we have
\[
 \boldsymbol{u}^{\rm inc}(x)=d \mathrm{e}^{{\rm i} \kappa_{\rm p}
d\cdot x}\quad \text{or} \quad \boldsymbol{u}^{\rm inc}(x)=d^{\perp} 
\mathrm{e}^{{\rm i} \kappa_{\rm s} d\cdot x},
\]
where the former is the compressional plane wave and the latter is the shear
plane wave. Here $d=(\cos\theta, \sin\theta)^\top$ is
the unit propagation direction vector, $\theta\in [0, 2\pi)$ is the incident
angle, $d^{\perp}=(-\sin\theta, \cos\theta)^\top$ is an orthonormal vector of
$d$,
and
\[
 \kappa_{\rm p}=\frac{\omega}{\sqrt{\lambda+2\mu}},\quad
\kappa_{\rm s}=\frac{\omega}{\sqrt{\mu}}
\]
are the compressional wavenumber and the shear wavenumber, respectively.

The displacement of the total field $\boldsymbol{u}$ also satisfies the Navier
equation
\[
\mu\Delta\boldsymbol{u}+(\lambda+\mu)\nabla\nabla\cdot\boldsymbol{u}
+\omega^2\boldsymbol{u}=0\quad {\rm in}\, \mathbb{R}^2\setminus \overline{D}. 	
\]
Since the obstacle is rigid, the total field $\boldsymbol u$ satisfies 
\[
\boldsymbol{u}=0\quad {\rm on}~\Gamma_D.
\]
The total field $\boldsymbol u$ consists of the incident field $\boldsymbol
u^{\rm inc}$ and the scattered field $\boldsymbol v$, i.e., 
\[
 \boldsymbol u=\boldsymbol u^{\rm inc}+\boldsymbol v. 
\]
It is easy to verify that the scattered field $\boldsymbol v$ satisfies
the boundary value problem
\begin{equation}\label{scatteredfield}
\begin{cases}
\mu\Delta\boldsymbol{v}+(\lambda+\mu)\nabla\nabla\cdot\boldsymbol{v}
+\omega^2\boldsymbol{v}=0\quad &{\rm in}~
\mathbb{R}^2\setminus\overline{D},\\
\boldsymbol{v}=-\boldsymbol{u}^{\rm inc}\quad &{\rm on}~\Gamma_D.
\end{cases}
\end{equation}
In addition, the scattered field $\boldsymbol v$ is required to satisfy
the Kupradze--Sommerfeld radiation condition
\[
\lim_{\rho\to\infty}\rho^{\frac{1}{2}}(\partial_\rho\boldsymbol v_{\rm p}-{\rm
i}\kappa_{\rm p}\boldsymbol v_{\rm p})=0,\quad
\lim_{\rho\to\infty}\rho^{\frac{1}{2}}(\partial_\rho\boldsymbol v_{\rm s}-{\rm
i}\kappa_{\rm s}\boldsymbol v_{\rm s})=0,\quad \rho=|x|,
\]
where
\[
\boldsymbol v_{\rm p}=-\frac{1}{\kappa_{\rm p}^2}\nabla\nabla\cdot\boldsymbol
v,\quad \boldsymbol v_{\rm s}=\frac{1}{\kappa_{\rm s}^2}{\bf curl}{\rm
curl}\boldsymbol v,
\]
are known as the compressional and shear wave components of $\boldsymbol v$,
respectively. Given a vector function $\boldsymbol v=(v_1, v_2)^\top$ and
a scalar function $v$, define the scalar and vector curl operators:
\[
{\rm curl}\boldsymbol v=\partial_{x_1}v_2-\partial_{x_2}v_1, \quad
{\bf curl}v=(\partial_{x_2}v, -\partial_{x_1}v)^\top.
\]

For any solution $\boldsymbol v$ of the elastic wave equation
\eqref{scatteredfield}, the Helmholtz decomposition reads
\begin{equation}\label{HelmDeco}
\boldsymbol{v}=\nabla\phi+\boldsymbol{\rm curl}\psi.
\end{equation}
Combining \eqref{HelmDeco} and \eqref{scatteredfield}, we may obtain the
Helmholtz equations
\[
\Delta\phi+\kappa_{\rm p}^{2}\phi=0, \quad \Delta\psi+\kappa_{\rm s}^{2}\psi=0.
\]
As usual, $\phi$ and $\psi$ are required to satisfy the Sommerfeld radiation
conditions
\begin{equation*}
\lim_{\rho\to\infty}\rho^{\frac{1}{2}}(\partial_{\rho}
\phi-{\rm i}\kappa_{\rm p}\phi)=0, \quad
\lim_{\rho\to\infty}\rho^{\frac{1}{2}}(\partial_{\rho}
\psi-{\rm i}\kappa_{\rm s}\psi)=0, \quad \rho=|x|.
\end{equation*}

Combining the Helmholtz decomposition and boundary condition on $\Gamma_D$
yields that
\[
 \boldsymbol{v}=\nabla\phi+{\bf curl}\psi=-\boldsymbol{u}^{\rm
inc}.
\]
Taking the dot product of the above equation with $\nu$ and
$\tau$, respectively, we get
\[
 \partial_\nu\phi+\partial_\tau\psi=f_1,\quad
\partial_\tau\phi-\partial_\nu\psi=f_2,
\]
where
\[
f_1=-\nu\cdot\boldsymbol{u}^{\rm inc},\quad
f_2=-\tau\cdot\boldsymbol{u}^{\rm inc}.
\]
In summary, the scalar potential functions $\phi, \psi$ satisfy the coupled
boundary value problem
\begin{align}\label{HelmholtzDec}
\begin{cases}
\Delta\phi+\kappa_{\rm p}^{2}\phi=0, \quad \Delta\psi+\kappa_{\rm s}^{2}\psi=0
\quad &{\rm in} ~\mathbb{R}^2\setminus\overline{D},\\
\partial_\nu\phi+\partial_\tau\psi=f_1, \quad
\partial_\tau\phi-\partial_\nu\psi=f_2 \quad &{\rm on} ~ \Gamma_D,\\
\displaystyle{\lim_{\rho\to\infty}\rho^{\frac{1}{2}}(\partial_{\rho}
\phi-{\rm i}\kappa_{\rm p}\phi)=0}, \quad
\displaystyle{\lim_{\rho\to\infty}\rho^{\frac{1}{2}}(\partial_{\rho}
\psi-{\rm i}\kappa_{\rm s}\psi)=0}, \quad &\rho=|x|.
\end{cases}
\end{align}

It is well known that a radiating solution of \eqref{scatteredfield} has the
asymptotic behavior of the form
\[
\boldsymbol{v}(x)=\frac{\mathrm{e}^{\mathrm{i}\kappa_{\rm p}|x|}}{\sqrt{|x|}}
\boldsymbol{v}_{\rm p}^\infty(\hat{x})+\frac{\mathrm{e}^{\mathrm{i}\kappa_{\rm
s}|x|}}{\sqrt{|x|}} \boldsymbol{v}_{\rm
s}^\infty(\hat{x})+\mathcal{O}\left(\frac{1}{|x|^{\frac{3}{2}}}\right),
\quad |x|\to\infty
\]
uniformly in all directions $\hat{x}:=x/|x|$, where $\boldsymbol{v}_{\rm
p}^\infty$ and $\boldsymbol{v}_{\rm s}^\infty$, defined on
the unit circle $\Omega$, are known as the compressional and shear far-field
pattern of $\boldsymbol{v}$, respectively. Define $(a_{\rm
p}, a_{\rm s}):=(1,0)$ or $(0,1)$, where $a_{\rm p}$ and $a_{\rm s}$ denote the
coefficients of compressional and shear wave respectively. The inverse
obstacle scattering problem for elastic waves can be stated as follows:

\begin{problem}[Phased IOSP]\label{problem_1}
Given an incident plane wave $\boldsymbol{u}^{\rm inc}$ for a fixed angular
frequency $\omega$, Lam\'{e} parameters $\lambda, \mu$, and a single
incident direction $d$, the IOSP is to determine the location and shape of the
boundary $\Gamma_D$ from the far-field data $(a_{\rm p}\boldsymbol{v}_{\rm
p}^\infty(\hat x),a_{\rm s}\boldsymbol{v}_{\rm s}^\infty(\hat x)), \forall\hat
x\in\Omega$, which is generated by the obstacle $D$. 
\end{problem}

\begin{figure}
\centering
\includegraphics[width=0.5\textwidth]{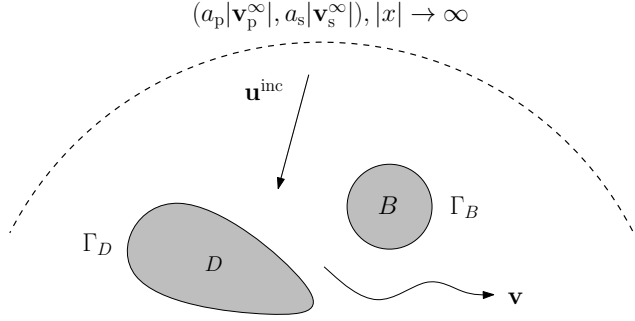}
\caption{The problem geometry of elastic obstacle scattering.}
\label{fig:illustration}
\end{figure}

In the next section, we will show that both the modulus of compressional and
shear far-field pattern have translation invariance for a shifted domain, when
only the compressional or shear plane wave is used as an incident field. It
implies that the inverse problem does not admit a unique solution by using the
phaseless far-field patterns. Our goal is to overcome this issue by introducing
a reference ball. As seen in Figure \ref{fig:illustration}, an
elastically rigid ball $B=\left\{x\in
\mathbb{R}^2: |x|<R \right\}\subset\mathbb{R}^2$ with boundary $\Gamma_B$ is
placed
next to the obstacle $D$. The domain $B$ is called the reference ball
and is used to break the translation invariance for the far-field pattern. To
this end, the phaseless inverse obstacle scattering problem is stated as
follows:

\begin{problem}[Phaseless IOSP]\label{problem_2}
Let $B=\left\{x\in \mathbb{R}^2: |x|<R \right\}\subset\mathbb{R}^2$ be an
artificially added rigid ball such that $D\cap B=\emptyset$. Given a
compressional or shear incident plane wave $\boldsymbol{u}^{\rm inc}$ for a
fixed angular frequency $\omega$, Lam\'{e} parameters $\lambda, \mu$, and a
single incident direction $d$, the IOSP is to determine the location and shape
of
the boundary $\Gamma_D$ from the phaseless far-field data $\big(a_{\rm
p}|\boldsymbol{v}_{\rm p}^\infty(\hat x)|, a_{\rm s} |\boldsymbol{v}_{\rm
s}^\infty(\hat x)|\big), \forall\hat x\in\Omega$, which is generated by the
scatterer $D\cup B$. 
\end{problem}

In the following, we shall introduce a system of nonlinear integral equations
and develop corresponding reconstruction algorithm for solving Problem
\ref{problem_1} and Problem \ref{problem_2}, respectively. 

\section{Far-field patterns}

In this section, we establish the relationship of the far-field patterns
between the scattered field $\boldsymbol v$ and the scalar
potentials $\phi, \psi$. 

Denote the fundamental solution to the two-dimensional Helmholtz equation by
\[
\Phi(x,y;\kappa)=\frac{\mathrm{i}}{4}H_0^{(1)}(\kappa|x-y|), \quad x\neq y,
\]
where $H_0^{(1)}$ is the Hankel function of the first kind with order zero. 

\begin{theorem}\label{Th1}
The radiating solution $\boldsymbol{v}$ to the Navier equation has the
asymptotic behavior 
\begin{equation*}
\boldsymbol{v}(x)=\frac{\mathrm{e}^{\mathrm{i}\kappa_{\rm
p}|x|}}{\sqrt{|x|}} \boldsymbol{v}_{\rm
p}^\infty+\frac{\mathrm{e}^{\mathrm{i}\kappa_{\rm s}|x|}}{\sqrt{|x|}}
\boldsymbol{v}_{\rm
s}^\infty+\mathcal{O}\left(\frac{1}{|x|^{\frac{3}{2}}}\right),
\quad |x|\to\infty
\end{equation*}
uniformly for all direction $\hat{x}$,  where the vectors
\begin{equation}\label{behaviour relation}
\boldsymbol{v}_{\rm p}^\infty(\hat{x})=\mathrm{i}\kappa_{\rm
p}\phi_{\infty}(\hat{x})\hat{x}, \qquad \boldsymbol{v}_{\rm
s}^\infty(\hat{x})=-\mathrm{i}\kappa_{\rm
s}\psi_{\infty}(\hat{x})\hat{x}^{\perp}
\end{equation}
defined on the unit circle $\Omega$ are the far-field patterns of
$\boldsymbol{v}_{\rm p}$ and $\boldsymbol{v}_{\rm s}$, and the complex-valued
functions $\phi_\infty(\hat{x})$ and $\psi_\infty(\hat{x})$ are the far-field
patterns corresponding to $\phi$ and $\psi$.
\end{theorem}

\begin{proof}
It follows from Green's formula in \cite[Theorem 2.5]{DR-book2013}
that we have 
	\begin{align*}
	&\phi(x)=\int_{\Gamma_D}\left\{\phi(y)\frac{\partial\Phi(x,y;\kappa_{\rm
p})}{\partial\nu(y)} - \frac{\partial\phi}{\partial\nu}(y)\Phi(x,y;\kappa_{\rm
p})\right\}\mathrm{d}s(y), \quad x\in\mathbb{R}^2\setminus\overline{D}.
	\end{align*}
	The corresponding far-field pattern is
	\begin{align*}
	&\phi_{\infty}(\hat{x})=\gamma_{\kappa_{\rm
p}}\int_{\Gamma_D}\left\{\phi(y)\frac{\partial\mathrm{e}^{-\mathrm{i}\kappa_{\rm
p}\hat{x}\cdot
y}}{\partial\nu(y)}-\frac{\partial\phi}{\partial\nu}(y)\mathrm{e}^{-\mathrm{i}
\kappa_{\rm p}\hat{x}\cdot y}\right\}\mathrm{d}s(y), 
	\end{align*}
	where $\gamma_{\kappa_{\rm p}}=
{\mathrm{e}^{\mathrm{i}\pi/4}}/{\sqrt{8\kappa_{\rm p}\pi}}$. By the
Helmholtz decomposition, the compressional wave can be represented by 
	\begin{equation}\label{gradient_phi}
	\boldsymbol{v}_{\rm p}(x)=
\nabla_x\phi(x)=\int_{\Gamma_D}\left\{\phi(y)\frac{\partial}{\partial\nu(y)}
\big(\nabla_x\Phi(x,y;\kappa_{\rm p})\big)-\frac{\partial\phi}{\partial\nu}(y)
\nabla_x\Phi(x,y;\kappa_{\rm p})\right\}\mathrm{d}s(y).
	\end{equation}
	Using straight forward calculations and noting
$H_1^{(1)}=-{H_0^{(1)}}'$, we obtain 
	$$
	\nabla_x\Phi(x,y;\kappa_{\rm p})=\frac{-\mathrm{i}\kappa_{\rm
p}}{4}H_1^{(1)}(\kappa_{\rm p}|x-y|)\frac{x-y}{|x-y|}.
	$$
	With the help of the asymptotic behavior of the Hankel
functions \cite[eqn. (3.82)]{DR-book2013}
	$$
	H_n^{(1)}(t)=\sqrt{\frac{2}{\pi
t}}\rm{exp}\bigg\{\mathrm{i}(t-\frac{n\pi}{2} -\frac{\pi}{4})\bigg\}
\left\{1+\mathcal{O}\left(\frac{1}{t}\right)\right\}
	$$
	and
	$$
	|x-y|=|x|-\hat{x}\cdot y+\mathcal{O}\left(\frac{1}{|x|}\right),
	$$
	we derive
	$$
	\nabla_x\Phi(x,y;\kappa_{\rm p})=\frac{\mathrm{e}^{\mathrm{i}\kappa_{\rm
p}|x|}}{\sqrt{|x|}}
	\left\{\mathrm{i}\gamma_{\kappa_{\rm p}}\kappa_{\rm
p}\mathrm{e}^{-\mathrm{i} \kappa_{\rm p}\hat{x}\cdot
y}\hat{x}+\mathcal{O}\left(\frac{1}{|x|}\right)\right\}
	$$
	and
	$$
	\frac{\partial}{\partial\nu(y)}\nabla_x\Phi(x,y;\kappa_{\rm
p})=\frac{\mathrm{e}^{\mathrm{i}\kappa_{\rm p}|x|}}{\sqrt{|x|}}
	\left\{\mathrm{i}\gamma_{\kappa_{\rm p}}\kappa_{\rm
p}\frac{\partial\mathrm{e}^{-\mathrm{i}\kappa_{\rm p}\hat{x}\cdot
y}}{\partial\nu(y)}\hat{x}+\mathcal{O}\left(\frac{1}{|x|}\right)\right\},
	$$
	Substituting the last two equations into \eqref{gradient_phi} yields 
	$$
	\boldsymbol{v}_{\rm p}=\frac{\mathrm{e}^{\mathrm{i}\kappa_{\rm
p}|x|}}{\sqrt{|x|}}
	\left\{\mathrm{i}\kappa_{\rm p}\phi_\infty\hat{x}
+\mathcal{O}\left(\frac{1}{|x|}\right)\right\}.
	$$
	
	Similarly, by noting $\boldsymbol{v}_{\rm s}=\boldsymbol{\rm
curl}_x\psi$ and
	$$
	\boldsymbol{\rm curl}_x\Phi(x,y;\kappa_{\rm
s})=\frac{\mathrm{i}\kappa_{\rm s}}{4}H_1^{(1)}(\kappa_{\rm
s}|x-y|)\frac{(x-y)^\perp}{|x-y|},
	$$ 
	we can obtain that
	$$
	\boldsymbol{v}_{\rm s}=\frac{\mathrm{e}^{\mathrm{i}\kappa_{\rm
s}|x|}}{\sqrt{|x|}}
	\left\{-\mathrm{i}\kappa_{\rm s}\psi_\infty\hat{x}^\perp
+\mathcal{O}\left(\frac{1}{|x|}\right)\right\},
	$$
	which completes the proof.
\end{proof}

In view of \eqref{behaviour relation}, we see that if $(a_{\rm
p}\boldsymbol{v}_{\rm
p}^\infty, a_{\rm s}\boldsymbol{v}_{\rm s} ^\infty)$ or $(a_{\rm
p}|\boldsymbol{v}_{\rm p}^\infty|, a_{\rm s}|\boldsymbol{v}_{\rm s} ^\infty|)$
is
known, then  the information of far-field pattern $(a_{\rm p}\phi_\infty,
a_{\rm s}\psi_\infty)$ or $(a_{\rm p}|\phi_\infty|, a_{\rm s}|\psi_\infty|)$
can be obtained. Hence, we may reconstruct the obstacle from the knowledge of
$(a_{\rm p}\phi_\infty, a_{\rm s}\psi_\infty)$ and $(a_{\rm p}|\phi_\infty|,
a_{\rm s}|\psi_\infty|)$ in Problem \ref{problem_1} and Problem \ref{problem_2},
respectively.

The following result show the translation invariance property of the phaseless
compressional and shear far-field patterns.

\begin{theorem}\label{Th2}
Assume that $\phi_\infty, \psi_\infty$ are the far-field patterns of the
scattered waves $\phi, \psi$ with incident plane wave $\boldsymbol{u}^{\rm
inc}(x)=a^{\rm inc}_{\rm p}d\mathrm{e}^{{\rm i} \kappa_{\rm p}d\cdot x}+a^{\rm
inc}_{\rm s}d^\perp \mathrm{e}^{{\rm i} \kappa_{\rm s}d\cdot x}$, where $(a^{\rm
inc}_{\rm p}, a^{\rm inc}_{\rm s})=(1,0)$ for the compressional incident plane
wave and $(a^{\rm inc}_{\rm p}, a^{\rm inc}_{\rm s})=(0,1)$ for the shear
incident plane wave. For the shifted domain $D_h:=\{x+h: x\in D\}$ with a fixed
vector $h\in\mathbb{R}^2$, the far-field patterns $\phi_\infty^h, \psi_\infty^h$
satisfy the relations
\begin{align}\label{invarance1}
\phi_\infty^h(\hat{x})=\mathrm{e}^{\mathrm{i}\kappa_{\rm p}
(d-\hat{x})\cdot h} \phi_\infty(\hat{x}), \quad
\psi_\infty^h(\hat{x})=\mathrm{e}^{\mathrm{i}(\kappa_{\rm
p}d-\kappa_{\rm s} \hat{x})\cdot h}\psi_\infty(\hat{x}),\quad (a^{\rm inc}_{\rm
p}, a^{\rm inc}_{\rm s})=(1,0)
\end{align}
and
 \begin{align}\label{invarance2}
\psi_\infty^h(\hat{x})=\mathrm{e}^{\mathrm{i}\kappa_{\rm s}
(d-\hat{x})\cdot h}\psi_\infty(\hat{x}), \quad
\phi_\infty^h(\hat{x})=\mathrm{e}^{\mathrm{i}(\kappa_{\rm
s}d-\kappa_{\rm p}\hat{x})\cdot h}\phi_\infty(\hat{x}),\quad (a^{\rm inc}_{\rm
p}, a^{\rm inc}_{\rm s})=(0,1).
\end{align}
\end{theorem}

\begin{proof}
We only give the proof for the compressional incident plane wave case,
i.e., \eqref{invarance1} for $(a^{\rm inc}_{\rm p}, a^{\rm inc}_{\rm s})=(1,0)$,
since the other case \eqref{invarance2} for $(a^{\rm inc}_{\rm p}, a^{\rm
inc}_{\rm s})=(0,1)$ can be proved similarly. 

We assume that the solution of \eqref{HelmholtzDec} is given as 
single-layer potentials with densities $g_1, g_2$:
\begin{align}
\phi(x)=\int_{\Gamma_D}\Phi(x,y;\kappa_{\rm p})g_1(y)\mathrm{d}s(y), \quad 
\psi(x)=\int_{\Gamma_D}\Phi(x,y;\kappa_{\rm s})g_2(y)\mathrm{d}s(y), \quad 
x\in\mathbb{R}^2\setminus\Gamma_D. \label{singlelayer}
\end{align}
Letting $x\in\mathbb{R}^2\setminus\overline{D}$ approach the boundary $\Gamma_D$
in \eqref{singlelayer}, and using the jump relation of single-layer potentials
and the boundary condition of \eqref{HelmholtzDec}, we deduce for
$x\in\Gamma_D$ that 
\begin{align}\label{boundaryIE_phi}
-\frac{1}{2}g_1(x)+\int_{\Gamma_D}\frac{\partial\Phi(x,y;\kappa_{\rm p})}
{\partial\nu(x)}g_1(y)\mathrm{d}s(y)+\int_{\Gamma_D}
\frac{\partial\Phi(x,y;\kappa_{\rm s})}
{\partial\tau(x)}g_2(y)\mathrm{d}s(y)=-\nu(x)\cdot\boldsymbol{u}^{\rm inc}(x)
\end{align}
and
\begin{align}\label{boundaryIE_psi}
\int_{\Gamma_D}\frac{\partial\Phi(x,y;\kappa_{\rm p})}
{\partial\tau(x)}g_1(y)\mathrm{d}s(y)+ \frac{1}{2}g_2(x)-\int_{\Gamma_D}
\frac{\partial\Phi(x,y;\kappa_{\rm s})}
{\partial\nu(x)}g_2(y)\mathrm{d}s(y)=-\tau(x)\cdot\boldsymbol{u}^{\rm inc}(x).
\end{align}
The corresponding far-field patterns can be represented as follows 
\begin{align} \label{singlelayer_far}
\phi_\infty(\hat{x})=\gamma_{\kappa_{\rm
p}}\int_{\Gamma_D}\mathrm{e}^{-\mathrm{i}\kappa_{\rm p} \hat{x}\cdot y}
g_1(y)\mathrm{d}s(y), \quad
\psi_\infty(\hat{x})=\gamma_{\kappa_{\rm s}}\int_{\Gamma_D}
\mathrm{e}^{-\mathrm{i}\kappa_{\rm s} \hat{x}\cdot y} g_2(y)\mathrm{d}s(y),
\quad\hat{x}\in\Omega. 
\end{align}

Furthermore, we assume that the densities $g_1^h, g_2^h$ solve the boundary
integral equations \eqref{boundaryIE_phi}--\eqref{boundaryIE_psi} with
$\Gamma_D$ replaced by $\Gamma_{D_h}$. We now show that if $g_1, g_2$
solve \eqref{boundaryIE_phi}--\eqref{boundaryIE_psi}, then
\begin{equation}  \label{density_invar}
g_1^h(x)=\mathrm{e}^{\mathrm{i}\kappa_{\rm p}d\cdot h}g_1(x-h), \qquad
g_2^h(x)=\mathrm{e}^{\mathrm{i}\kappa_{\rm p}d\cdot h}g_2(x-h)
\end{equation} 
also solve the boundary integral equations
\eqref{boundaryIE_phi}--\eqref{boundaryIE_psi} with $\Gamma_D$ replaced by
$\Gamma_{D_h}$. In fact, substituting above equations into the left side of 
\eqref{boundaryIE_phi}--\eqref{boundaryIE_psi} with $\Gamma_D$ replaced by
$\Gamma_{D_h}$ and setting $\tilde{x}=x-h$, $\tilde{y}=y-h$, we get for
$x\in\partial D_h$ that 
\begin{align*}
&-\frac{1}{2}g_1^h(x)+\int_{\Gamma_{D_h}}\frac{\partial\Phi(x,y;\kappa_{\rm p})}
{\partial\nu(x)}g_1^h(y)\mathrm{d}s(y)+\int_{\Gamma_{D_h}}\frac{\partial\Phi(x,
y;\kappa_{\rm s})} {\partial\tau(x)}g_2^h(y)\mathrm{d}s(y) \\
=&\mathrm{e}^{\mathrm{i}\kappa_{\rm p}d\cdot h}
\left(-\frac{1}{2}g_1^h(\tilde{x})+\int_{\Gamma_D}
\frac{\partial\Phi(\tilde{x},\tilde{y};\kappa_{\rm p})}
{\partial\nu(\tilde{x})}g_1^h(\tilde{y})\mathrm{d}s(\tilde{y})+\int_{\Gamma_D}
\frac{\partial\Phi(\tilde{x},\tilde{y};\kappa_{\rm s})}
{\partial\tau(\tilde{x})}g_2^h(\tilde{y})\mathrm{d}s(\tilde{y})\right) \\
=&\mathrm{e}^{\mathrm{i}\kappa_{\rm p}d\cdot h}
\big(-\nu(\tilde{x})\cdot\boldsymbol{u}^{\rm inc}(\tilde{x})\big)
=-\nu(x)\cdot\boldsymbol{u}^{\rm inc}(x).
\end{align*}
Similarly, \eqref{boundaryIE_psi} can be handled in
the same way. Thus, the relation \eqref{density_invar} follows from the fact
that the system of boundary integral
equations \eqref{boundaryIE_phi}--\eqref{boundaryIE_psi} for $D_h$ has a unique
solution \cite{LaiLi}.

Combining \eqref{singlelayer_far} and \eqref{density_invar}, we obtain 
\begin{align*}
\phi_\infty^h(\hat{x})&=\gamma_{\kappa_{\rm
p}}\int_{\Gamma_{D_h}}\mathrm{e}^{-\mathrm{i}\kappa_{\rm p} \hat{x}\cdot
y}g_1^h(y)\mathrm{d}s(y) \\
&=\gamma_{\kappa_{\rm p}}\int_{\Gamma_{D_h}} \mathrm{e}^{-\mathrm{i}\kappa_{\rm
p} \hat{x}\cdot(y-h)} \mathrm{e}^{-\mathrm{i}\kappa_{\rm p}\hat{x}\cdot
h}\mathrm{e}^{\mathrm{i}\kappa_{\rm p} d\cdot h} g_1(y-h)\mathrm{d}s(y) 
=\mathrm{e}^{\mathrm{i}\kappa_{\rm p} (d-\hat{x})\cdot h}\phi_\infty(\hat{x})
\end{align*}
and
\begin{align*}
\psi_\infty^h(\hat{x})&=\gamma_{\kappa_{\rm s}} \int_{\Gamma_{D_h}}
\mathrm{e}^{-\mathrm{i}\kappa_{\rm s} \hat{x}\cdot y}g_2^h(y)\mathrm{d}s(y) \\
&=\gamma_{\kappa_{\rm s}} \int_{\Gamma_{D_h}}\mathrm{e}^{-\mathrm{i}\kappa_{\rm
s} \hat{x}\cdot(y-h)} \mathrm{e}^{-\mathrm{i}\kappa_{\rm s}
\hat{x}\cdot h}\mathrm{e}^{\mathrm{i}\kappa_{\rm p} d\cdot h}
g_2(y-h)\mathrm{d}s(y)
=\mathrm{e}^{\mathrm{i}(\kappa_{\rm p}d-\kappa_{\rm s} \hat{x})\cdot h}
\psi_\infty(\hat{x}),
\end{align*}
which completes the proof. 
\end{proof}

Theorem \ref{Th2} implies that the compressional and shear far-field patterns
are invariant under translations of the obstacle $D$ for the compressional or
shear plane incident wave.

\section{Nystr\"{o}m-type discretization for boundary integral equations}

In this section, we present a Nystr\"{o}m-type discretization to solve
the coupled system \eqref{boundaryIE_phi}--\eqref{boundaryIE_psi} . We first
introduce
the single-layer integral operator
$$
(S_\kappa g)(x)=2\int_{\Gamma_D}\Phi(x,y;\kappa)g(y)\mathrm{d}s(y), \quad
x\in\Gamma_D,
$$
and the corresponding far-field integral operator
$$
(S_\kappa^\infty
g)(\hat{x})=\gamma_\kappa\int_{\Gamma_D}\mathrm{e}^{-\mathrm{i}\kappa
\hat{x}\cdot y}g(y)\mathrm{d}s(y), \quad\hat{x}\in\Omega.
$$
In addition, we need to introduce the normal derivative boundary integral
operator
$$
(K_\kappa g)(x)=2\int_{\Gamma_D}\frac{\partial\Phi(x,y;\kappa)}
{\partial\nu(x)}g(y)\mathrm{d}s(y), \qquad x\in\Gamma_D,
$$
and the tangential derivative boundary integral operator
$$
(H_\kappa g)(x)=2\int_{\Gamma_D}\frac{\partial\Phi(x,y;\kappa)}
{\partial\tau(x)}g(y)\mathrm{d}s(y), \qquad x\in\Gamma_D.
$$
Then, the coupled boundary integral equations
\eqref{boundaryIE_phi}--\eqref{boundaryIE_psi} can be rewritten in the operator
form
\begin{align}
-g_1+K_{\kappa_{\rm p}}g_1+H_{\kappa_{\rm s}} g_2=2f_1,\label{direct field
eqn1}\\ 
H_{\kappa_{\rm p}}g_1+g_2-K_{\kappa_{\rm s}}g_2=2f_2. \label{direct field
eqn2}
\end{align}
The corresponding far-field patterns of \eqref{singlelayer_far} can be
represented as follows 
\begin{align*}
\phi_\infty(\hat{x})=(S_{\kappa_{\rm p}}^\infty g_1)(\hat{x}), \quad
\psi_\infty(\hat{x})=(S_{\kappa_{\rm s}}^\infty g_2)(\hat{x}),
\qquad\hat{x}\in\Omega. 
\end{align*}

\subsection{Parametrization} 

For simplicity, the boundary $\Gamma_D$ is assumed to be a star-shaped curve
with the parametric form
\begin{equation*}
\Gamma_D=\{p(\hat{x})=c+r(\hat{x})\hat{x}: ~c=(c_1,c_2)^\top,\
\hat{x}\in\Omega\},
\end{equation*} 
where $\Omega=\{\hat{x}(t)=(\cos t, \sin t)^\top: ~0\leq t< 2\pi\}$.
We introduce the parameterized integral operators which are still denoted by
$S_\kappa$, $S_\kappa^\infty$, $K_\kappa$, and $H_\kappa$ for convenience, i.e.,
\begin{align*}
 \big(S_\kappa(p,\varphi_j)\big)(t)&=\frac{\mathrm{i}}{2}\int_0^{2\pi}
H_0^{(1)}(\kappa|p(t)-p(\varsigma)|)\varphi_j(\varsigma)\mathrm{d}\varsigma,\\
\big(S^\infty_\kappa(p,\varphi_j)\big)(t)&=\gamma_{\kappa}\int_0^{2\pi}\mathrm{e
} ^{-\mathrm{i}\kappa\hat{x}(t)\cdot
p(\varsigma)}\varphi_j(\varsigma)\mathrm{d}\varsigma,\\
\big(K_\kappa(p,\varphi_j)\big)(t)&=\frac{1}{G(t)}\int_0^{2\pi}\widetilde
K(t,\varsigma;\kappa)\varphi_j(\varsigma)\mathrm{d}\varsigma,\\
\big(H_\kappa(p,\varphi_j)\big)(t)&=\frac{1}{G(t)}\int_0^{2\pi}  \widetilde
H(t,\varsigma;\kappa)\varphi_j(\varsigma)\mathrm{d}\varsigma,
\end{align*}
where $\varphi_j(\varsigma)=G(\varsigma)g_j(p(\varsigma))$, $j=1,
2$, $G(\varsigma):=|p'(\varsigma)|=\sqrt{(r'(\varsigma))^2+r^2(\varsigma)}$
is the Jacobian of the transformation,
\begin{align*}
\widetilde K(t,\varsigma;\kappa)&=\frac{\mathrm{i}\kappa}{2}n(t)\cdot[
p(\varsigma)-p(t)] \frac{H_1^{(1)}(\kappa|p(t)-p(\varsigma)|)}{
|p(t)-p(\varsigma)|},\\
\widetilde H(t,\varsigma;\kappa)&=\frac{\mathrm{i}\kappa}{2}n(t)^\perp\cdot[
p(\varsigma)-p(t)] \frac{H_1^{(1)}(\kappa|p(t)-p(\varsigma)|)}{
|p(t)-p(\varsigma)|},
\end{align*}
and 
\begin{align*}
 n(t)=\nu(x(t))|p'(t)|=\big(p'_2(t), -p'_1(t)\big),\quad
 n(t)^\perp=\tau(x(t))|p'(t)|=\big(p'_1(t), p'_2(t)\big).
\end{align*}
Hence, equations \eqref{direct field eqn1}--\eqref{direct field eqn2}
can be reformulated as the parametrized integral equations
\begin{align}
-\varphi_1 + K_{\kappa_{\rm p}}(p,\varphi_1)G + H_{\kappa_{\rm
s}}(p,\varphi_2)G&=w_1, \label{direct parafield eqn1}\\
\varphi_2+H_{\kappa_{\rm p}}(p,\varphi_1)G
-K_{\kappa_{\rm s}}(p,\varphi_2)G&=w_2, \label{direct parafield eqn2}
\end{align}
where $w_j=2(f_j\circ p)G$, $j=1, 2$.

\subsection{Discretization}

We adopt the Nystr\"{o}m method for the discretization of the boundary
integrals. We refer to \cite{Kress-JCAM1995} for an application of the
Nystr\"{o}m method to solve the acoustic wave scattering problem by using a
hypersingular integral equation. 

The kernel $\widetilde{K}$ of the parameterized normal derivative integral
operator can be written in the form of
$$
\widetilde{K}(t,\varsigma;\kappa)=\widetilde{K}_1(t,
\varsigma;\kappa)\ln\Big(4\sin^2\frac{t-\varsigma}{2}\Big)+\widetilde{K}_2(t,
\varsigma;\kappa),
$$
where
\begin{align*}
\widetilde K_1(t,\varsigma;\kappa)&=
\frac{\kappa}{2\pi}n(t)\cdot\big[p(t)-p(\varsigma)\big]\frac{
J_1(\kappa|p(t)-p(\varsigma)|)}{|p(t)-p(\varsigma)|}, \\
\widetilde K_2(t,\varsigma;\kappa)&=\widetilde K(t,\varsigma;\kappa)-\widetilde
K_1(t,\varsigma;\kappa)\ln\Big(4\sin^2\frac{t-\varsigma}{2}\Big).
\end{align*}
It can be shown that the diagonal terms are
$$
\widetilde K_1(t,t;\kappa)=0, \qquad \widetilde K_2(t,t;\kappa)=
\frac{1}{2\pi}\frac{n(t)\cdot p''(t)}{|p'(t)|^2}.
$$

Noting $H_1^{(1)}=J_1+\mathrm{i}Y_1$, where $J_1$ and
$Y_1$ are the Bessel and Neumann functions of order one, and using the power
series \cite[eqns. (3.74) and (3.75)]{DR-book2013}
\begin{align*}
J_1(z)&:=\sum_{k=0}^{\infty}\frac{(-1)^k}{k!(k+1)!}\Big(\frac{z}{2}\Big)^{2k+1}
\\
Y_1(z)&:=\frac{2}{\pi}\Big\{\ln\frac{z}{2}+C\Big\}J_1(z)-\frac{2}{\pi}\frac{1}{z
}-\frac{1}{\pi}\sum_{k=0}^{\infty}\frac{(-1)^k}{k!(k+1)!}
\Big\{\psi(k+1)+\psi(k)\Big\}\Big(\frac{z}{2}\Big)^{2k+1}
\end{align*}
where $\psi(k)=\sum_{m=1}^k\frac{1}{m}$ with definition $\psi(0)=0$
and Euler's constant $C=0.57721\cdots$, we can split the kernel $\widetilde{H}$
of the parameterized tangential derivative integral operator into 
$$
\widetilde{H}(t,\varsigma;\kappa)=\widetilde{H}_1(t,\varsigma;\kappa)\frac{1}{
\sin(\varsigma-t)}+\widetilde{H}_2(t,\varsigma;\kappa)\ln\Big(4\sin^2\frac{
t-\varsigma}{2}\Big)+\widetilde{H}_3(t,\varsigma;\kappa),
$$
where the functions
\begin{align*}
\widetilde H_1(t,\varsigma;\kappa)&=
\frac{1}{\pi}n(t)^\perp\cdot\big[p(\varsigma)-p(t)\big]\frac{\sin(\varsigma-t)}{
|p(t)-p(\varsigma)|^2}, \\
\widetilde H_2(t,\varsigma;\kappa)&=
\frac{\kappa}{2\pi}n(t)^\perp\cdot\big[p(t)-p(\varsigma)\big]\frac{
J_1(\kappa|p(t)-p(\varsigma)|)}{|p(t)-p(\varsigma)|},\\
\widetilde H_3(t,\varsigma;\kappa)&=\widetilde H(t,\varsigma;\kappa)-\widetilde
H_1(t,\varsigma;\kappa)\frac{1}{\sin(\varsigma-t)}-\widetilde
H_2(t,\varsigma;\kappa)\ln\Big(4\sin^2\frac{t-\varsigma}{2}\Big)
\end{align*}
are analytic with diagonal entries given by 
$$
\widetilde H_1(t,t;\kappa)=\frac{1}{\pi}, \quad \widetilde H_2(t,t;\kappa)=0,
\quad \widetilde H_3(t,t;\kappa)=0.
$$

Let $\varsigma_j^{(n)}:=\pi j/n$, $j=0,\cdots,2n-1$ be an equidistant set of
quadrature points. For the singular integrals, we employ the following
quadrature rules via the trigonometric interpolation
\begin{align}
&\int_{0}^{2\pi}\ln\Big(4\sin^2\frac{t-\varsigma}{2}\Big)f(\varsigma)\mathrm{d
} \varsigma\approx\sum_{j=0}^{2n-1}R_j^{(n)}(t)f(\varsigma_j^{(n)}),
\label{quadrature1} \\
&\int_{0}^{2\pi}\frac{f(\varsigma)}{\sin(\varsigma-t)}\mathrm{d}
\varsigma\approx\sum_{j=0}^{2n-1}T_j^{(n)}(t)f(\varsigma_j^{(n)}),
\label{quadrature2}
\end{align}
where the quadrature weights are given by
\begin{align}\label{weight_R}
R_j^{(n)}(t)=-\frac{2\pi}{n}\sum_{m=1}^{n-1}\frac{1}{m}\cos\Big[
m(t-\varsigma_j^{(n)})\Big]
-\frac{\pi}{n^2}\cos\Big[n(t-\varsigma_j^{(n)})\Big] 
\end{align}
and
\begin{align*}
T_j^{(n)}(t)=
\begin{cases} \displaystyle 
-\frac{2\pi}{n}\sum_{m=0}^{(n-3)/2}\sin\Big[ (2m+1)(t-\varsigma_j^{(n)})\Big]
-\frac{\pi}{n}\sin\Big[n(t-\varsigma_j^{(n)})\Big], ~ &n=1,3,5,\cdots, \\
\displaystyle
-\frac{2\pi}{n}\sum_{m=0}^{n/2-1}\sin\Big[ (2m+1)(t-\varsigma_j^{(n)})\Big],
~ &n=2,4,6,\cdots.
\end{cases}
\end{align*}
Using the Lagrange basis for the trigonometric interpolation \cite[eqn.
(11.12)]{Kress-book2014}), we derive the weight $T_j^{(n)}$ 
by calculating the integrals
\begin{align*}
&\int_{0}^{2\pi}\frac{\cos k\varsigma}{\sin \varsigma}\mathrm{d}\varsigma=0,
\quad k=0,1,2,\cdots, \quad
&\int_{0}^{2\pi}\frac{\sin k\varsigma}{\sin \varsigma}\mathrm{d}\varsigma=
\begin{cases}
2\pi, ~ &k=1,3,5\cdots, \\
0,    ~ &k=2,4,6,\cdots,
\end{cases}
\end{align*} 
in the sense of Cauchy principal value. The details of \eqref{weight_R}
can be found in \cite{Kress-book2014}. For the smooth integrals, we simply use
the trapezoidal rule 
\begin{align} \label{traperule}
\int_{0}^{2\pi}f(\varsigma)\mathrm{d}\varsigma\approx\frac{\pi}{n}\sum_{j=0}^{
2n-1}f(\varsigma_j^{(n)}).
\end{align}

By \eqref{quadrature1} and \eqref{quadrature2},
the full discretization of \eqref{direct parafield eqn1}--\eqref{direct
parafield eqn2} can be deduced as
\begin{align*}
w_{1,i}^{(n)}= & -\varphi^{(n)}_{1,i}+
\sum_{j=0}^{2n-1}\left(R_{|i-j|}^{(n)}\widetilde
K_1(\varsigma_i^{(n)},\varsigma_j^{(n)};\kappa_{\rm p})+ \frac{\pi}{n}\widetilde
K_2(\varsigma_i^{(n)},\varsigma_j^{(n)};\kappa_{\rm
p})\right)\varphi_{1,j}^{(n)} \\
&+\sum_{j=0}^{2n-1}\left(-T_{i-j}^{(n)}\widetilde
H_1(\varsigma_i^{(n)},\varsigma_j^{(n)};\kappa_{\rm s})+
R_{|i-j|}^{(n)}\widetilde H_2(\varsigma_i^{(n)},\varsigma_j^{(n)};\kappa_{\rm
s})+ \frac{\pi}{n}\widetilde H_3(\varsigma_i^{(n)},\varsigma_j^{(n)};\kappa_{\rm
s})\right)\varphi_{2,j}^{(n)},
\end{align*}
\begin{align*}
w_{2,i}^{(n)}= &\varphi^{(n)}_{2,i}+  \sum_{j=0}^{2n-1}\left(-T_{i-j}^{(n)}
\widetilde H_1(\varsigma_i^{(n)}, \varsigma_j^{(n)};\kappa_{\rm
p})+R_{|i-j|}^{(n)}H_2(\varsigma_i^{(n)},\varsigma_j^{(n)};\kappa_{\rm p})+
\frac{\pi}{n}\widetilde H_3(\varsigma_i^{(n)},\varsigma_j^{(n)};\kappa_{\rm
p})\right)\varphi_{1,j}^{(n)} \\
&- \sum_{j=0}^{2n-1} \left(R_{|i-j|}^{(n)}\widetilde
K_1(\varsigma_i^{(n)},\varsigma_j^{(n)};\kappa_{\rm s})+\frac{\pi}{n}\widetilde
K_2(\varsigma_i^{(n)},\varsigma_j^{(n)};\kappa_{\rm s})\right)
\varphi_{2,j}^{(n)},
\end{align*}
where $w_{l,i}^{(n)}=w_l(\varsigma_i^{(n)})$,
$\varphi_{l,j}^{(n)}=\varphi_l(\varsigma_j^{(n)})$ for $i,j=0,\cdots,2n-1$,
$l=1,2$, and
\begin{align*}
R_j^{(n)}&:=R_j^{(n)}(0)=-\frac{2\pi}{n}\sum_{m=1}^{n-1}\frac{1}{m}\cos\frac{
mj\pi}{n}-\frac{(-1)^j\pi}{n^2}, \\
T_j^{(n)}&:=T_j^{(n)}(0)=
\frac{2\pi}{n}\sum_{m=0}^{\tilde{n}}\sin \frac{(2m+1)j\pi}{n}, \qquad
\tilde{n}=
\begin{cases}
(n-3)/2, &n=1,3,5,\cdots, \\
n/2-1,     &n=2,4,6,\cdots.
\end{cases}
\end{align*}

\section{Reconstruction methods}

In this section, we introduce a system of nonlinear equations and develop
corresponding reconstruction methods for Problem 1 and Problem 2,
respectively.

\subsection{The phased IOSP}

On $\Gamma_D$, it follows from the boundary integral equations
\eqref{direct field eqn1}--\eqref{direct field eqn2} that the field equations
are
\begin{align}
-g_1+K_{\kappa_{\rm p}}g_1+H_{\kappa_{\rm s}} g_2=2f_1, \label{field eqn1}\\ 
H_{\kappa_{\rm p}}g_1+g_2-K_{\kappa_{\rm s}}g_2=2f_2. \label{field eqn2}
\end{align}
The data equation is given by
\begin{align} \label{data eqn}
a_{\rm p}S^\infty_{\kappa_{\rm p}}g_1+a_{\rm s}S^\infty_{\kappa_{\rm
s}}g_2=a_{\rm p}\phi_\infty+a_{\rm s}\psi_\infty.
\end{align}
The field equations and data equation \eqref{field eqn1}--\eqref{data eqn} can
be reformulated as the parametrized integral equations
\begin{align}
-\varphi_1 + K_{\kappa_{\rm p}}(p,\varphi_1)G + H_{\kappa_{\rm
s}}(p,\varphi_2)G&=w_1, \label{parafield eqn1}\\
\varphi_2+H_{\kappa_{\rm p}}(p,\varphi_1)G
-K_{\kappa_{\rm s}}(p,\varphi_2)G&=w_2, \label{parafield eqn2}\\
a_{\rm p}S^\infty_{\kappa_{\rm p}}(p,\varphi_1)+a_{\rm s} S^\infty_{\kappa_{\rm
s}}(p,\varphi_2)=a_{\rm p}\phi_\infty&+
a_{\rm s}\psi_\infty, \label{paradata eqn}
\end{align}
where $w_j=2(f_j\circ p)G$, $j=1, 2$.

In the reconstruction process, when an approximation of the boundary $\Gamma_D$
is available, the field equations \eqref{parafield eqn1}--\eqref{parafield eqn2}
are solved for the densities $g_1$ and $g_2$. Once the approximated densities
$g_1$ and $g_2$ are computed, the update of the boundary $\Gamma_D$ can be
obtained by solving the linearized data equation \eqref{paradata eqn} with
respect to $\Gamma_D$.

\subsubsection{Iterative scheme}

The linearization of \eqref{paradata eqn} with respect to $p$ requires the
Fr\'{e}chet derivative of the parameterized integral operator
$S_\kappa^\infty$, which can be explicitly calculated as follows
\begin{align} 
\Big({S_\kappa^\infty}'[p;\varphi]q\Big)(t)=&-\mathrm{i}\kappa\gamma_\kappa
\int_0^{2\pi}\mathrm{e}^{-\mathrm{i}\kappa\hat{x}(t)\cdot
p(\varsigma)}\hat{x}(t)\cdot q(\varsigma)\varphi(\varsigma)\mathrm{d}\varsigma
\nonumber \\
=&-\mathrm{i}\kappa\gamma_\kappa\int_0^{2\pi}\exp\Big(-\mathrm{i}
\kappa\big(c_1\cos t+c_2\sin t+r(\varsigma)\cos(t-\varsigma)\big)\Big)
\nonumber \\
&\qquad\quad\times\Big(\Delta c_1\cos t+\Delta c_2\sin t +\Delta r(\varsigma)
\cos(t-\varsigma)\Big)\varphi(\varsigma)\,\mathrm{d}\varsigma,
\label{FreSinfty}
\end{align}
where 
$$q(\varsigma)=(\Delta c_1, \Delta c_2)+\Delta r(\varsigma)
(\cos\varsigma,\sin\varsigma)$$ 
gives the update of the boundary
$\Gamma_D$. Then, the linearization of \eqref{paradata eqn} leads to
\begin{equation}\label{linear paradata eqn}
a_{\rm p}{S_{\kappa_{\rm p}}^\infty}'[p;\varphi_1]q+a_{\rm s} {S_{\kappa_{\rm
s}}^\infty}'[p;\varphi_2]q=w,
\end{equation}
where
\begin{align*}
w:= & a_{\rm p}\Big(\phi_\infty-S^\infty_{\kappa_{\rm
p}}(p,\varphi_1)\Big)+a_{\rm s} \Big(\psi_\infty-S^\infty_{\kappa_{\rm
s}}(p,\varphi_2)\Big).
\end{align*}

As usual, a stopping criteria is necessary to terminate the iteration. For our
iterative procedure, the following relative error estimator is used
\begin{equation}
{\rm E}_k:=\frac{\left\|a_{\rm p}\Big(\phi_\infty-S^\infty_{\kappa_{\rm p}}
(p^{(k)},\varphi_1)\Big)+a_{\rm s} \Big(\psi_\infty-S^\infty_{\kappa_{\rm s}}
(p^{(k)},\varphi_2)\Big)\right\|_{L^2}}{\big\|a_{\rm p}\phi_\infty+a_{\rm
s}\psi_\infty\big\|_{L^2}}\leq\epsilon, \label{relativeerror}
\end{equation}
where $\epsilon$ is a user-specified small positive constant depending on the
noise level and $p^{(k)}$ is the $k$th approximation of the boundary
$\Gamma_D$.

We are now in a position to present the iterative algorithm for the inverse
obstacle scattering problem with phased far-field data as {\bf Algorithm I}.

\begin{table}[ht]
\begin{tabular}{cp{.8\textwidth}}
\toprule
\multicolumn{2}{l}{{\bf Algorithm I:}\quad Iterative algorithm
for the phased IOSP} \\
\midrule
Step 1 & Send an incident plane wave $\boldsymbol{u}^{\rm
inc}$ with fixed $\omega, \lambda, \mu$, and a fixed incident direction
$d\in\Omega$, and then collect the corresponding far-field data
$\phi_\infty$ or $\psi_\infty$ for the scatterer $D$; \\
Step 2 & Select an initial star-like curve $\Gamma^{(0)}$
for the boundary $\Gamma_D$ and the error tolerance $\epsilon$. Set $k=0$; \\
Step 3 & For the curve $\Gamma^{(k)}$, compute the densities
$\varphi_1$ and $\varphi_2$ from \eqref{parafield eqn1}--\eqref{parafield eqn2};
\\
Step 4 & Solve \eqref{linear paradata eqn} to obtain the
updated approximation $\Gamma^{(k+1)}:=\Gamma^{(k)}+q$ and evaluate the error
$E_{k+1}$ defined in \eqref{relativeerror}; \\
Step 5 & If $E_{k+1}\geq\epsilon$, then set $k=k+1$ and go
to Step 3. Otherwise, the current approximation $\Gamma^{(k+1)}$ is taken to be
the final reconstruction of $\Gamma_D$. \\
\bottomrule
\end{tabular}
\end{table}

\subsubsection{Discretization} 

We use the Nystr\"{o}m method which is described in Section 4 for the full
discretizations of \eqref{parafield eqn1}--\eqref{parafield eqn2}. Now we
discuss the discretization of the linearized equation \eqref{linear paradata
eqn} and obtain the update by using least squares with Tikhonov regularization
\cite{Kress-IP2003}. As for a finite dimensional space to approximate the
radial function $r$ and its update $\Delta r$, we choose the space of
trigonometric polynomials of the form
\begin{equation}\label{updataq}
\Delta r(\tau)=\sum_{m=0}^M\alpha_m\cos{m\tau} +\sum_{m=1}^M\beta_m\sin{m\tau},
\end{equation}
where the integer $M>1$ denotes the truncation number. For simplicity, we
reformulate the equation \eqref{linear paradata eqn} by introducing the
following definitions
\begin{align*}
&L_1(t,\varsigma;\kappa,\varphi):=-\mathrm{i}\kappa\gamma_\kappa\exp\left\{
-\mathrm{i}\kappa\Big(c_1\cos t+c_2\sin
t+r(\varsigma)\cos(t-\varsigma)\Big)\right\}\cos t~ \varphi(\varsigma),\\
&L_2(t,\varsigma;\kappa,\varphi):=-\mathrm{i}\kappa\gamma_\kappa\exp\left\{
-\mathrm{i}\kappa\Big(c_1\cos t+c_2\sin
t+r(\varsigma)\cos(t-\varsigma)\Big)\right\}\sin t~ \varphi(\varsigma),\\
&L_{3,m}(t,\varsigma;\kappa,\varphi):=-\mathrm{i}\kappa\gamma_\kappa\exp\left\{
-\mathrm{i}\kappa\Big(c_1\cos t+c_2\sin t+r(\varsigma)
\cos(t-\varsigma)\Big)\right\}\cos(t-\varsigma)\cos
m\varsigma~\varphi(\varsigma),\\
&L_{4,m}(t,\varsigma;\kappa,\varphi):=-\mathrm{i}\kappa\gamma_\kappa\exp\left\{
-\mathrm{i}\kappa\Big(c_1\cos t+c_2\sin t+
r(\varsigma)\cos(t-\varsigma)\Big)\right\}\cos(t-\varsigma)\sin
m\varsigma~\varphi(\varsigma).
\end{align*}
Combining \eqref{FreSinfty}--\eqref{linear paradata eqn} and using the 
trapezoidal rule \eqref{traperule}, we get the discretized linear system
\begin{align}\label{discrelinear}
a_{\rm p}&\bigg(\sum_{l=1}^2 B_{l}^c(\varsigma_i^{(\bar{n})};\kappa_{\rm
p},\varphi_{1})\Delta
c_l+\sum_{m=0}^M\alpha_mB_{1,m}^r(\varsigma_i^{(\bar{n})};\kappa_{\rm
p},\varphi_{1})
+\sum_{m=1}^M\beta_mB_{2,m}^r(\varsigma_i^{(\bar{n})};\kappa_{\rm
p},\varphi_{1})\bigg) \nonumber\\
&+ a_{\rm s}\bigg(\sum_{l=1}^2 B_j^c(\varsigma_i^{(\bar{n})};\kappa_{\rm
s},\varphi_{2})\Delta c_l
+\sum_{m=0}^M\alpha_mB_{1,m}^r(\varsigma_i^{(\bar{n})};\kappa_{\rm 
s},\varphi_{2})
+\sum_{m=1}^M\beta_mB_{2,m}^r(\varsigma_i^{(\bar{n})};\kappa_{\rm
s},\varphi_{2})\bigg)\notag\\
&\hspace{2cm}= w(\varsigma_i^{(\bar{n})})
\end{align}
to determine the real coefficients $\Delta c_1$, $\Delta c_2$, $\alpha_m$, and
$\beta_m$, where
\begin{align*}
B_l^c(\varsigma_i^{(\bar{n})};\kappa,\varphi)=\frac{\pi}{n}\sum_{j=0}^{2n-1}
L_l(\varsigma_i^{(\bar{n})},\varsigma_j^{(n)};\kappa,\varphi),\quad l=1, 2, 
\end{align*}
and 
\begin{align*}
B_{1,m}^r(\varsigma_i^{(\bar{n})};\kappa,\varphi)&=\frac{\pi}{n}\sum_{j=0}^{2n-1
} L_{3,m}(\varsigma_i^{(\bar{n})},\varsigma_j^{(n)};\kappa,\varphi), \\
B_{2,m}^r(\varsigma_i^{(\bar{n})};\kappa,\varphi)&=\frac{\pi}{n}\sum_{j=0}^{2n-1
}L_{4,m}(\varsigma_i^{(\bar{n})},\varsigma_j^{(n)};\kappa,\varphi).
\end{align*}

In general, $2M+1\ll 2n$, and due to the ill-posedness, the overdetermined
system \eqref{discrelinear} is solved via the Tikhonov regularization. Hence the
linear system \eqref{discrelinear} is reformulated into minimizing the following
function
\begin{align}
&\sum_{i=0}^{2n-1}\Bigg|a_{\rm p}\bigg(\sum_{l=1}^2
B_{l}^c(\varsigma_i^{(\bar{n})};\kappa_{\rm p},\varphi_{1})\Delta
c_l+\sum_{m=0}^M\alpha_mB_{1,m}^r(\varsigma_i^{(\bar{n})};\kappa_{\rm
p},\varphi_{1})
+\sum_{m=1}^M\beta_mB_{2,m}^r(\varsigma_i^{(\bar{n})};\kappa_{\rm
p},\varphi_{1})\bigg) \nonumber \\
&\qquad + a_{\rm s}\bigg(\sum_{l=1}^2 B_j^c(\varsigma_i^{(\bar{n})};\kappa_{\rm
s},\varphi_{2})\Delta c_l
+\sum_{m=0}^M\alpha_mB_{1,m}^r(\varsigma_i^{(\bar{n})};\kappa_{\rm
s},\varphi_{2})
+\sum_{m=1}^M\beta_mB_{2,m}^r(\varsigma_i^{(\bar{n})};\kappa_{\rm
s},\varphi_{2})\bigg)\notag\\
&\qquad- w(\varsigma_i^{(\bar{n})})\Bigg|^2 
 +\lambda\bigg(|\Delta c_1|^2+|\Delta
c_2|^2+2\pi\Big[\alpha_0^2+\frac{1}{2}\sum_{m=1}
^M(1+m^2)^2(\alpha_m^2+\beta_m^2)\Big]\bigg), \label{RLHuygens3}
\end{align}
where $\lambda>0$ is a regularization parameter. It is easy to show that the
minimizer of \eqref{RLHuygens3} is the solution of the system
\begin{align}\label{EqualRLHuygens3}
\lambda
I\xi+\Re(\widetilde{B}^*\widetilde{B})\xi=\Re(\widetilde{B}^*\widetilde{w}),
\end{align}
where
\begin{align*}
\widetilde{B}=\Big(&a_{\rm p}B_1^c(:,\kappa_{\rm p}, \varphi_{1})+a_{\rm
s}B_1^c(:,\kappa_{\rm s}, \varphi_{2}), a_{\rm p}B_2^c(:,\kappa_{\rm p},
\varphi_{1})+a_{\rm s}B_2^c(:,\kappa_{\rm s}, \varphi_{2}), \\
&a_{\rm p}B_{1,0}^r(:,\kappa_{\rm p}, \varphi_{1})+a_{\rm
s}B_{1,0}^r(:,\kappa_{\rm s}, \varphi_{2}), \cdots, a_{\rm p}
B_{1,M}^r(:,\kappa_{\rm p}, \varphi_{1})+a_{\rm s}B_{1,M}^r(:,\kappa_{\rm s},
\varphi_{2}), \\
&a_{\rm p}B_{2,1}^r(:,\kappa_{\rm p}, \varphi_{1})+a_{\rm
s}B_{2,1}^r(:,\kappa_{\rm s}, \varphi_{2}),
\cdots,a_{\rm p}B_{2,M}^r(:,\kappa_{\rm p}, \varphi_{1})+a_{\rm
s}B_{2,M}^r(:,\kappa_{\rm s}, \varphi_{2})\Big)_{(2n)\times(2M+3)} 
\end{align*}
and
\begin{align*}
\noindent
&\xi=(\Delta c_1, \Delta c_2, \alpha_0,\cdots, \alpha_M,\beta_1,\cdots,
\beta_M)^\top, \\
&\widetilde{I}=\mathrm{diag}\{1, 1, 2\pi, \pi(1+1^2)^2, \cdots, \pi(1+M^2)^2,
\pi(1+1^2)^2, \cdots, \pi(1+M^2)^2\}, \\
&\widetilde{w}=(w(\tau_0^{(n)}),\cdots, w(\tau_{2n-1}^{(n)}))^\top.
\end{align*}
Thus, we obtain the new approximation 
$$
p^{new}(\hat{x})=(c+\Delta c)+\big(r(\hat{x})+\Delta r(\hat{x})\big)\hat{x}.
$$
\subsection{The phaseless IOSP}

To incorporate the reference ball, we find the solution of
\eqref{HelmholtzDec} with $D$ replaced by $D\cup B$ in the form of single-layer
potentials with densities $g_{1,\sigma}$ and $g_{2,\sigma}$:
\begin{align}
\phi(x)=\sum_\sigma\int_{\Gamma_\sigma}\Phi(x,y;\kappa_{\rm
p})g_{1,\sigma}(y)\mathrm{d}s(y), \quad 
\psi(x)=\sum_\sigma\int_{\Gamma_\sigma}\Phi(x,y;\kappa_{\rm
s})g_{2,\sigma}(y)\mathrm{d}s(y),  \label{singlelayerDB}
\end{align}
for $x\in\mathbb{R}^2\setminus\Gamma_{D\cup B}$, where $\sigma=D,
B$. We introduce integral operators
$$
(K_\kappa^{\sigma,\varrho}
g)(x)=2\int_{\Gamma_\sigma}\frac{\partial\Phi(x,y;\kappa)}
{\partial\nu(x)}g(y)\mathrm{d}s(y), \quad x\in\Gamma_\varrho,
$$
$$
(H_\kappa^{\sigma,\varrho}
g)(x)=2\int_{\Gamma_\sigma}\frac{\partial\Phi(x,y;\kappa)}
{\partial\tau(x)}g(y)\mathrm{d}s(y), \quad x\in\Gamma_\varrho,
$$
and the corresponding far-field pattern
$$
(S_{\kappa,\sigma}^\infty
g)(\hat{x})=\gamma_\kappa\int_{\Gamma_\sigma}\mathrm{e}^{-\mathrm{i}\kappa
\hat{x}\cdot y}g(y)\mathrm{d}s(y), \quad\hat{x}\in\Omega,
$$
where $\varrho=D, B$. Letting $x\in\mathbb{R}^2\setminus\overline{D\cup B}$
approach the boundary $\Gamma_D$ and $\Gamma_B$ respectively in
\eqref{singlelayerDB}, and using the jump relation of single-layer potentials
and the boundary condition of \eqref{HelmholtzDec} for $\Gamma_{D}\cup
\Gamma_{B}$, we deduce the field equations in the
operator form
\begin{align}
-g_{1,D}+\sum_\sigma K^{\sigma,D}_{\kappa_{\rm p}}g_{1,\sigma}+\sum_\sigma
H^{\sigma,D}_{\kappa_{\rm s}}g_{2,\sigma}=2f_1 \quad{\rm on}~\Gamma_D,
\label{phafield eqn1}\\ 
g_{2,D}+\sum_\sigma H^{\sigma,D}_{\kappa_{\rm p}}g_{1,\sigma}-\sum_\sigma
K^{\sigma,D}_{\kappa_{\rm s}}g_{2,\sigma}=2f_2 
\quad{\rm on}~\Gamma_D, \label{phafield eqn2}\\ 
-g_{1,B}+\sum_\sigma K^{\sigma,B}_{\kappa_{\rm p}}g_{1,\sigma}+\sum_\sigma
H^{\sigma,B}_{\kappa_{\rm s}}g_{2,\sigma}=2f_1
\quad{\rm on}~\Gamma_B, \label{phafield eqn3}\\ 
g_{2,B}+\sum_\sigma H^{\sigma,B}_{\kappa_{\rm p}}g_{1,\sigma}-\sum_\sigma
K^{\sigma,B}_{\kappa_{\rm s}}g_{2,\sigma}=2f_2 \quad{\rm on}~\Gamma_B.
\label{phafield eqn4}
\end{align}
The phaseless data equation can be written as 
\begin{align}\label{phadata eqn}
a_{\rm p}\bigg|\sum_\sigma S^\infty_{\kappa_{\rm
p},\sigma}g_{1,\sigma}\bigg|^2+a_{\rm s}\bigg|\sum_\sigma S^\infty_{\kappa_{\rm
s},\sigma}g_{2,\sigma}\bigg|^2 =a_{\rm p}|\phi_\infty|^2+a_{\rm
s}|\psi_\infty|^2.
\end{align}

In the reconstruction process, for a given approximated boundary
$\Gamma_D$, the field equations \eqref{phafield eqn1}--\eqref{phafield eqn4}
can be solved for the densities $g_{1,\sigma}$ and $g_{2,\sigma}$. Once
$g_{1,\sigma}$ and $g_{2,\sigma}$ are computed, the update of the boundary
$\Gamma_D$ can be obtained by linearizing \eqref{phadata eqn} with respect to
$\Gamma_D$.

\subsubsection{Parametrization and iterative scheme}

For simplicity, the boundary $\Gamma_D$ and $\Gamma_B$ are assumed to be
star-shaped curves with the parametric forms
\begin{align*}
&\Gamma_D=\{p_D(\hat{x})=c+r(\hat{x})\hat{x}: ~c=(c_1,c_2)^\top,\
\hat{x}\in\Omega\}, \\
&\Gamma_B=\{p_B(\hat{x})=b+R\hat{x}: ~b=(b_1,b_2)^\top,\ \hat{x}\in\Omega\},
\end{align*} 
where $\Omega=\{\hat{x}(t)=(\cos t, \sin t)^\top: ~0\leq t< 2\pi\}$.
Using the parametric forms of the boundaries $\Gamma_D$ and $\Gamma_B$, we
introduce the parameterized integral operators which are still represented by
$S_\kappa$, $S_\kappa^\infty$, $K_\kappa$, and $H_\kappa$ for convenience, i.e.,
\begin{align*}
&\big(S_\kappa^{\sigma,\varrho}\varphi_{j,\sigma}\big)(t)=\frac{\mathrm{i}}{2}
\int_0^{2\pi}
H_0^{(1)}(\kappa|p_\varrho(t)-p_\sigma(\varsigma)|)\varphi_{j,\sigma}
(\varsigma)\mathrm{d}\varsigma, \\
&\big(S^\infty_\kappa(p,\varphi_{j,\sigma})\big)(t)=\gamma_{\kappa}\int_0^{2\pi}
\mathrm{e}^{-\mathrm{i}\kappa\hat{x}(t)\cdot
p(\varsigma)}\varphi_{j,\sigma}(\varsigma)\mathrm{d}\varsigma,
\\
&\big(K_\kappa^{\sigma,\varrho}\varphi_{j,\sigma}\big)(t)=\frac{1}{G(t)}\int_0^{
2\pi}\widetilde K^{\sigma,\varrho}
(t,\varsigma;\kappa)\varphi_{j,\sigma}(\varsigma)\mathrm{d}\varsigma,\\
&\big(H_\kappa^{\sigma,\varrho}\varphi_{j,\sigma}\big)(t)=\frac{1}{G(t)}\int_0^{
2\pi}\widetilde H^{\sigma,\varrho}
(t,\varsigma;\kappa)\varphi_{j,\sigma}(\varsigma)\mathrm{d}\varsigma,
\end{align*} 
where the integral kernels are
\begin{align*}
&\widetilde K^{\sigma,\varrho}
(t,\varsigma;\kappa)=\frac{\mathrm{i}\kappa}{2}n(t)\cdot[
p_\sigma(\varsigma)-p_\varrho(t)]
\frac{H_1^{(1)}(\kappa|p_\varrho(t)-p_\sigma(\varsigma)|)}{
|p_\varrho(t)-p_\sigma(\varsigma)|},
\\
&\widetilde H^{\sigma,\varrho}
(t,\varsigma;\kappa)=\frac{\mathrm{i}\kappa}{2}n(t)^\perp\cdot[
p_\sigma(\varsigma)-p_\varrho(t)]
\frac{H_1^{(1)}(\kappa|p_\varrho(t)-p_\sigma(\varsigma)|)}{
|p_\varrho(t)-p_\sigma(\varsigma)|}.
\end{align*}
Here
$\varphi_{j,\sigma}(\varsigma)=G_\sigma(\varsigma)g_j(p_\sigma(\varsigma))$,
$j=1, 2$, $\sigma=D, B$, where
$G_D(\varsigma):=|p'(\varsigma)|=\sqrt{(r'(\varsigma))^2+r^2(\varsigma)}$ and
$G_B=R$ are Jacobian of the transformation, and  
\begin{align*}
& n(t)=\nu(x(t))|p'(t)|=\big(p'_2(t), -p'_1(t)\big),\\
& n(t)^\perp=\tau(x(t))|p'(t)|=\big(p'_1(t), p'_2(t)\big).
\end{align*}
The field equations
and data equation \eqref{phafield eqn1}--\eqref{phafield eqn4} can be
reformulated as the parametrized integral equations
\begin{align}
-&\varphi_{1,D} + (K_{\kappa_{\rm p}}^{D,D}\varphi_{1,D}) G_D + (K_{\kappa_{\rm
p}}^{B,D}\varphi_{1,B}) G_D + (H_{\kappa_{\rm s}}^{D,D}\varphi_{2,D})G_D +
(H_{\kappa_{\rm s}}^{B,D}\varphi_{2,B})G_D = w_{1,D}, \label{phaparafield
eqn1}\\
&\varphi_{2,D} + (H_{\kappa_{\rm p}}^{D,D}\varphi_{1,D}) G_D + (H_{\kappa_{\rm
p}}^{B,D}\varphi_{1,B}) G_D -
(K_{\kappa_{\rm s}}^{D,D}\varphi_{2,D})G_D - (K_{\kappa_{\rm
s}}^{B,D}\varphi_{2,B})G_D = w_{2,D}, \label{phaparafield eqn2}\\
-&\varphi_{1,B} + (K_{\kappa_{\rm p}}^{D,B}\varphi_{1,D}) G_B +(K_{\kappa_{\rm
p}}^{B,B}\varphi_{1,B})G_B + (H_{\kappa_{\rm s}}^{D,B}\varphi_{2,D}) G_B +
(H_{\kappa_{\rm s}}^{B,B}\varphi_{2,B}) G_B = w_{1,B}, \label{phaparafield
eqn3}\\
&\varphi_{2,B} + (H_{\kappa_{\rm p}}^{D,B}\varphi_{1,D}) G_B + (H_{\kappa_{\rm
p}}^{B,B}\varphi_{1,B}) G_B -
(K_{\kappa_{\rm s}}^{D,B}\varphi_{2,D})G_B -
(K_{\kappa_{\rm s}}^{B,B}\varphi_{2,B})G_B = w_{2,B}, \label{phaparafield
eqn4}
\end{align}
and the data equation \eqref{phadata eqn} can be written as 
\begin{equation}\label{phaparadata eqn}
 a_{\rm p}\sum_\sigma \big|S^\infty_{\kappa_{\rm
p}}(p_\sigma,\varphi_{1,_\sigma})\big|^2+a_{\rm s} \sum_\sigma
\big|S^\infty_{\kappa_{\rm s}} (p_\sigma,\varphi_{2,_\sigma})\big|^2=a_{\rm
p}|\phi_\infty|^2+a_{\rm s}|\psi_\infty|^2, 
\end{equation}
with $w_{j,\sigma}=2(f_j\circ p_\sigma)G_\sigma$, $j=1, 2$, $\sigma=D, B$. 

It follows from the Fr\'{e}chet derivative operator
${S_\kappa^\infty}'[p;\varphi]q$ in \eqref{FreSinfty} that the linearization of
\eqref{phaparadata eqn} leads to
\begin{equation}\label{linear phaparadata eqn}
a_{\rm p}2\Re\Big(\overline{\sum_\sigma S^\infty_{\kappa_{\rm p}}
(p_\sigma,\varphi_{1,\sigma})}{S_{\kappa_{\rm
p}}^\infty}'[p_D;\varphi_{1,D}]q\Big)+a_{\rm s}2\Re\Big(\overline{\sum_\sigma
S^\infty_{\kappa_{\rm s}} (p_\sigma,\varphi_{2,\sigma})}{S_{\kappa_{\rm
s}}^\infty}'[p_D;\varphi_{2,D}]q\Big)=\breve{w},
\end{equation}
where
\begin{align*}
\breve{w}:= & a_{\rm p}\Big(|\phi_\infty|^2-\big|\sum_\sigma
S^\infty_{\kappa_{\rm p}} (p_\sigma,\varphi_{1,\sigma})\big|^2\Big)+a_{\rm s}
\Big(|\psi_\infty|^2-\big|\sum_\sigma S^\infty_{\kappa_{\rm s}}
(p_\sigma,\varphi_{2,\sigma})\big|^2\Big).
\end{align*}
Again, we may choose the following relative error estimator to terminate the
iteration
\begin{align}
{\rm E}_k:=&\frac{\displaystyle \left\|a_{\rm
p}\phi_k+a_{\rm s}\psi_k
\right\|_{L^2}} {\Big\|a_{\rm
p}|\phi_\infty|^2+a_{\rm s}|\psi_\infty|^2\Big\|_{L^2}} 
\leq\epsilon, \label{reference relativeerror}
\end{align}
where $\epsilon>0$ is the tolerance parameter which depends on the noise
level and $p_D^{(k)}$ is the $k$th approximation of the boundary $\Gamma_D$ and 
\begin{align*}
 \phi_k&=|\phi_\infty|^2-\big|S^\infty_{\kappa_{\rm p}}
(p_D^{(k)},\varphi_{1,D})+S^\infty_{\kappa_{\rm p}}
(p_B,\varphi_{1,B})\big|^2,\\
\psi_k&=|\psi_\infty|^2-\big|S^\infty_{\kappa_{\rm s}}
(p_D^{(k)},\varphi_{2,D})+S^\infty_{\kappa_{\rm s}}
(p_B,\varphi_{2,B})\big|^2.
\end{align*}

The iterative algorithm for the phaseless IOSP is given by {\bf Algorithm II}. 

\begin{table}[ht]
\begin{tabular}{cp{.8\textwidth}}
\toprule
\multicolumn{2}{l}{{\bf Algorithm II:}\quad Iterative algorithm
for the phaseless IOSP} \\
\midrule
Step 1 & Sent an incident plane wave $\boldsymbol{u}^{\rm
inc}$ with fixed $\omega, \lambda, \mu$ and a fixed incident direction
$d\in\Omega$, and then collect the corresponding phaseless far-field data
$|\phi_\infty|$ or $|\psi_\infty|$ for the scatterer $D\cup B$; \\
Step 2 & Select an initial star-like curve $\Gamma^{(0)}$
for the boundary $\Gamma_D$ and the error tolerance $\epsilon$. Set $k=0$; \\
Step 3 & For the curve $\Gamma^{(k)}$, compute the densities
$\varphi_{1,\sigma}$ and $\varphi_{2,\sigma}$ from \eqref{phaparafield
eqn1}--\eqref{phaparafield eqn4}; \\
Step 4 & Solve \eqref{linear phaparadata eqn} to obtain the
updated approximation $\Gamma^{(k+1)}:=\Gamma^{(k)}+q$ and evaluate the error
$E_{k+1}$ defined in \eqref{reference relativeerror}; \\
Step 5 & If $E_{k+1}\geq\epsilon$, then set $k=k+1$ and go
to Step 3. Otherwise, the current approximation $\Gamma^{(k+1)}$ is served as
the final reconstruction of $\Gamma_D$. \\
\bottomrule
\end{tabular}
\end{table}

\subsubsection{Discretization}

We point out that the kernels $\widetilde K^{\sigma,\varrho}$ and $\widetilde
H^{\sigma,\varrho}$ are singular when $\sigma=\varrho$. By the 
quadrature rules \eqref{quadrature1}--\eqref{quadrature2}, the full
discretization of \eqref{phaparafield eqn1}--\eqref{phaparafield eqn4} can be
deduced as follows.
\begin{align*}
w_{1,D}^{(n),i}= & -\varphi^{(n),i}_{1,D}+
\sum_{j=0}^{2n-1}\left(R_{|i-j|}^{(n)}\widetilde
K_1^D(\varsigma_i^{(n)},\varsigma_j^{(n)};\kappa_{\rm p})+
\frac{\pi}{n}\widetilde K_2^D(\varsigma_i^{(n)},\varsigma_j^{(n)};\kappa_{\rm
p})\right)\varphi_{1,D}^{(n),j} \\
&+\sum_{j=0}^{2n-1}\frac{\pi}{n}\widetilde
K^{B,D}(\varsigma_i^{(n)},\varsigma_j^{(n)};\kappa_{\rm
p})\varphi_{1,B}^{(n),j} 
+\sum_{j=0}^{2n-1}\frac{\pi}{n}\widetilde
H^{B,D}(\varsigma_i^{(n)},\varsigma_j^{(n)};\kappa_{\rm
s})\varphi_{2,B}^{(n),j}\\ 
&+\sum_{j=0}^{2n-1}\left(-T_{i-j}^{(n)}\widetilde
H_1^D(\varsigma_i^{(n)},\varsigma_j^{(n)};\kappa_{\rm s})+
R_{|i-j|}^{(n)}\widetilde H_2^D(\varsigma_i^{(n)},\varsigma_j^{(n)};\kappa_{\rm
s}) + \frac{\pi}{n}\widetilde
H_3^D(\varsigma_i^{(n)},\varsigma_j^{(n)};\kappa_{\rm
s})\right)\varphi_{2,D}^{(n),j}, \end{align*}
\begin{align*}
w_{2,D}^{(n),i}= & \varphi^{(n),i}_{2,D}+
\sum_{j=0}^{2n-1}\left(-T_{i-j}^{(n)}\widetilde
H_1^D(\varsigma_i^{(n)},\varsigma_j^{(n)};\kappa_{\rm
p})+R_{|i-j|}^{(n)}\widetilde
H_2^D(\varsigma_i^{(n)},\varsigma_j^{(n)};\kappa_{\rm p})+
\frac{\pi}{n}\widetilde H_3^D(\varsigma_i^{(n)},\varsigma_j^{(n)};\kappa_{\rm
p})\right)\varphi_{1,D}^{(n),j} \\
&+\sum_{j=0}^{2n-1}\frac{\pi}{n}\widetilde
H^{B,D}(\varsigma_i^{(n)},\varsigma_j^{(n)};\kappa_{\rm p})\varphi_{1,B}^{(n),j}
-
\sum_{j=0}^{2n-1}\left(R_{|i-j|}^{(n)}\widetilde
K_1^D(\varsigma_i^{(n)},\varsigma_j^{(n)};\kappa_{\rm s})+
\frac{\pi}{n}\widetilde K_2^D(\varsigma_i^{(n)},\varsigma_j^{(n)};\kappa_{\rm
s})\right)\varphi_{2,D}^{(n),j} \\
&-\sum_{j=0}^{2n-1}\frac{\pi}{n}\widetilde
K^{B,D}(\varsigma_i^{(n)},\varsigma_j^{(n)};\kappa_{\rm
s})\varphi_{2,B}^{(n),j},
\end{align*}
\begin{align*}
w_{1,B}^{(n),i}= & -\varphi^{(n),i}_{1,B}
+\sum_{j=0}^{2n-1}\frac{\pi}{n}\widetilde
K^{D,B}(\varsigma_i^{(n)},\varsigma_j^{(n)};\kappa_{\rm
p})\varphi_{1,D}^{(n),j}+\sum_{j=0}^{2n-1}\frac{\pi}{n}\widetilde
H^{D,B}(\varsigma_i^{(n)},\varsigma_j^{(n)};\kappa_{\rm s})\varphi_{2,D}^{(n),j}
\\
&+\sum_{j=0}^{2n-1}\left(R_{|i-j|}^{(n)}\widetilde
K_1^B(\varsigma_i^{(n)},\varsigma_j^{(n)};\kappa_{\rm p})+
\frac{\pi}{n}\widetilde K_2^B(\varsigma_i^{(n)},\varsigma_j^{(n)};\kappa_{\rm
p})\right)\varphi_{1,B}^{(n),j} \\
&+ \sum_{j=0}^{2n-1}\left(-T_{i-j}^{(n)}\widetilde
H_1^B(\varsigma_i^{(n)},\varsigma_j^{(n)};\kappa_{\rm
s})+R_{|i-j|}^{(n)}\widetilde
H_2^B(\varsigma_i^{(n)},\varsigma_j^{(n)};\kappa_{\rm s})+
\frac{\pi}{n}\widetilde H_3^B(\varsigma_i^{(n)},\varsigma_j^{(n)};\kappa_{\rm
s})\right)\varphi_{2,B}^{(n),j},
\end{align*}
\begin{align*}
w_{2,B}^{(n),i}= &\varphi^{(n),i}_{2,B}
+\sum_{j=0}^{2n-1}\frac{\pi}{n}\widetilde
H^{D,B}(\varsigma_i^{(n)},\varsigma_j^{(n)};\kappa_{\rm
p})\varphi_{1,D}^{(n),j}-\sum_{j=0}^{2n-1}\frac{\pi}{n}\widetilde
K^{D,B}(\varsigma_i^{(n)},\varsigma_j^{(n)};\kappa_{\rm s})\varphi_{2,D}^{(n),j}
\\
&+ \sum_{j=0}^{2n-1}\left(-T_{i-j}^{(n)}\widetilde
H_1^B(\varsigma_i^{(n)},\varsigma_j^{(n)};\kappa_{\rm
p})+R_{|i-j|}^{(n)}\widetilde
H_2^B(\varsigma_i^{(n)},\varsigma_j^{(n)};\kappa_{\rm p})+
\frac{\pi}{n}\widetilde H_3^B(\varsigma_i^{(n)},\varsigma_j^{(n)};\kappa_{\rm
p})\right)\varphi_{1,B}^{(n),j} \\
& -\sum_{j=0}^{2n-1}\left(R_{|i-j|}^{(n)}\widetilde
K_1^B(\varsigma_i^{(n)},\varsigma_j^{(n)};\kappa_{\rm s})+
\frac{\pi}{n}\widetilde K_2^B(\varsigma_i^{(n)},\varsigma_j^{(n)};\kappa_{\rm
s})\right)\varphi_{2,B}^{(n),j},
\end{align*}
where $w_{l,\sigma}^{(n),i}=w_{l,\sigma}(\varsigma_i^{(n)})$,
$\varphi_{l,\sigma}^{(n),j}= \varphi_{l,\sigma}(\varsigma_j^{(n)})$ for
$i,j=0,\cdots,2n-1$, $l=1,2$, $\sigma=D, B$.

In addition, it is convenient to introduce the following definitions
\begin{align*}
&M_D(t,\varsigma;\kappa,\varphi):=\gamma_\kappa\exp\left\{-\mathrm{i}
\kappa\Big(c_1\cos t+c_2\sin t+r(\varsigma)
\cos(t-\varsigma)\Big)\right\}\varphi(\varsigma), \\
&M_B(t,\varsigma;\kappa,\varphi):=\gamma_\kappa\exp\left\{-\mathrm{i}
\kappa\Big(c_1\cos t+c_2\sin t+R
\cos(t-\varsigma)\Big)\right\}\varphi(\varsigma),\\
&S_{\kappa_{\rm p}}^\infty(\varsigma_i^{(\bar{n})})
=\frac{\pi}{n}\sum_{j=0}^{2n-1}\Big(M_D(\varsigma_i^{(\bar{n})},\varsigma_j^{(n)
};\kappa_{\rm p},\varphi_{1,D})+M_B(\varsigma_i^{(\bar{n})},\varsigma_j^{(n)}
;\kappa_{\rm p},\varphi_{1,B})\Big),\\
&S_{\kappa_{\rm s}}^\infty(\varsigma_i^{(\bar{n})})
=\frac{\pi}{n}\sum_{j=0}^{2n-1}\Big(M_D(\varsigma_i^{(\bar{n})},\varsigma_j^{
(n)};\kappa_{\rm s},\varphi_{2,D})
+M_B(\varsigma_i^{(\bar{n})},\varsigma_j^{(n)};\kappa_{\rm
s},\varphi_{2,B})\Big).
\end{align*}
Then, we get the discretized linear system
\begin{align}\label{phaseless discrelinear}
a_{\rm p}&\bigg(\sum_{l=1}^2 A_{l}^c(\varsigma_i^{(\bar{n})};\kappa_{\rm
p},\varphi_{1})\Delta
c_l+\sum_{m=0}^M\alpha_mA_{1,m}^r(\varsigma_i^{(\bar{n})};\kappa_{\rm
p},\varphi_{1})
+\sum_{m=1}^M\beta_mA_{2,m}^r(\varsigma_i^{(\bar{n})};\kappa_{\rm
p},\varphi_{1})\bigg) \nonumber\\
&+ a_{\rm s}\bigg(\sum_{l=1}^2 A_j^c(\varsigma_i^{(\bar{n})};\kappa_{\rm
s},\varphi_{2})\Delta c_l
+\sum_{m=0}^M\alpha_mA_{1,m}^r(\varsigma_i^{(\bar{n})};\kappa_{\rm
s},\varphi_{2})
+\sum_{m=1}^M\beta_mA_{2,m}^r(\varsigma_i^{(\bar{n})};\kappa_{\rm
s},\varphi_{2})\bigg)\notag\\
&\hspace{2cm}= \breve{w}(\varsigma_i^{(\bar{n})}),
\end{align}
which is to determine the real coefficients $\Delta c_1$, $\Delta c_2$,
$\alpha_m$ and $\beta_m$. Here 
\begin{align*}
A_l^c(\varsigma_i^{(\bar{n})};\kappa,\varphi)=2\Re\Big\{\frac{\pi}{n}\overline{
S_{\kappa}^\infty(\varsigma_i^{(\bar{n})})}\sum_{j=0}^{2n-1}L_l(\varsigma_i^{
(\bar{n})},\varsigma_j^{(n)};\kappa,\varphi)\Big\}
\end{align*}
for $l=1, 2$, and
\begin{align*}
&A_{1,m}^r(\varsigma_i^{(\bar{n})};\kappa,\varphi)=2\Re\Big\{\frac{\pi}
{n}\overline{S_{\kappa}^\infty(\varsigma_i^{(\bar{n})})}\sum_{j=0}^{2n-1}L_{3,m}
(\varsigma_i^{(\bar{n})},\varsigma_j^{(n)};\kappa,\varphi)\Big\}, \\
&A_{2,m}^r(\varsigma_i^{(\bar{n})};\kappa,\varphi)=2\Re\Big\{\frac{\pi}
{n}\overline{S_{\kappa}^\infty(\varsigma_i^{(\bar{n})})}\sum_{j=0}^{2n-1}L_{4,m}
(\varsigma_i^{(\bar{n})},\varsigma_j^{(n)};\kappa,\varphi)\Big\}.
\end{align*}

Similarly, the overdetermined system \eqref{phaseless discrelinear} can be 
solved by using the Tikhonov regularization with an $H^2$ penalty term described
in Section 4.1.2. The details are omitted.

\begin{table}
\caption{Parametrization of the exact boundary curves.}\label{partable}
\begin{tabular}{lll}
\toprule  
Type           &Parametrization\\
\midrule  
Apple-shaped obstacle  &
$p_D(t)=\displaystyle\frac{0.55(1+0.9\cos{t}+0.1\sin{2t})}{1+0.75\cos{t}}(\cos{t
},\sin{t}), \quad t\in [0,2\pi]$ \\
~\\
Peanut-shaped obstacle & 
$p_D(t)=0.5\sqrt{0.25\cos^2{t}+\sin^2{t}}(\cos{t},\sin{t}), \quad
t\in[0,2\pi]$\\   
\bottomrule 
\end{tabular}
\end{table}

\begin{figure}
\centering 
\subfigure[Reconstruction with $1\%$ noise, $\epsilon=0.01$ ]
{\includegraphics[width=0.4\textwidth]{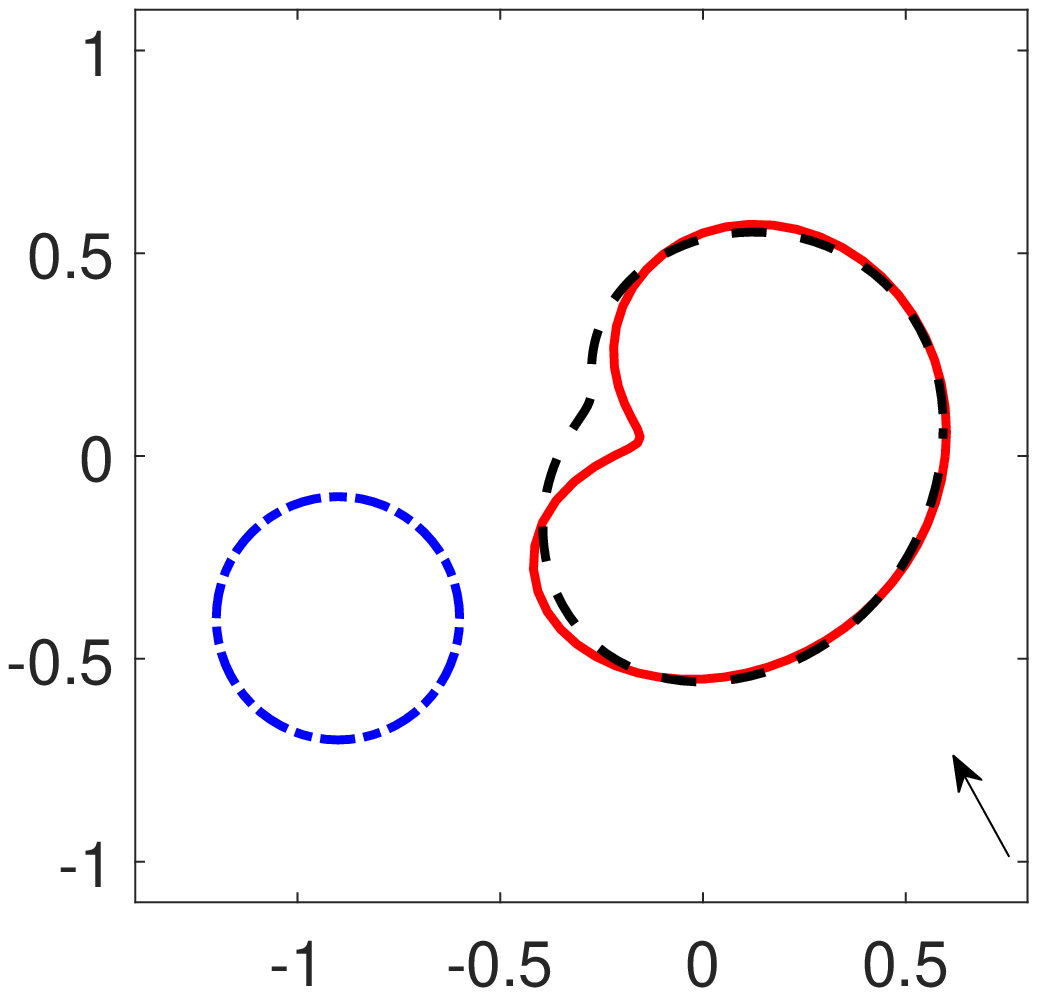}}
\subfigure[Relative error with $1\%$ noise]
{\includegraphics[width=0.4\textwidth]{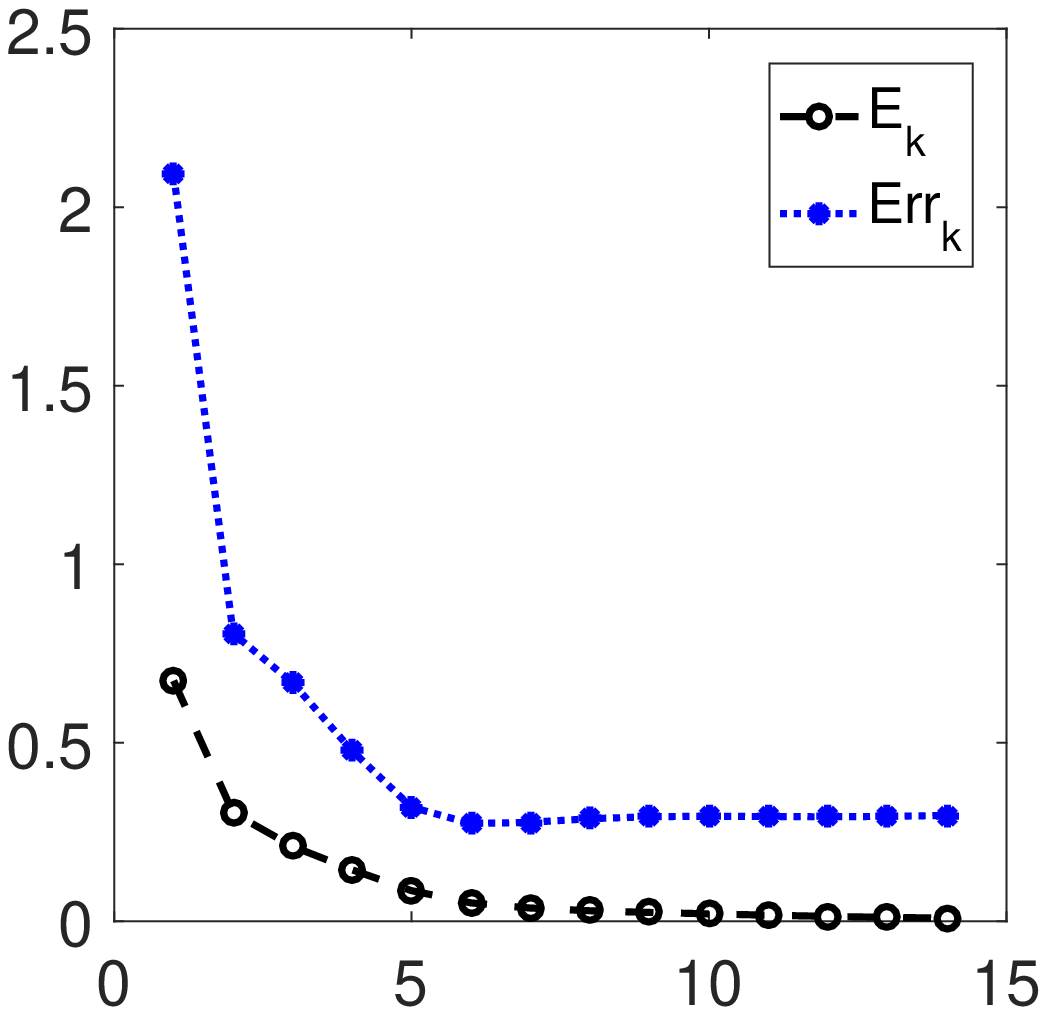}} 
\subfigure[Reconstruction with $5\%$ nois, $\epsilon=0.025$]
{\includegraphics[width=0.4\textwidth]{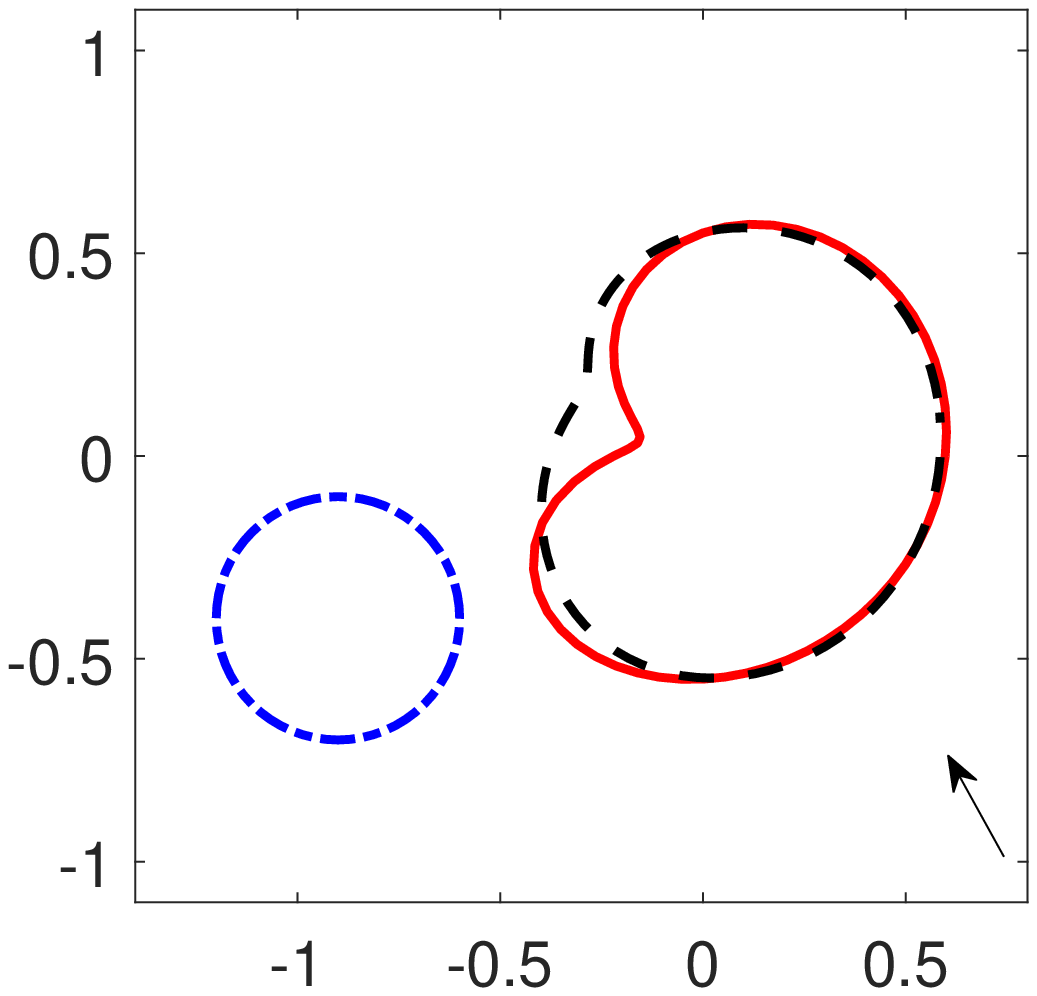}}
\subfigure[Relative error with $5\%$ noise]
{\includegraphics[width=0.4\textwidth]{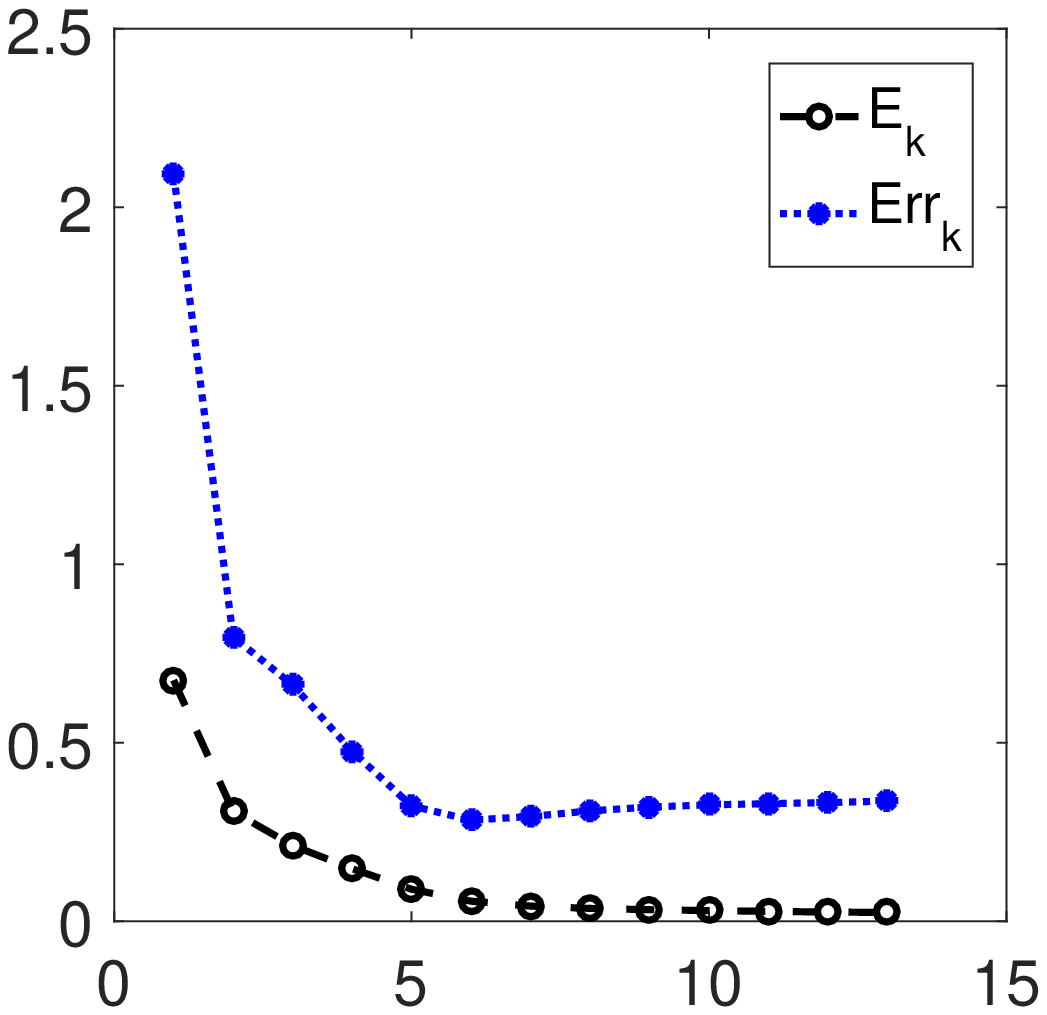}}
\caption{Reconstructions of an apple-shaped obstacle with phased data at
different
levels of noise (see example 1). The initial guess is given by
$(c_1^{(0)},c_2^{(0)})=(-0.9,
0.4), r^{(0)}=0.3$ and the incident angle $\theta=5\pi/8$.}\label{IOSP-2}
\end{figure}

\begin{figure}
\centering 
\subfigure[Reconstruction with $1\%$ noise, $\epsilon=0.006$ ]
{\includegraphics[width=0.4\textwidth]{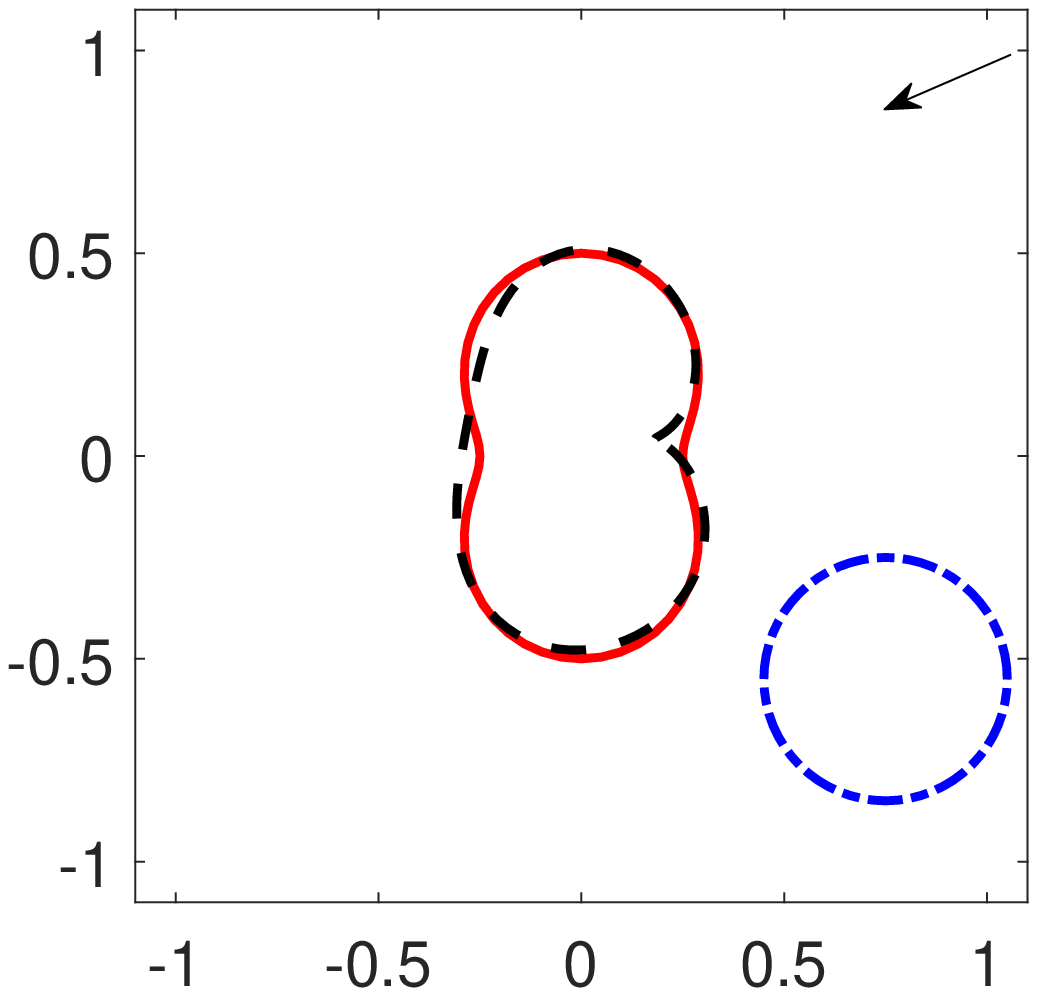}}
\subfigure[Relative error with $1\%$ noise]
{\includegraphics[width=0.4\textwidth]{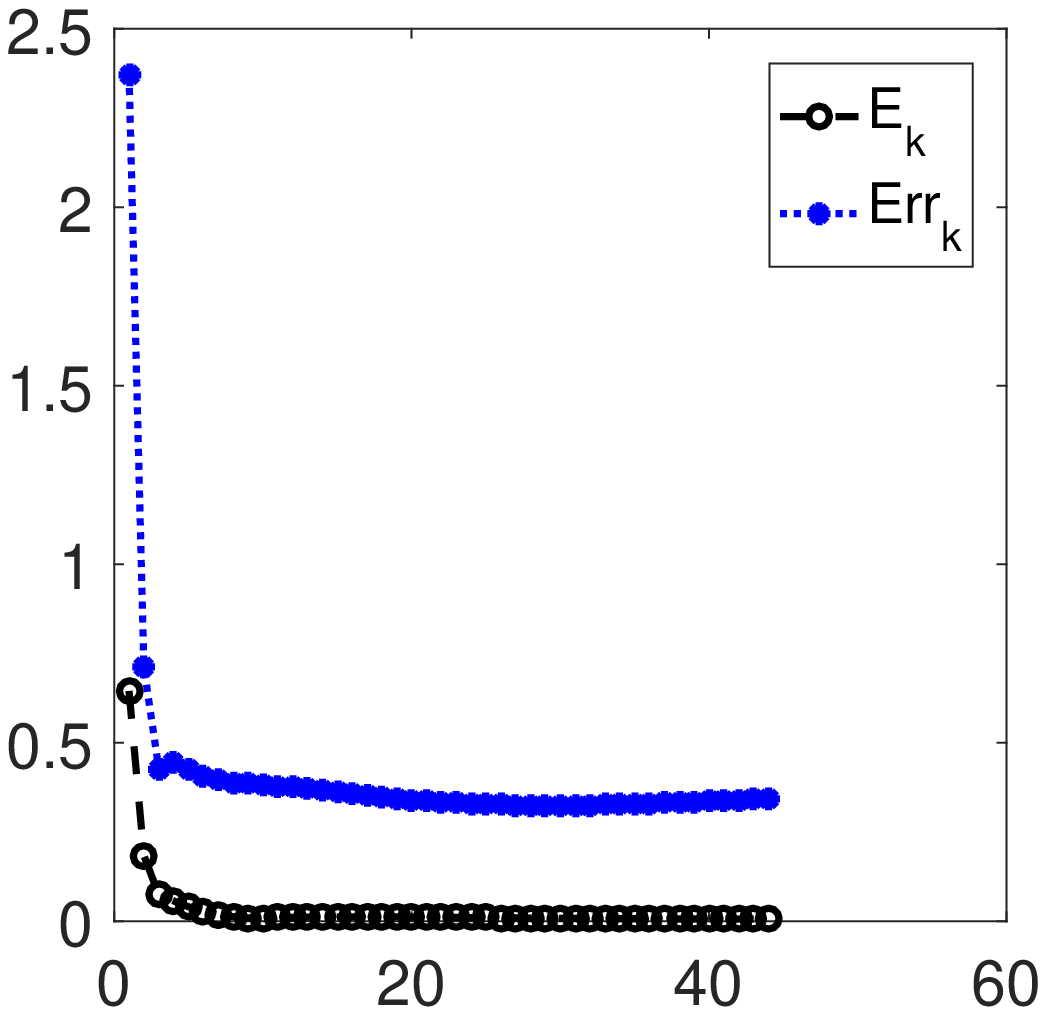}} 
\subfigure[Reconstruction with $5\%$ nois, $\epsilon=0.025$]
{\includegraphics[width=0.4\textwidth]{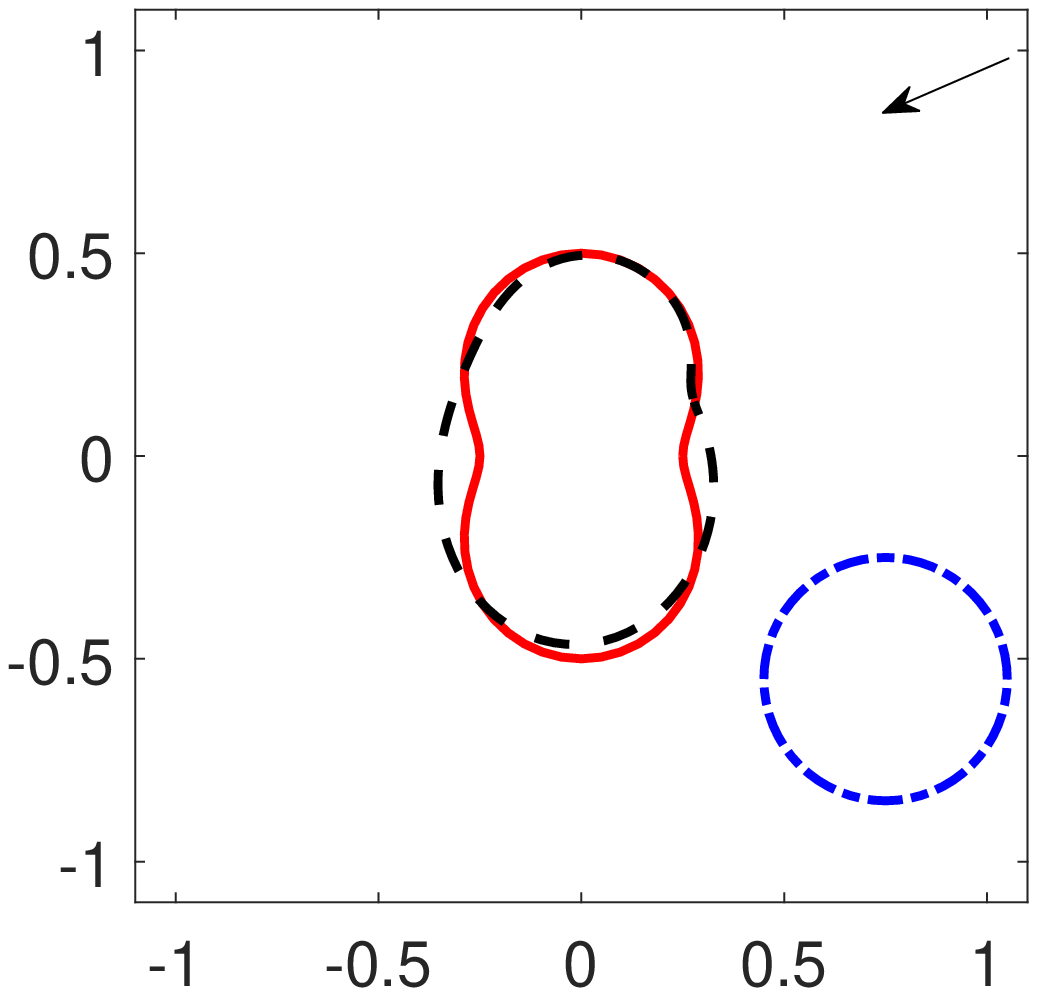}}
\subfigure[Relative error with $5\%$ noise]
{\includegraphics[width=0.4\textwidth]{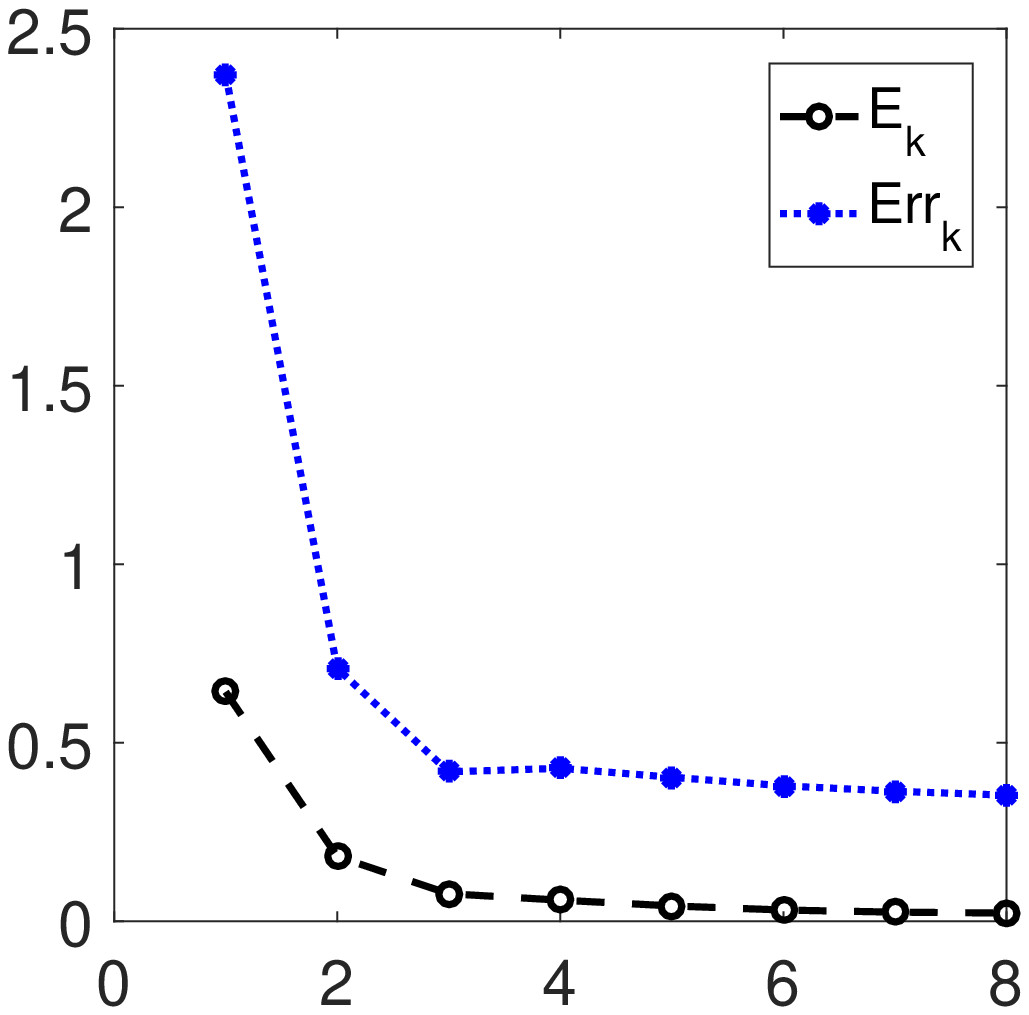}}
\caption{Reconstructions of a peanut-shaped obstacle with phased data at
different
levels of noise (see example 1). The initial guess is given by
$(c_1^{(0)},c_2^{(0)})=(0.75,
-0.55), r^{(0)}=0.3$ and the incident angle $\theta=7\pi/6$.}\label{IOSP-5}
\end{figure}

\begin{figure}
\centering 
\includegraphics[width=0.4\textwidth]{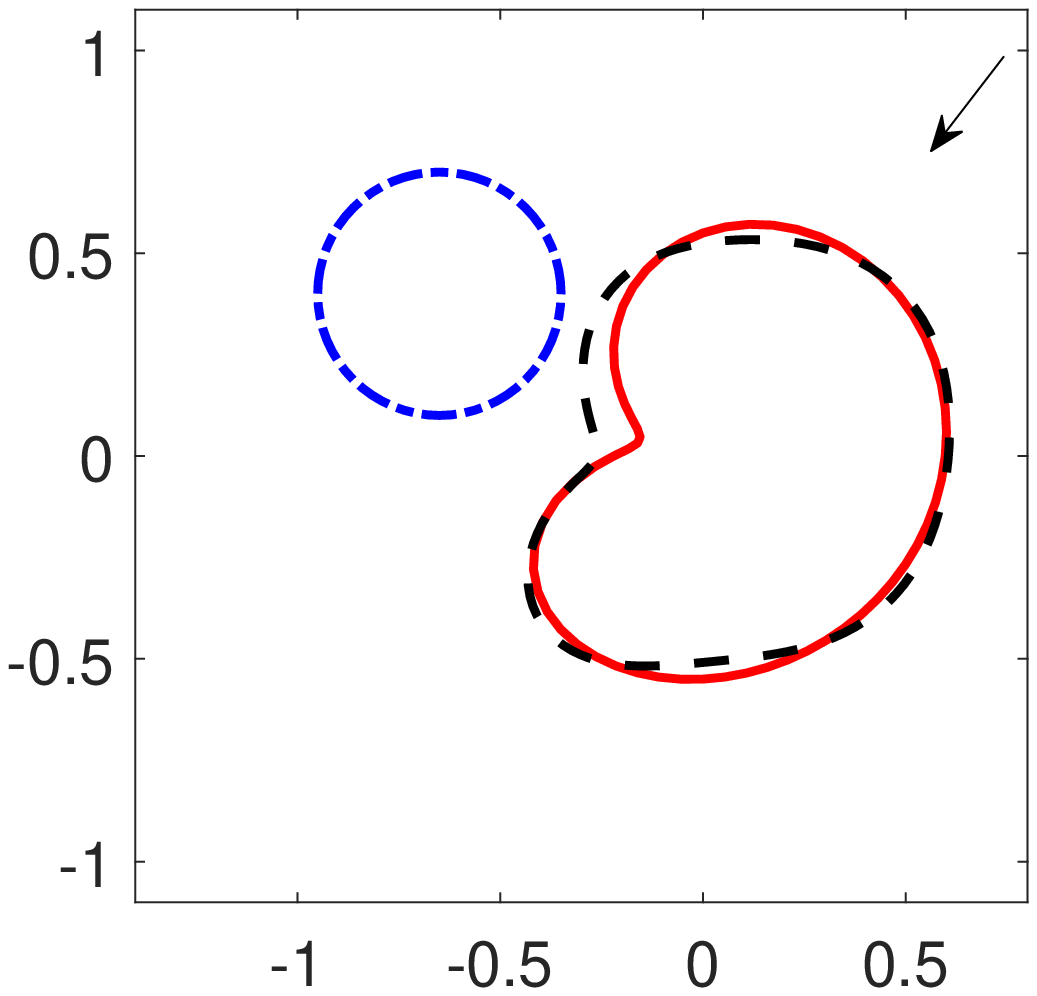}
\includegraphics[width=0.4\textwidth]{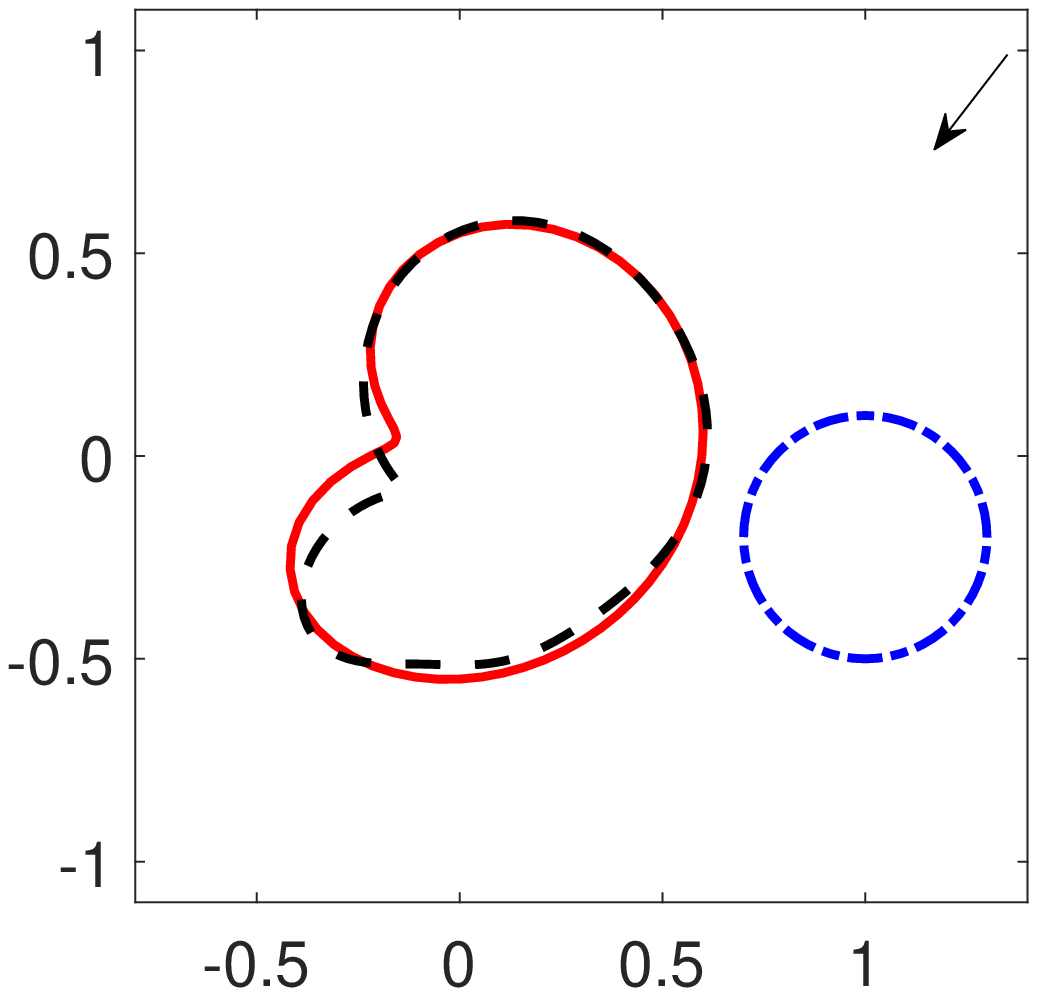}
\caption{Reconstructions of an apple-shaped obstacle with different
initial guesses, where $1\%$ noise is added and the incident angle
$\theta=4\pi/3$ (see example 1). (left) $(c_1^{(0)},c_2^{(0)})=(-0.65, 0.4)$,
$r^{(0)}=0.3$,
$\epsilon=0.015$; (right) $(c_1^{(0)},c_2^{(0)})=(1, -0.2)$, $r^{(0)}=0.3$,
$\epsilon=0.01$.}\label{IOSP-3}
\end{figure}

\begin{figure}
\centering 
\includegraphics[width=0.4\textwidth]{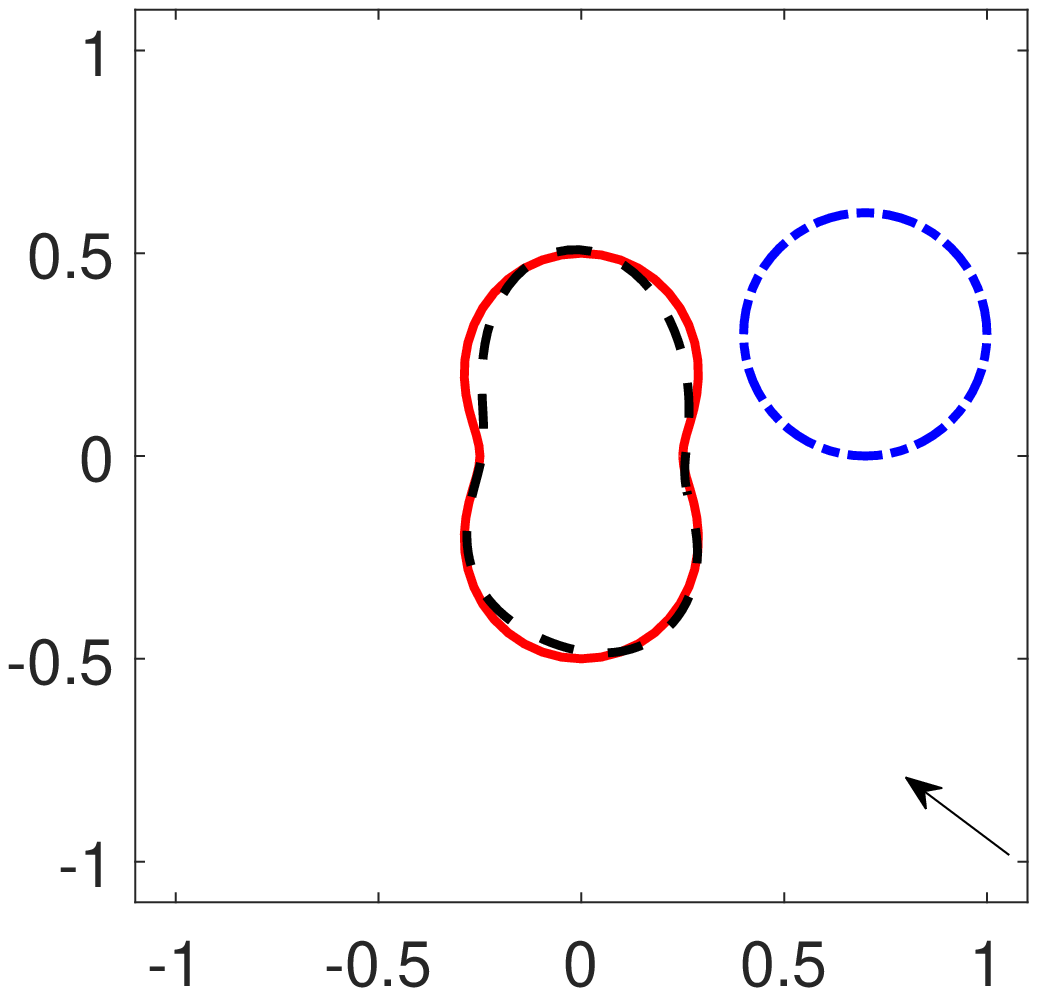}
\includegraphics[width=0.4\textwidth]{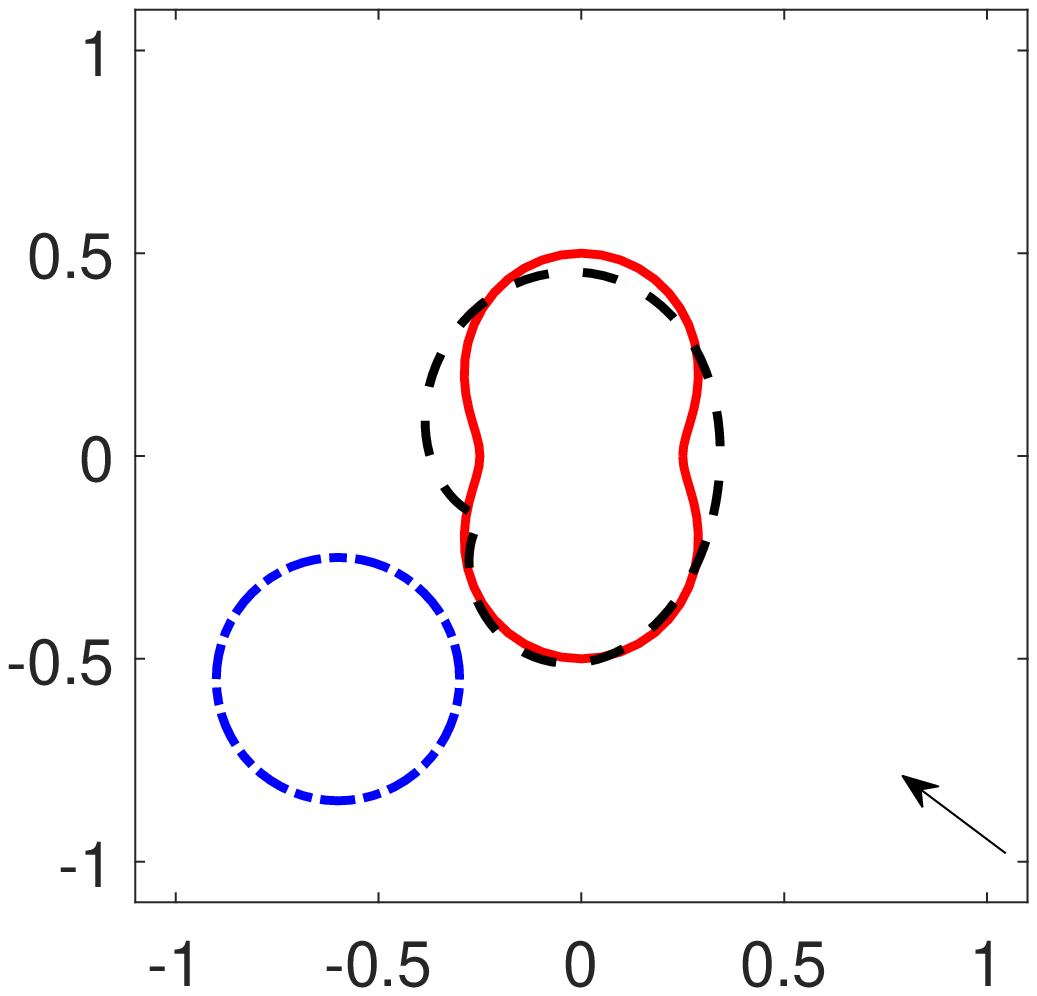}
\caption{Reconstructions of a penut-shaped obstacle with different
initial guesses, where $1\%$ noise is added and the incident angle
$\theta=7\pi/4$ (see example 1). (left) $(c_1^{(0)},c_2^{(0)})=(0.7, 0.3)$,
$r^{(0)}=0.3$,
$\epsilon=0.01$; (right) $(c_1^{(0)},c_2^{(0)})=(-0.6, -0.55)$, $r^{(0)}=0.3$,
$\epsilon=0.025$.}\label{IOSP-6}
\end{figure}

\begin{figure}
\centering 
\includegraphics[width=0.4\textwidth]{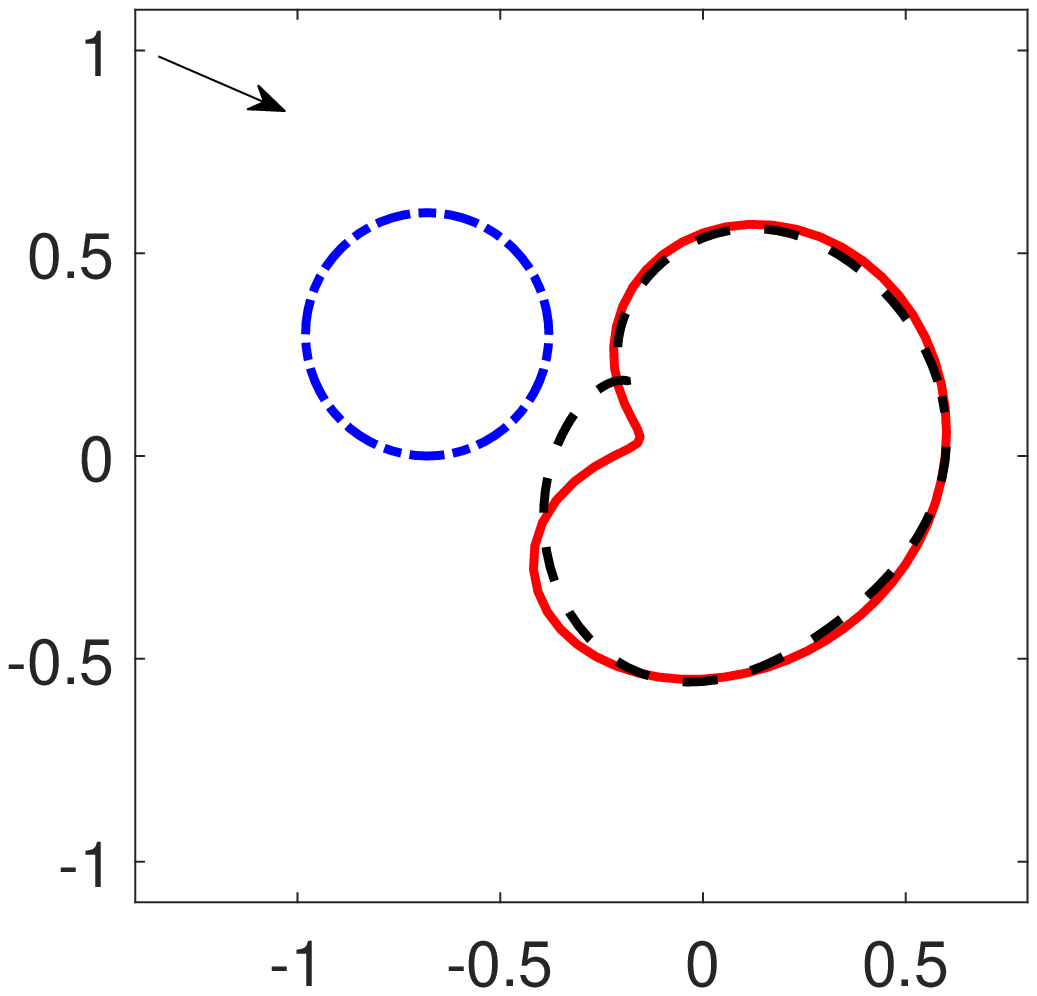}
\includegraphics[width=0.4\textwidth]{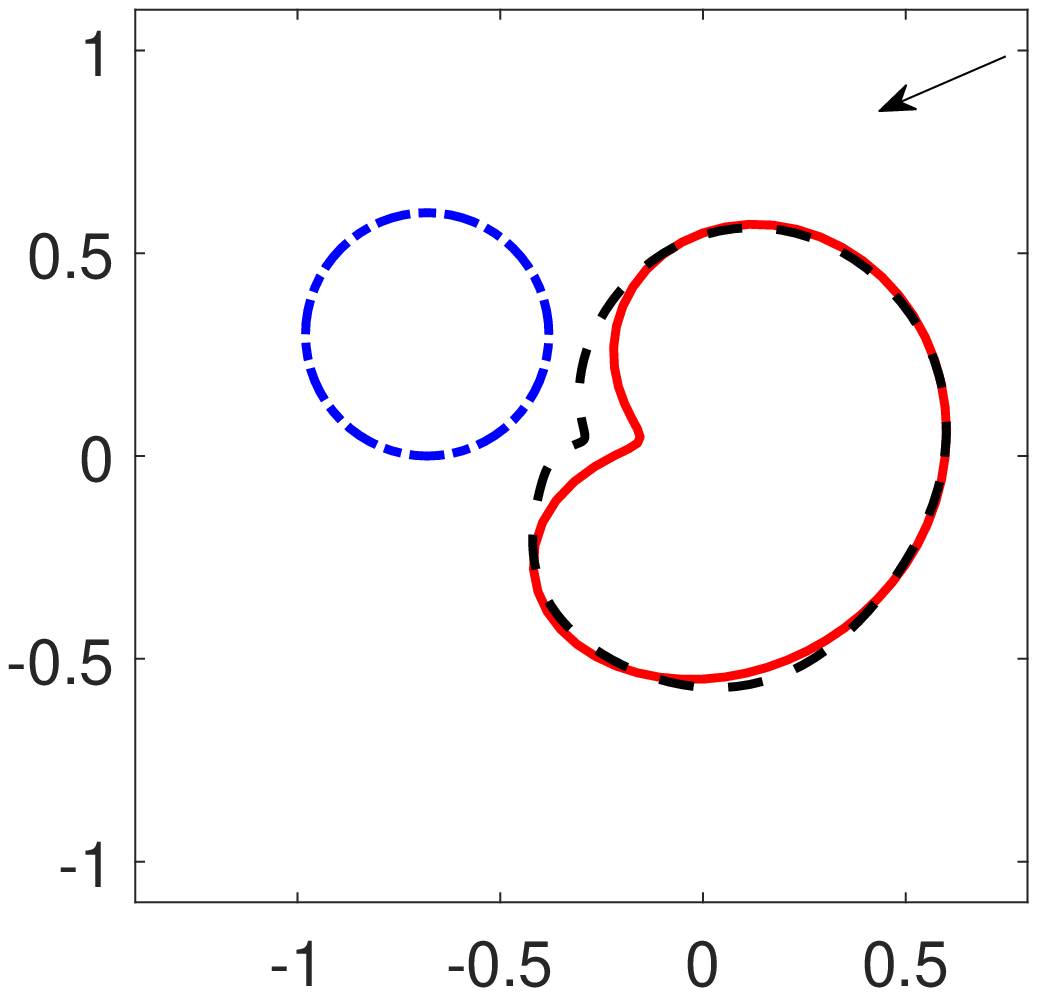}
\caption{Reconstructions of an apple-shaped obstacle with different incident
directions, where $1\%$ noise is added and the initial guess
is given by $(c_1^{(0)},c_2^{(0)})=(-0.7, 0.3)$, $r^{(0)}=0.3$ (see example 1).
(left) incident
angle $\theta=11\pi/6$, $\epsilon=0.01$; (right) incident angle $\theta=7\pi/6$,
$\epsilon=0.01$.}\label{IOSP-4}
\end{figure}

\begin{figure}
\centering 
\includegraphics[width=0.4\textwidth]{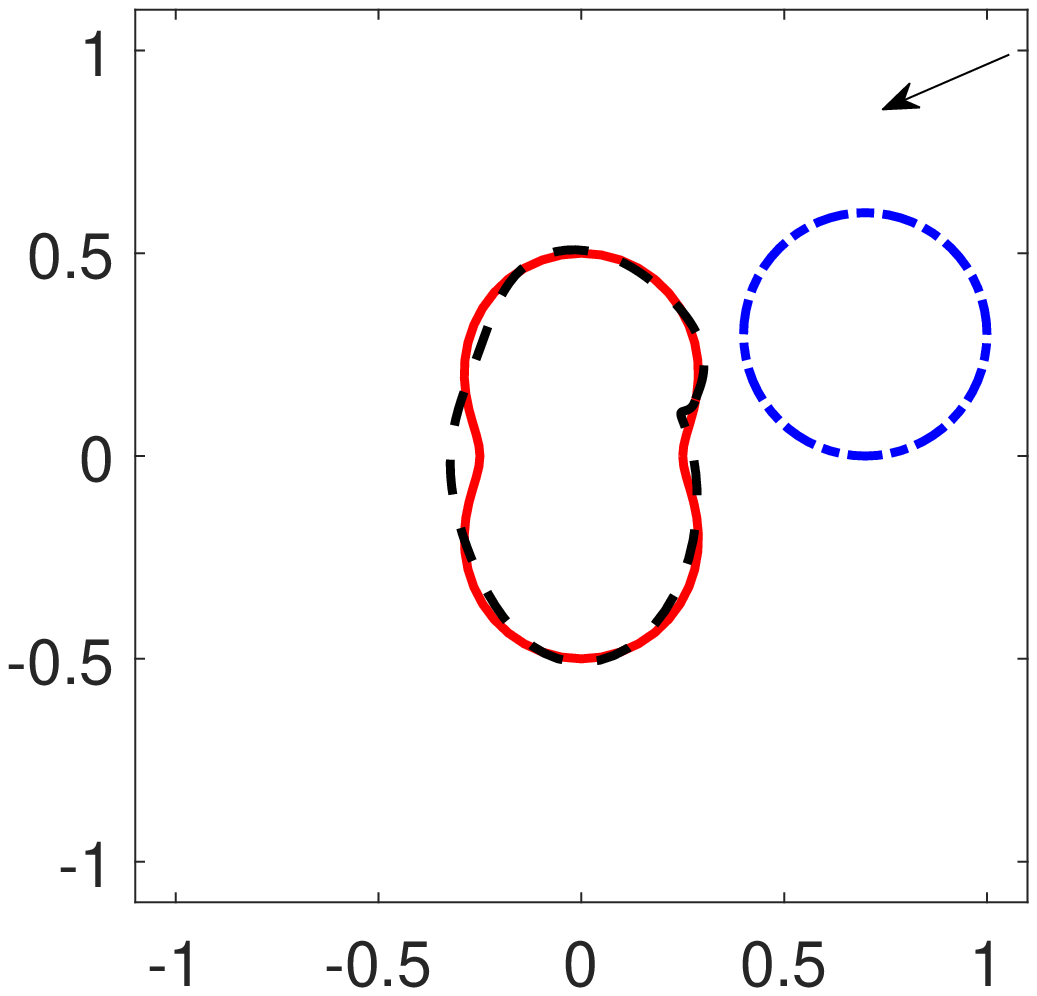}
\includegraphics[width=0.4\textwidth]{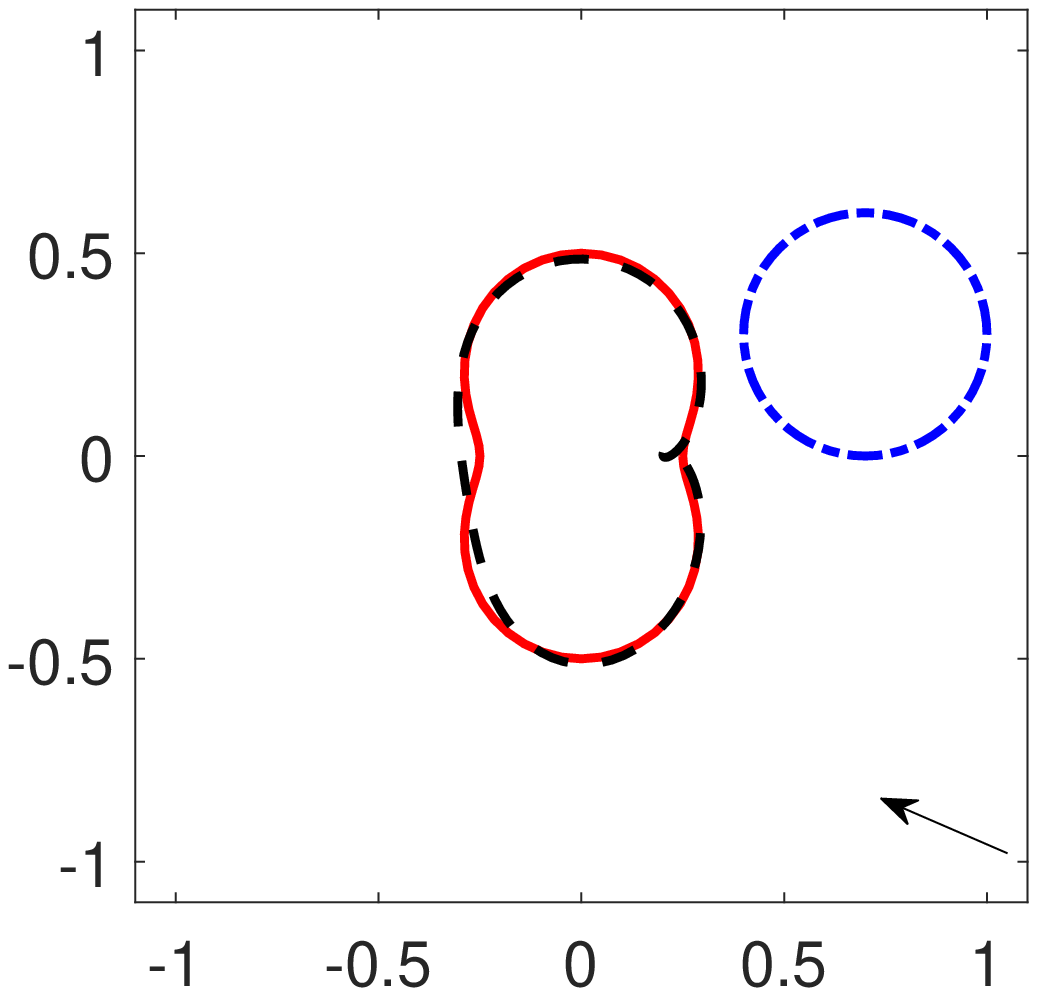}
\caption{Reconstructions of a peanut-shaped obstacle with different incident
directions, where $1\%$ noise is added and the initial guess
is given by $(c_1^{(0)},c_2^{(0)})=(0.7, 0.3)$, $r^{(0)}=0.3$ (see example 1).
(a) incident
angle $\theta=7\pi/6$, $\epsilon=0.006$; (b) incident angle $\theta=5\pi/6$,
$\epsilon=0.006$.}\label{IOSP-7}
\end{figure}

\begin{figure}
\centering 
\subfigure[Reconstruction with $1\%$ noise, $\epsilon=0.005$]
{\includegraphics[width=0.4\textwidth]{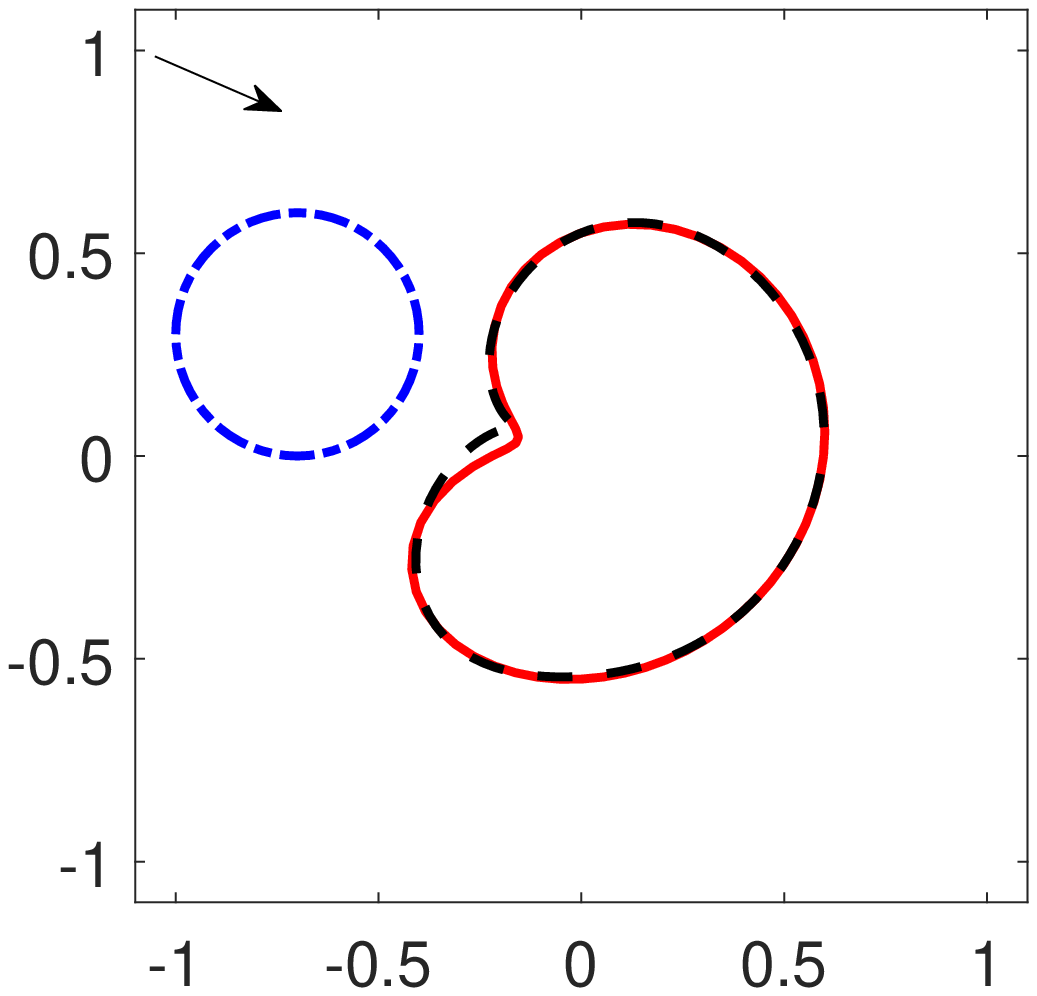}}
\subfigure[Relative error with $1\%$ noise]
{\includegraphics[width=0.4\textwidth]{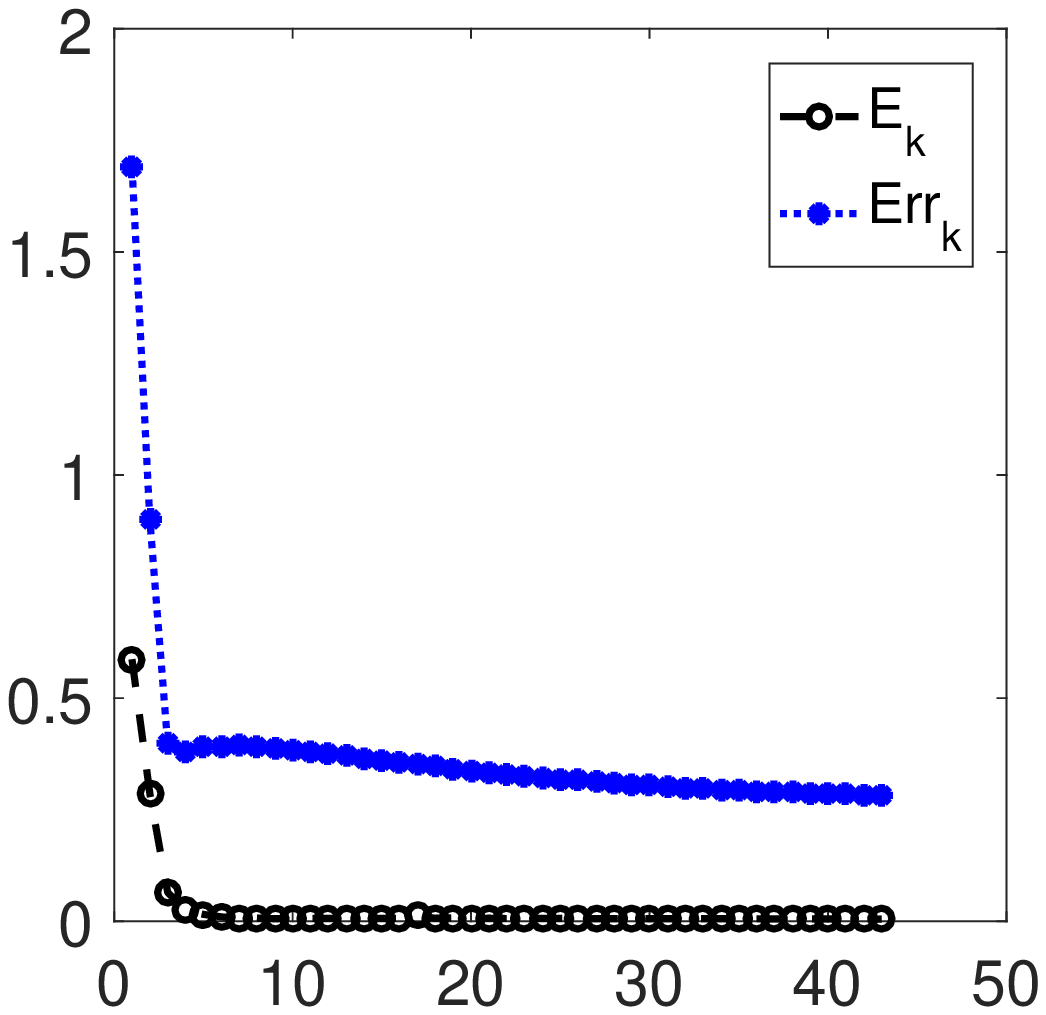}}
\subfigure[Reconstruction with $5\%$ nois, $\epsilon=0.025$]
{\includegraphics[width=0.4\textwidth]{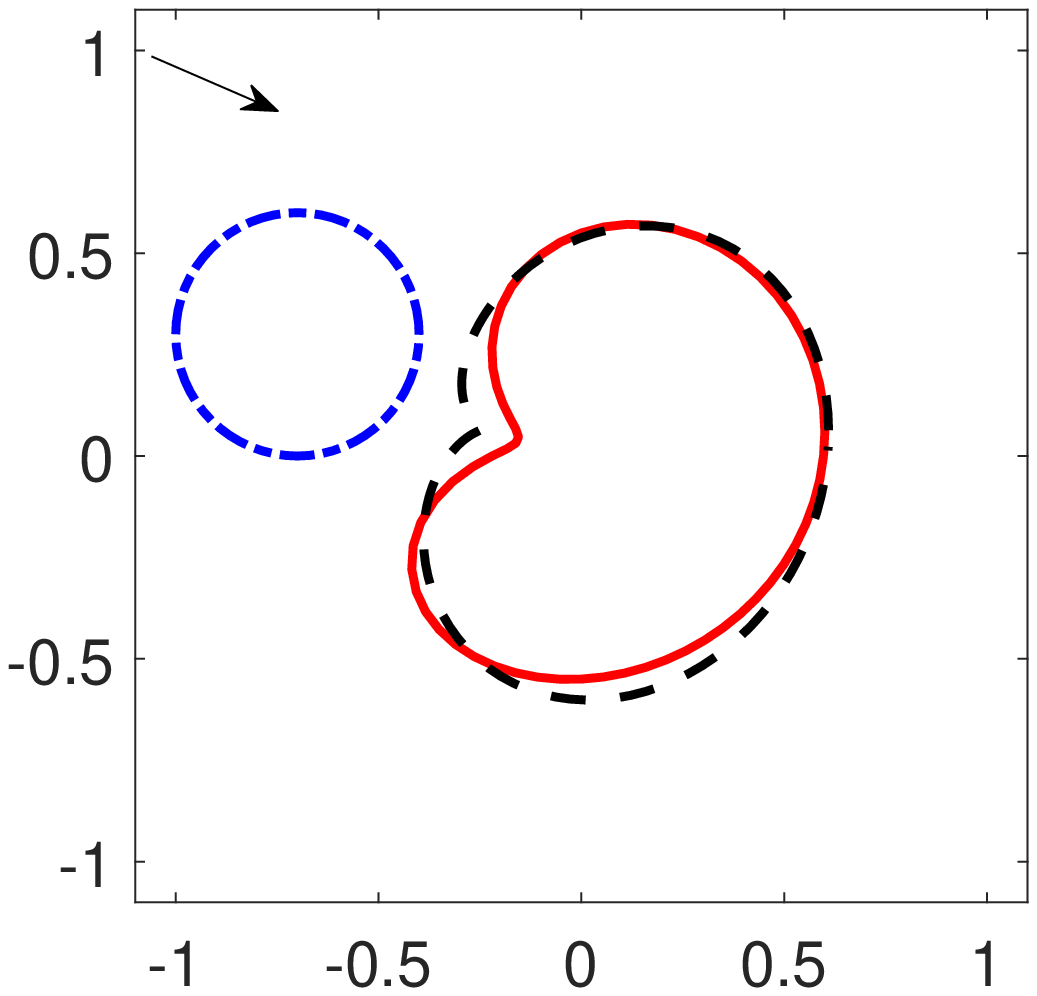}}
\subfigure[Relative error with $5\%$ noise]
{\includegraphics[width=0.4\textwidth]{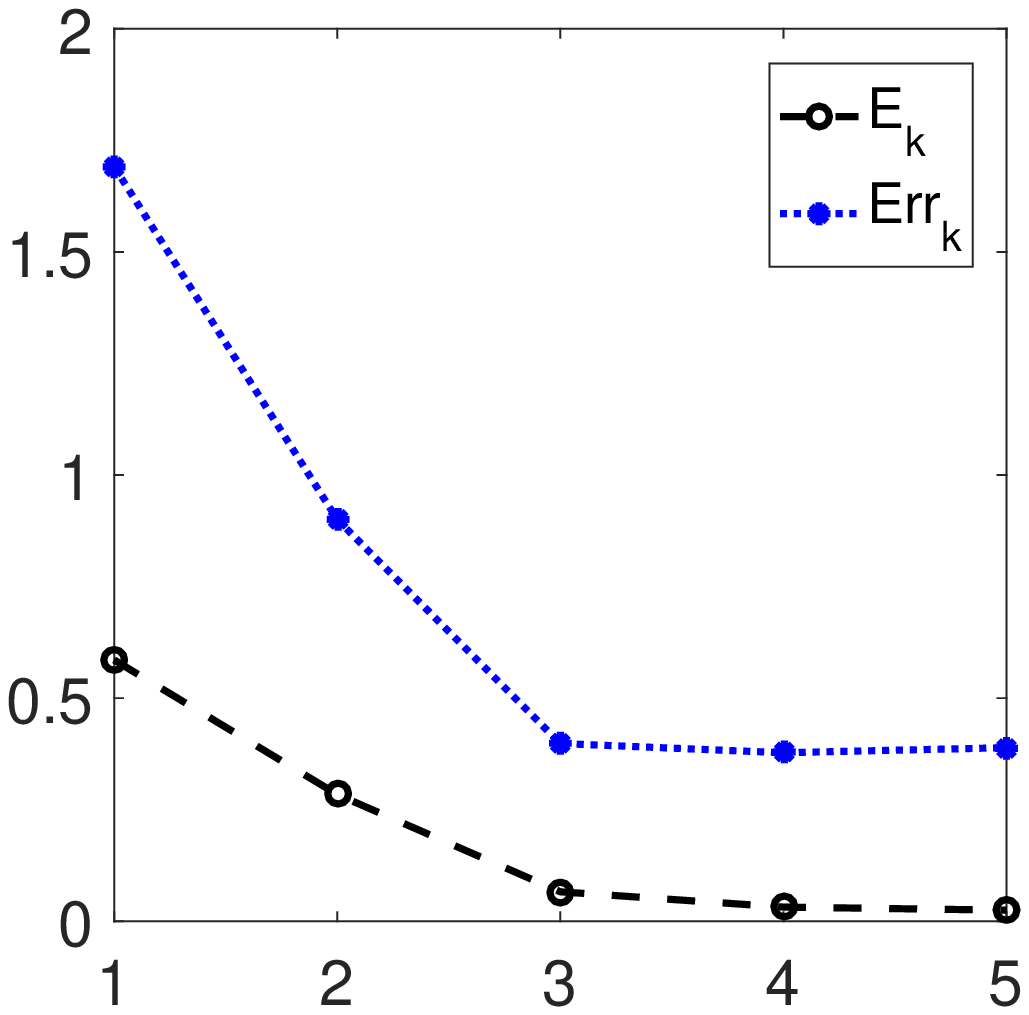}}
\caption{Reconstructions of an apple-shaped obstacle with phased data at
different
levels of noise and a reference ball (see example 2). The initial guess is given
by $(c_1^{(0)},c_2^{(0)})=(-0.7,
0.3), r^{(0)}=0.3$, the incident angle $\theta=11\pi/6$, and the reference
ball is $(b_1,b_2)=(5, 0), R=0.5$.}\label{IOSPreference-8}
\end{figure}

\begin{figure}
\centering 
\subfigure[Reconstruction with $1\%$ noise, $\epsilon=0.006$ ]
{\includegraphics[width=0.4\textwidth]{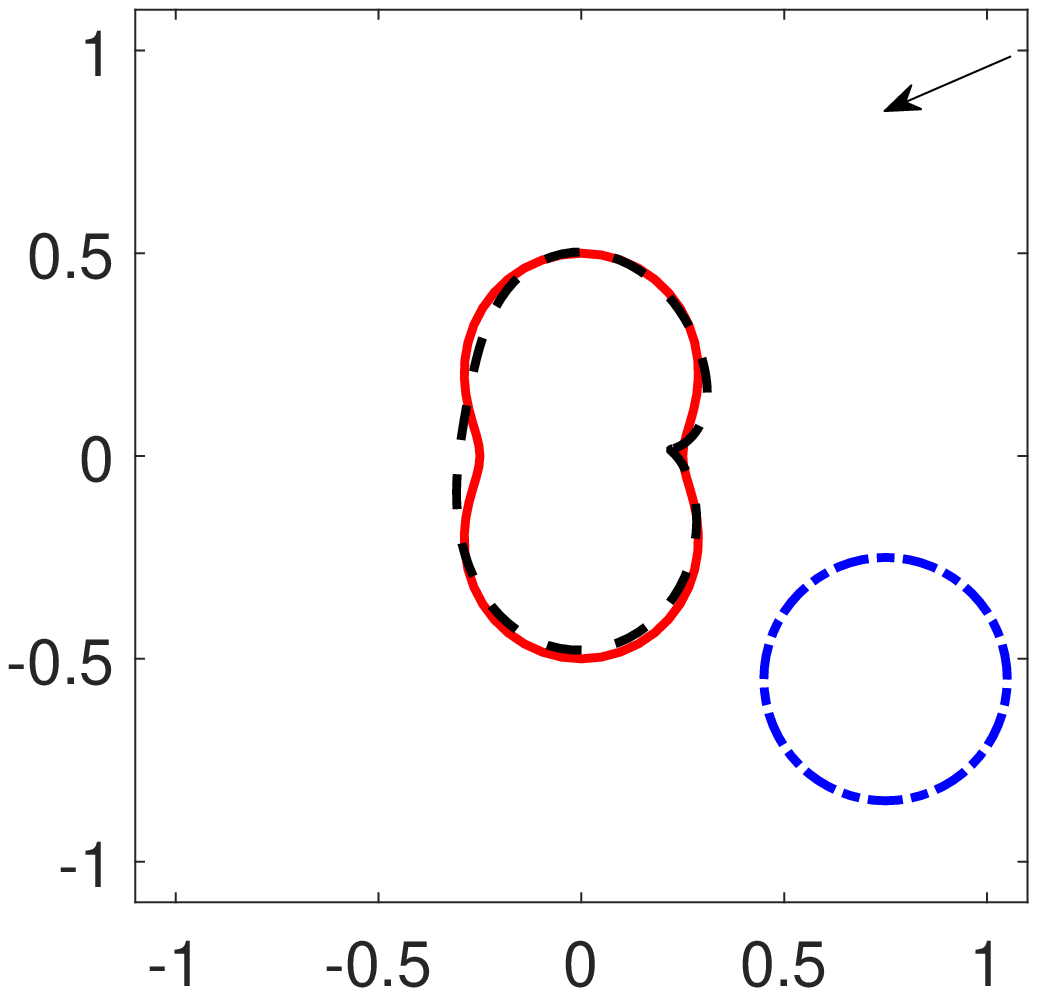}}
\subfigure[Relative error with $1\%$ noise]
{\includegraphics[width=0.4\textwidth]{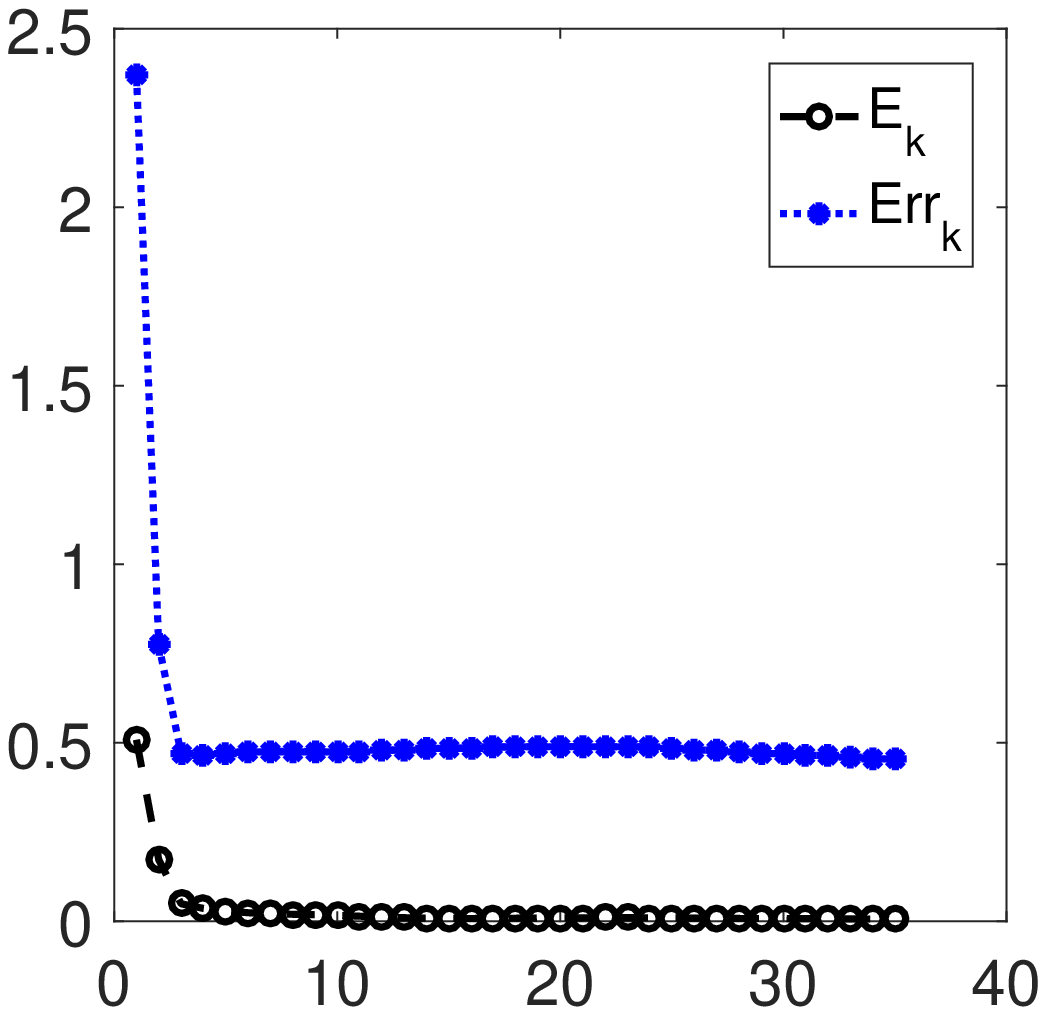}}
\subfigure[Reconstruction with $5\%$ nois, $\epsilon=0.025$]
{\includegraphics[width=0.4\textwidth]{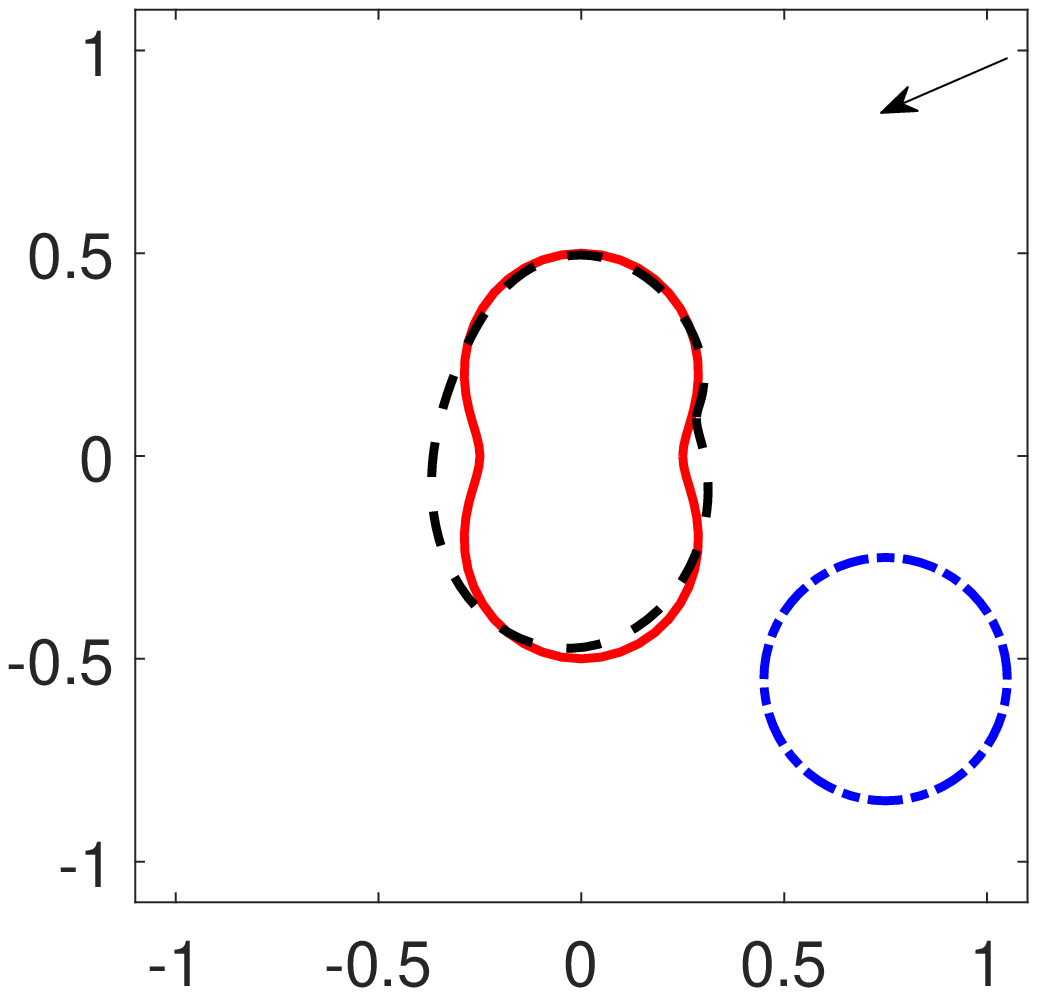}}
\subfigure[Relative error with $5\%$ noise]
{\includegraphics[width=0.4\textwidth]{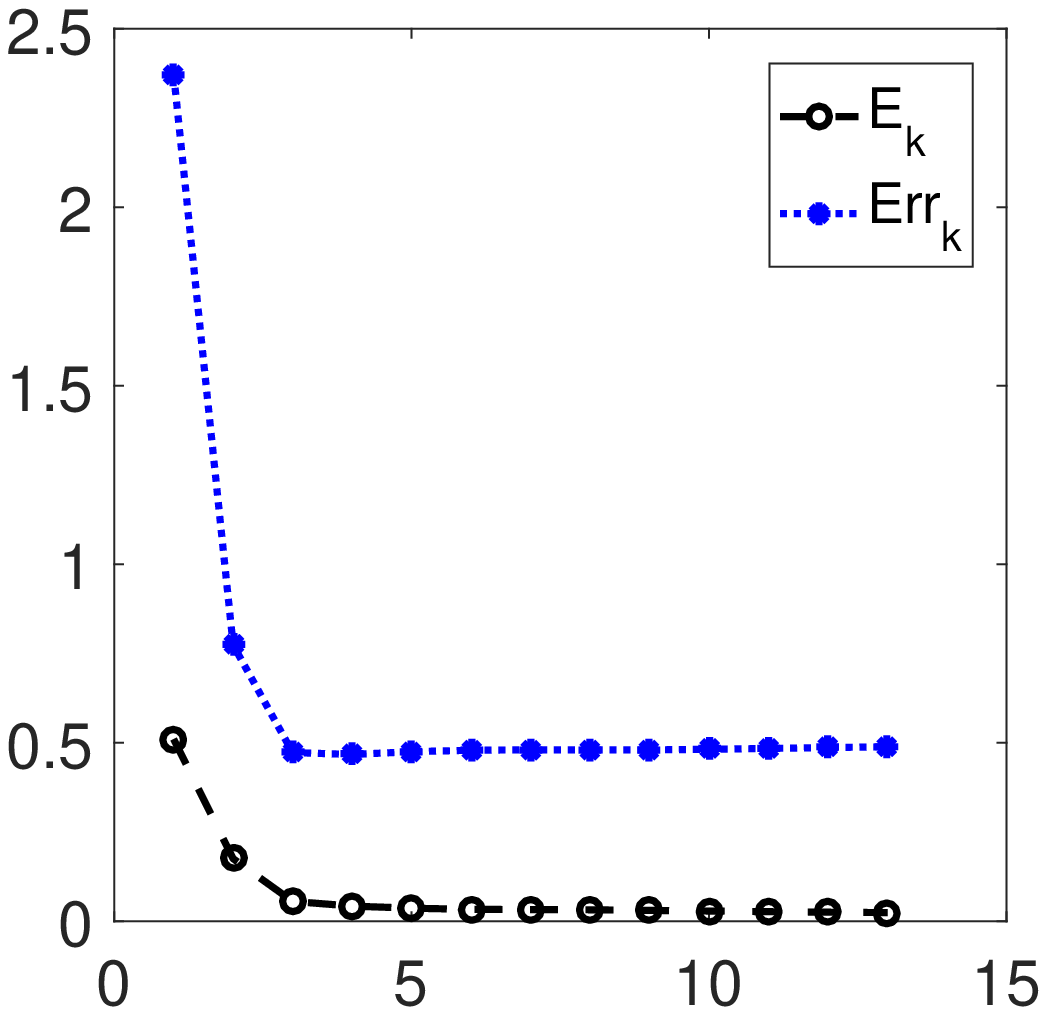}}
\caption{Reconstructions of a peanut-shaped obstacle with phased data at
different
levels of noise and a reference ball (see example 2). The initial guess is given
by $(c_1^{(0)},c_2^{(0)})=(0.75,
-0.55), r^{(0)}=0.3$, the incident angle $\theta=7\pi/6$, and the reference
ball is $(b_1,b_2)=(9, 0), R=0.5$.}\label{IOSPreference-12}
\end{figure}

\begin{figure}
\centering 
\includegraphics[width=0.4\textwidth]{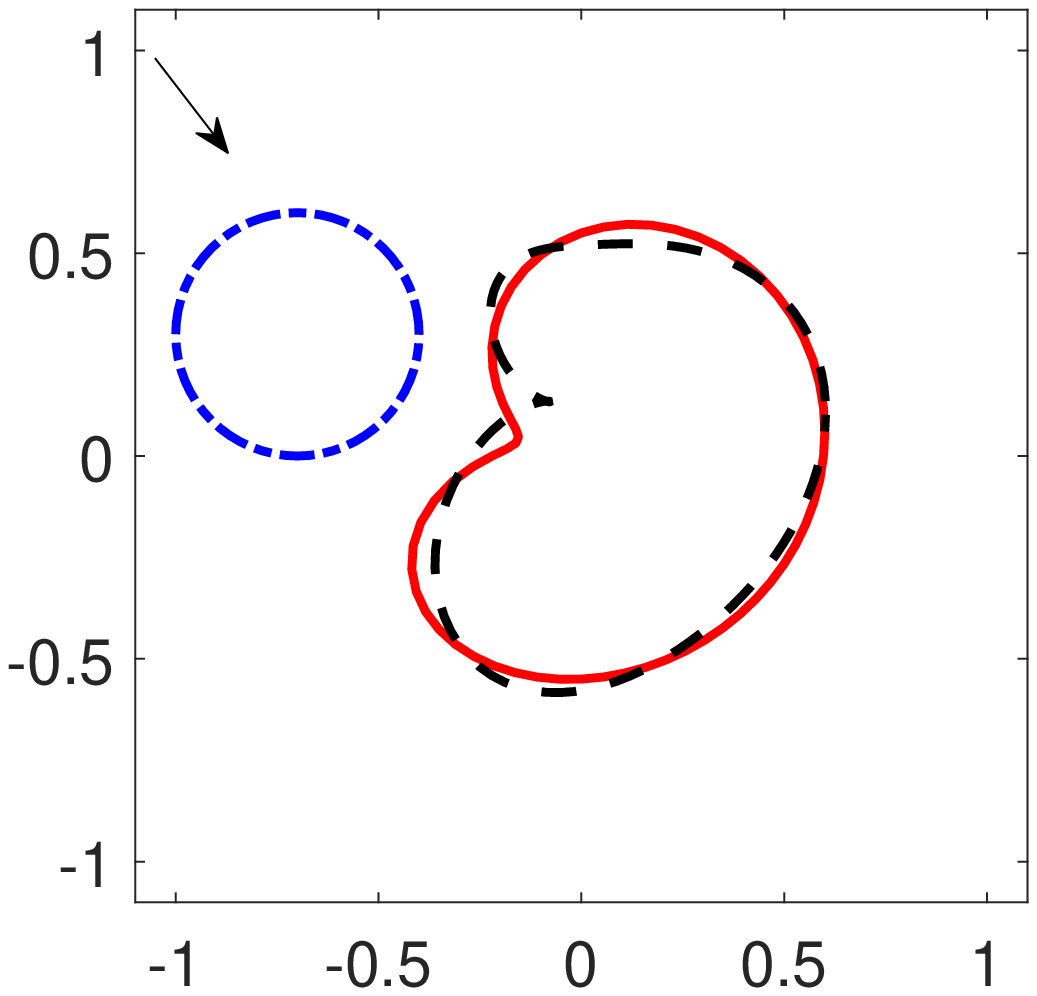}
\includegraphics[width=0.4\textwidth]{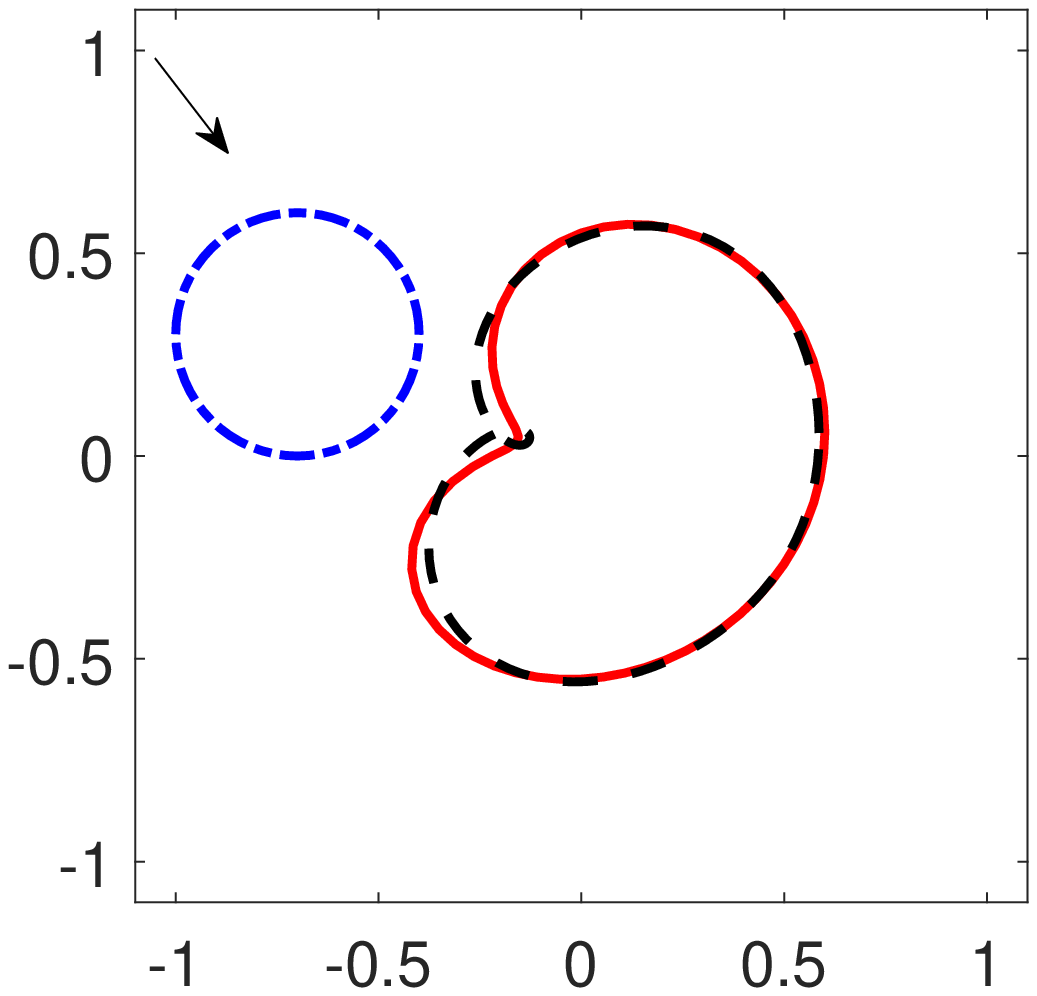}
\caption{Reconstructions of an apple-shaped obstacle with different reference
balls, where $1\%$ noise is added, the inciedent angle $\theta=5\pi/3$, and
the initial guess is given by $(c_1^{(0)},c_2^{(0)})=(-0.7, 0.3), r^{(0)}=0.3$
(see example 2).
(left) $(b_1, b_2)=(5, 0), R=0.4$, $\epsilon=0.01$; (right) $(b_1, b_2)=(6, 0),
R=0.9$, $\epsilon=0.006$.}\label{IOSPreference-11}
\end{figure}

\begin{figure}
\centering 
\includegraphics[width=0.4\textwidth]{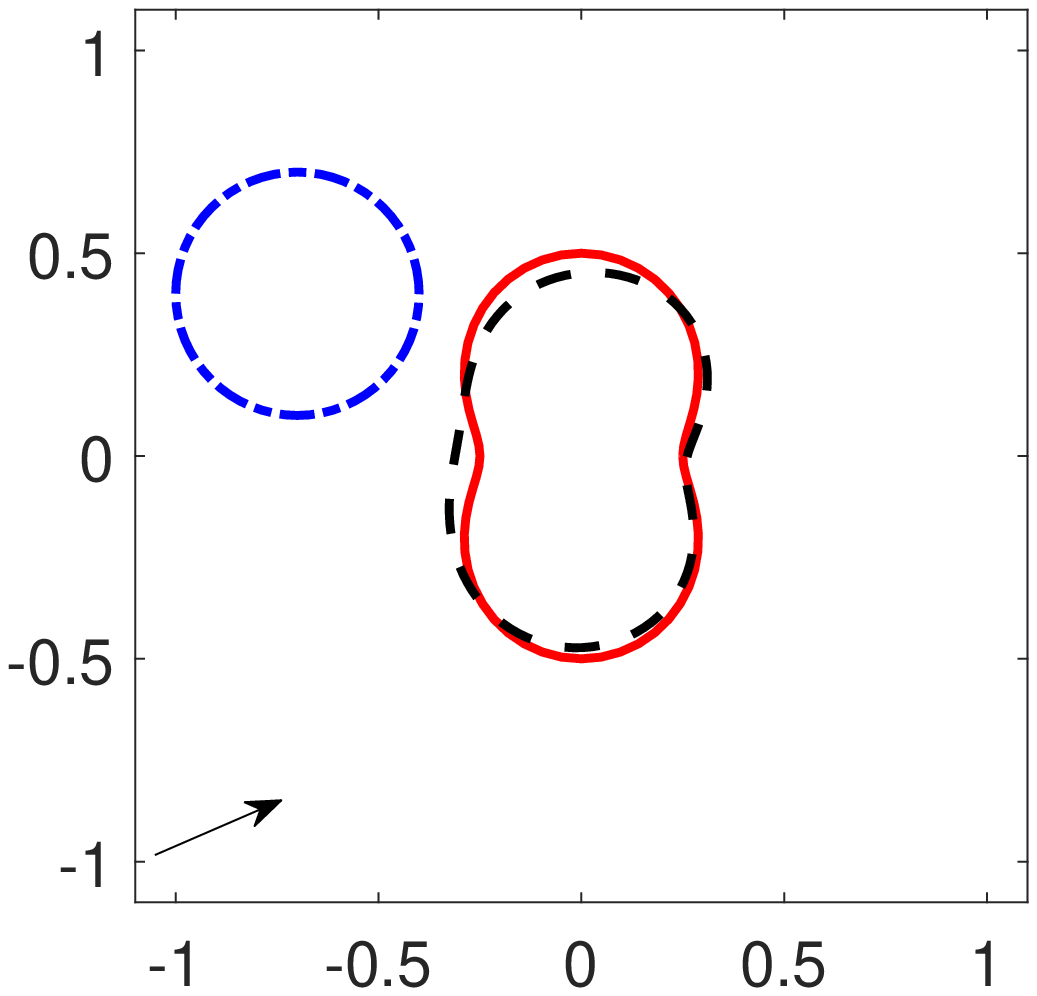}
\includegraphics[width=0.4\textwidth]{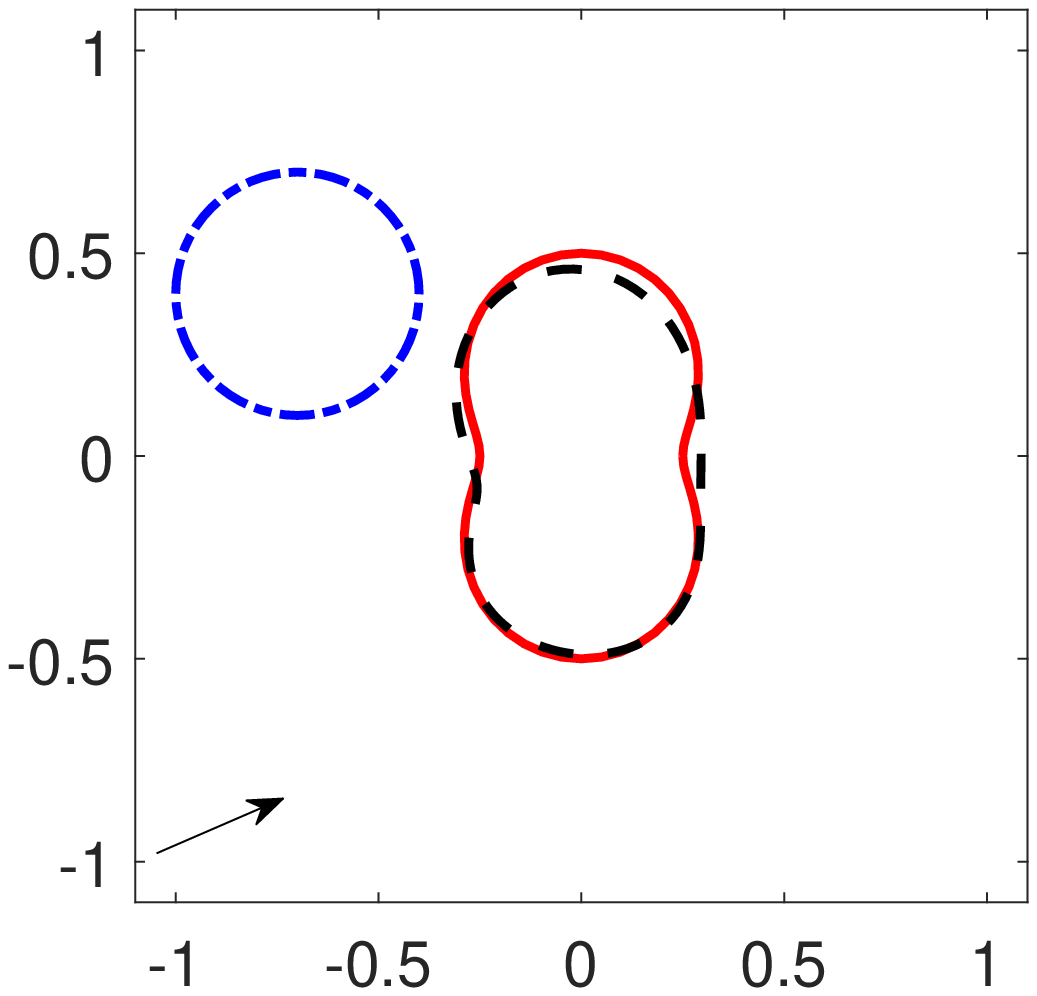}
\caption{Reconstructions of a peanut-shaped obstacle with different
reference balls, where $1\%$ noise is added, the incident angle $\theta=\pi/6$,
and the initial guess is given by $(c_1^{(0)},c_2^{(0)})=(-0.7, 0.4)$,
$r^{(0)}=0.3$ (see example 2). (left) $(b_1, b_2)=(7.5, 0), R=0.6$,
$\epsilon=0.01$; (right)
$(b_1, b_2)=(0, 7), R=0.6$, $\epsilon=0.01$.}\label{IOSPreference-14}
\end{figure}

\begin{figure}
\centering 
\subfigure[Recnstruction with $1\%$ noise, $\epsilon=0.005$ ]
{\includegraphics[width=0.4\textwidth]{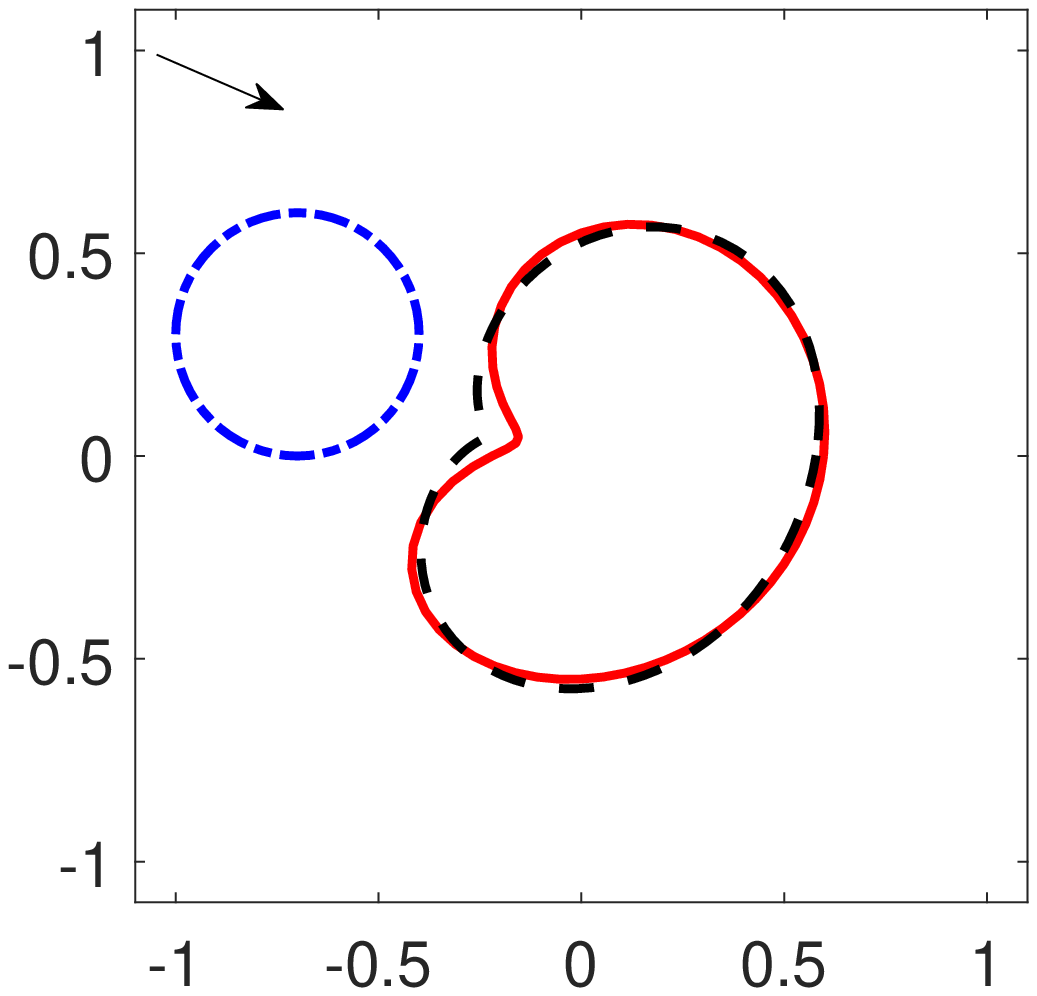}}
\subfigure[Relative error with $1\%$ noise]
{\includegraphics[width=0.4\textwidth]{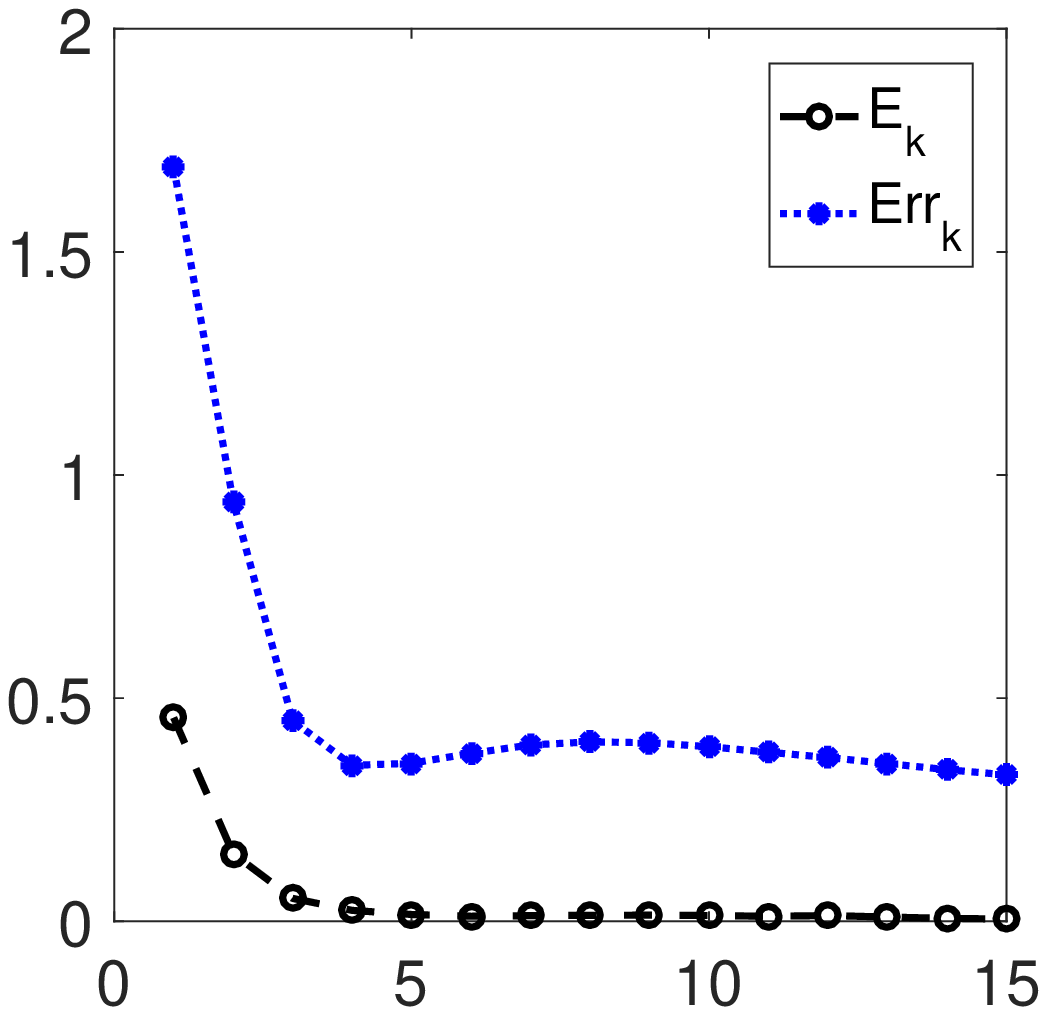}} 
\subfigure[Reconstruction with $5\%$ nois, $\epsilon=0.025$]
{\includegraphics[width=0.4\textwidth]{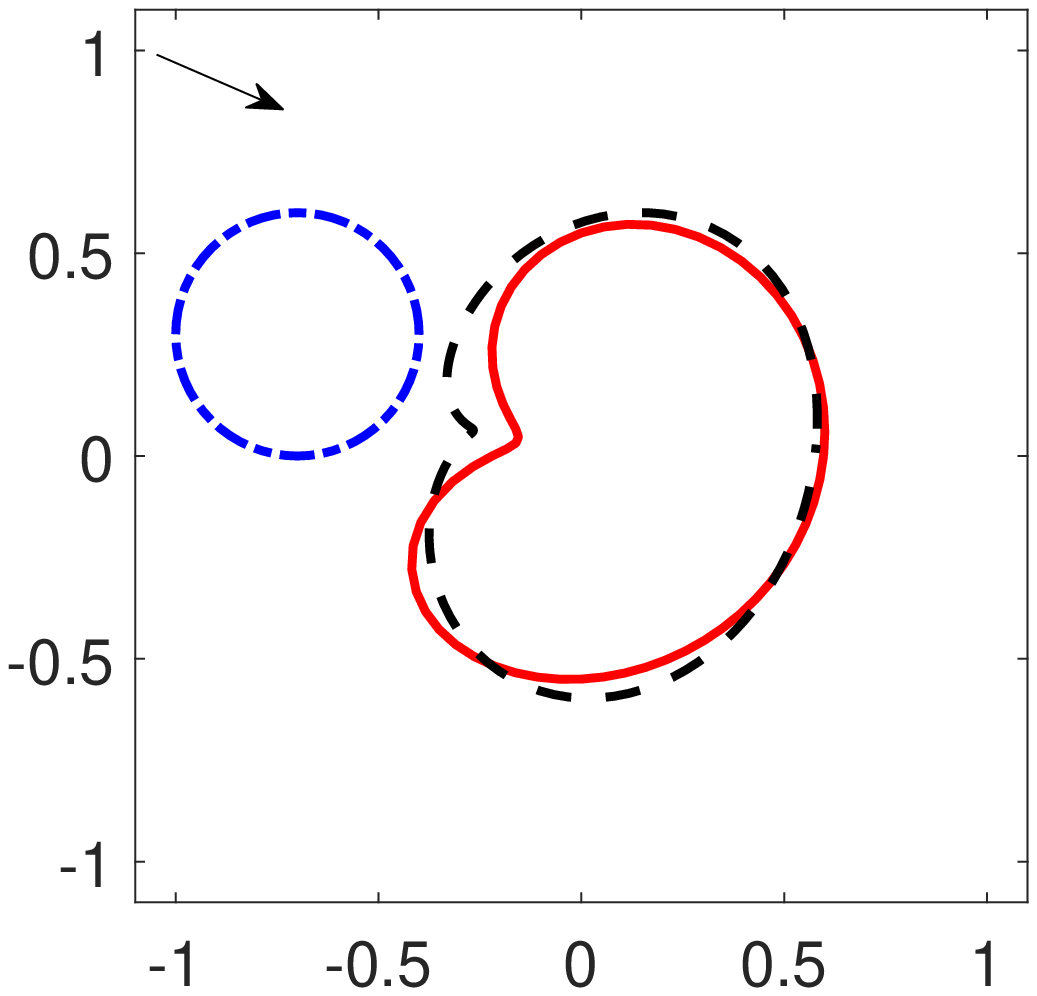}}
\subfigure[Relative error with $5\%$ noise]
{\includegraphics[width=0.4\textwidth]{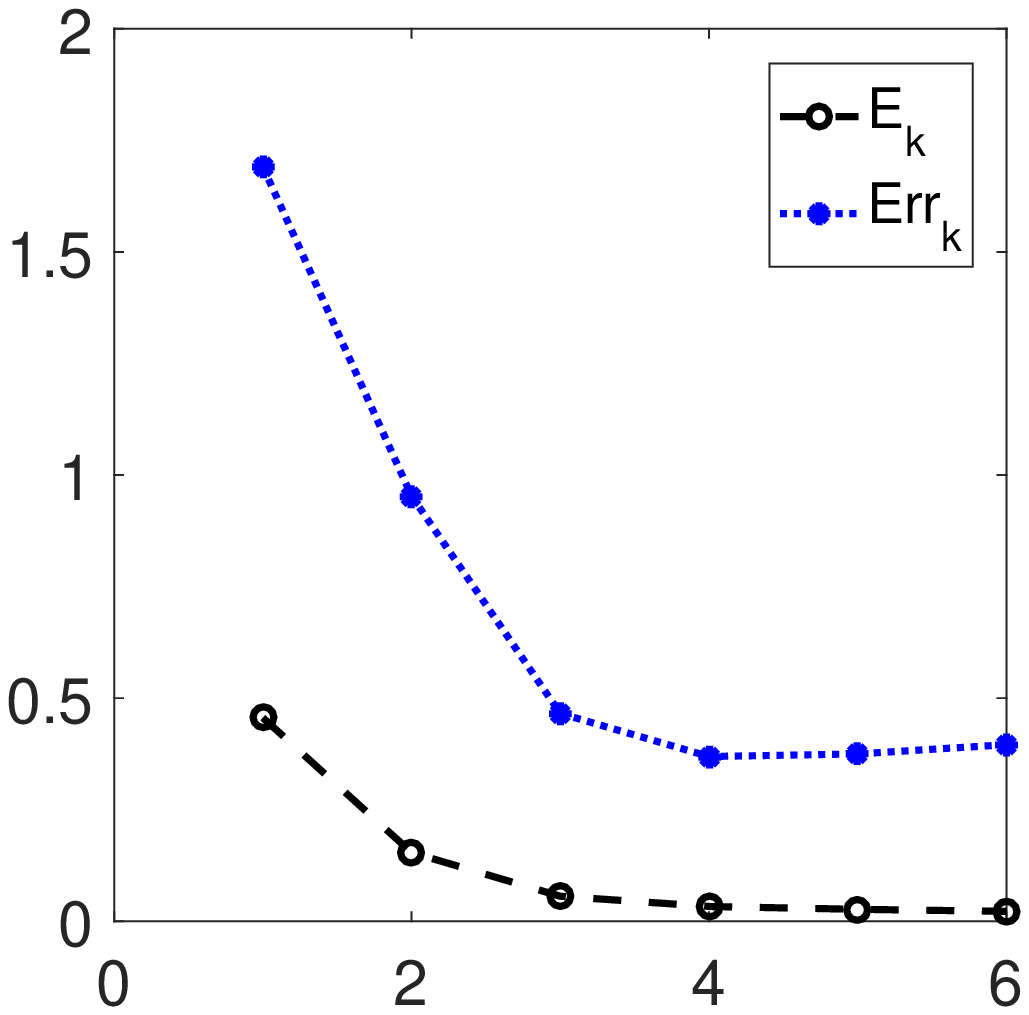}}
\caption{Reconstructions of an apple-shaped obstacle with different
levels of noise by using phaseless data and a reference ball (see example 3).
The initial guess is given by 
$(c_1^{(0)},c_2^{(0)})=(-0.7, 0.3), r^{(0)}=0.3$, the incident angle 
$\theta=11\pi/6$, and the reference ball is $(b_1,b_2)=(5, 0),
R=0.5$.}\label{PhaselessIOSP-15}
\end{figure}

\begin{figure}
\centering 
\subfigure[Reconstruction with $1\%$ noise, $\epsilon=0.02$ ]
{\includegraphics[width=0.4\textwidth]{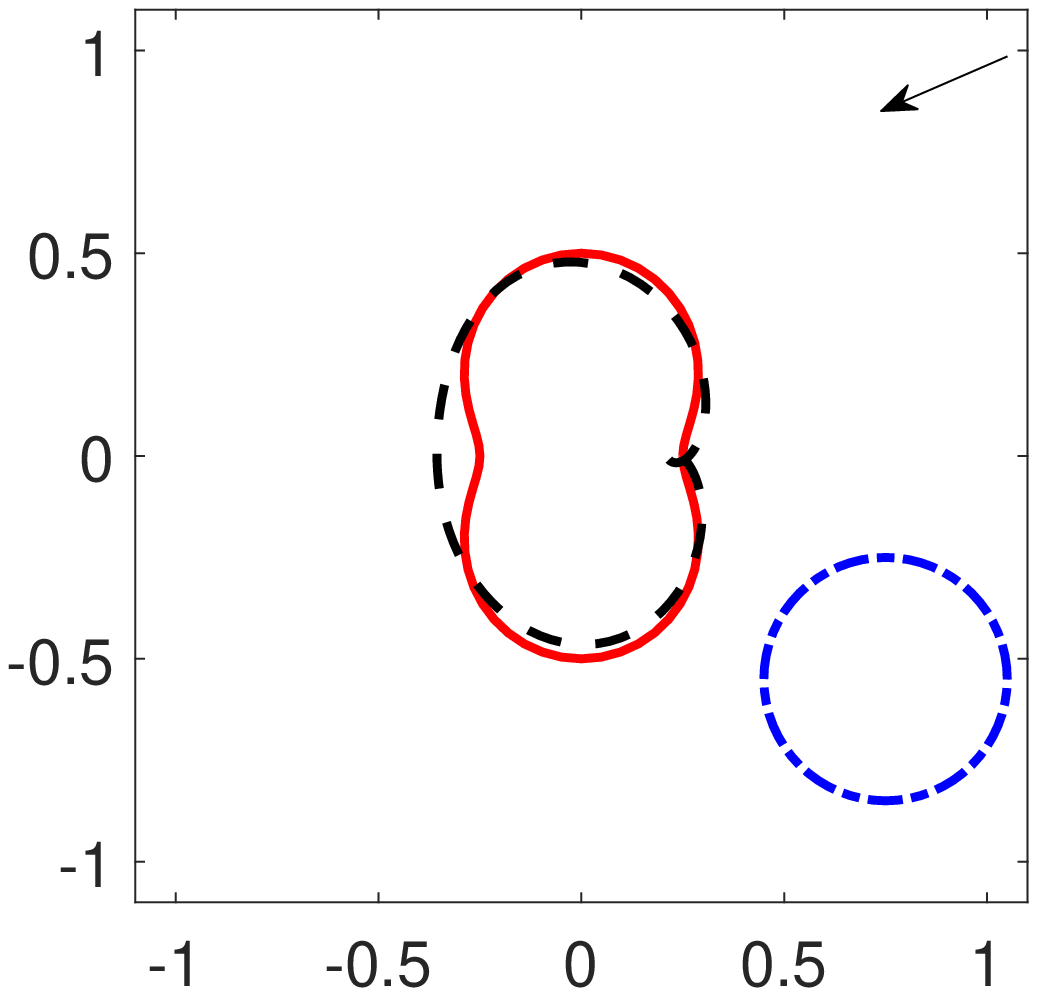}}
\subfigure[Relative error with $1\%$ noise]
{\includegraphics[width=0.4\textwidth]{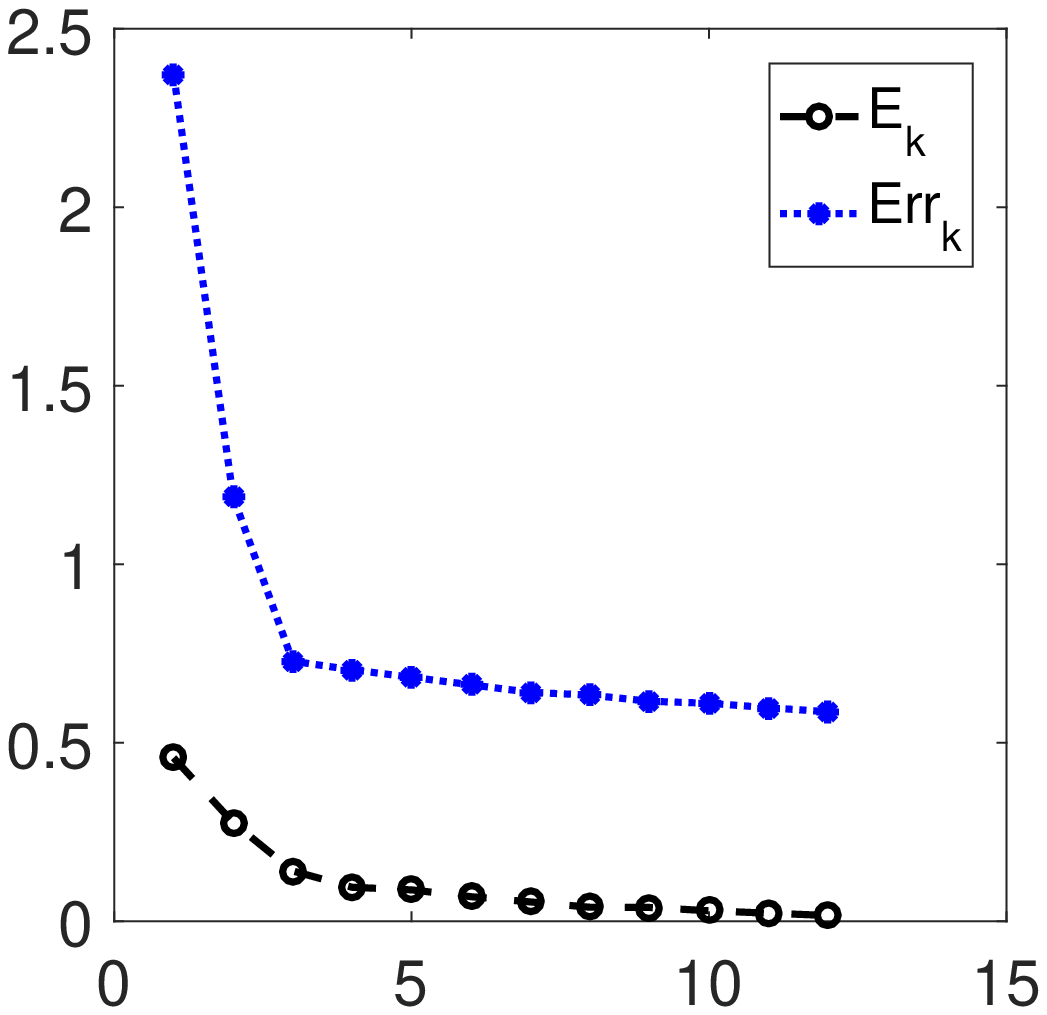}}
\subfigure[Reconstruction with $5\%$ nois, $\epsilon=0.04$]
{\includegraphics[width=0.4\textwidth]{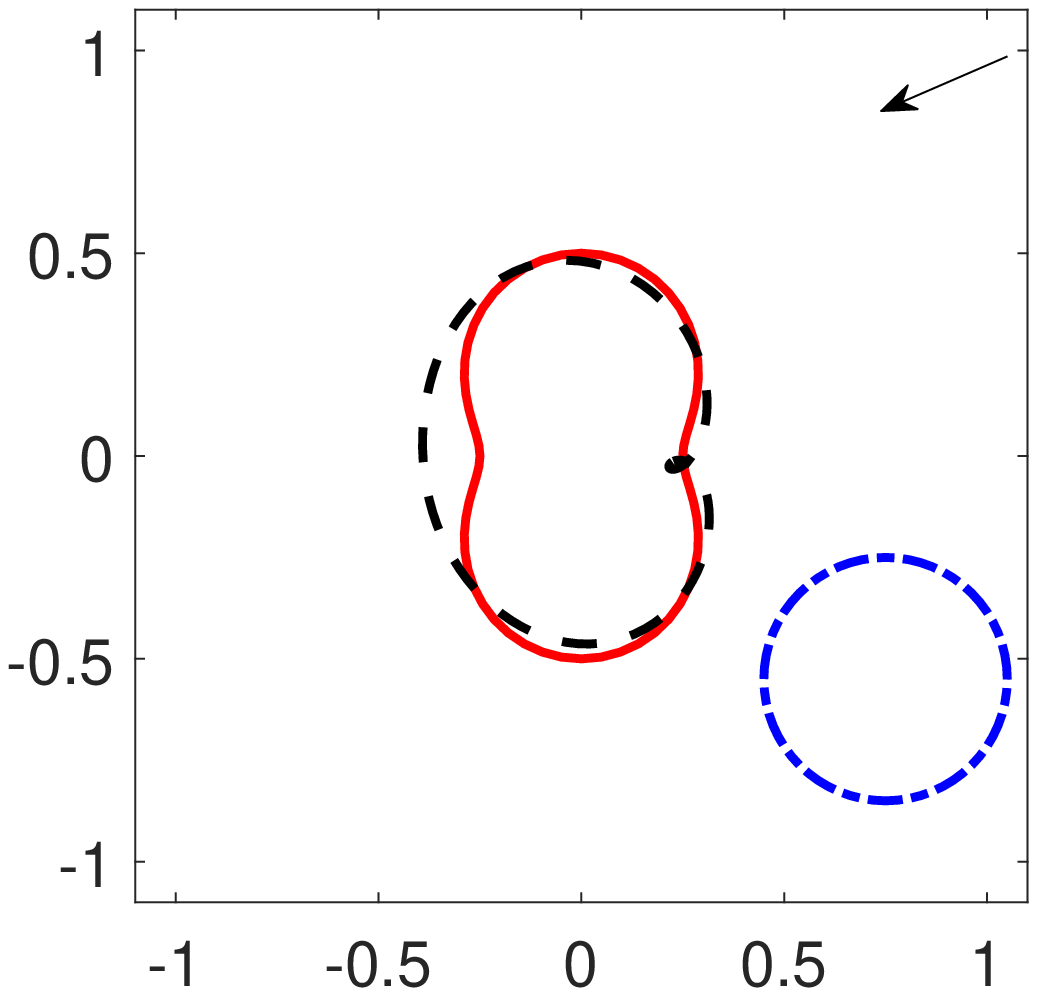}}
\subfigure[Relative error with $5\%$ noise]
{\includegraphics[width=0.4\textwidth]{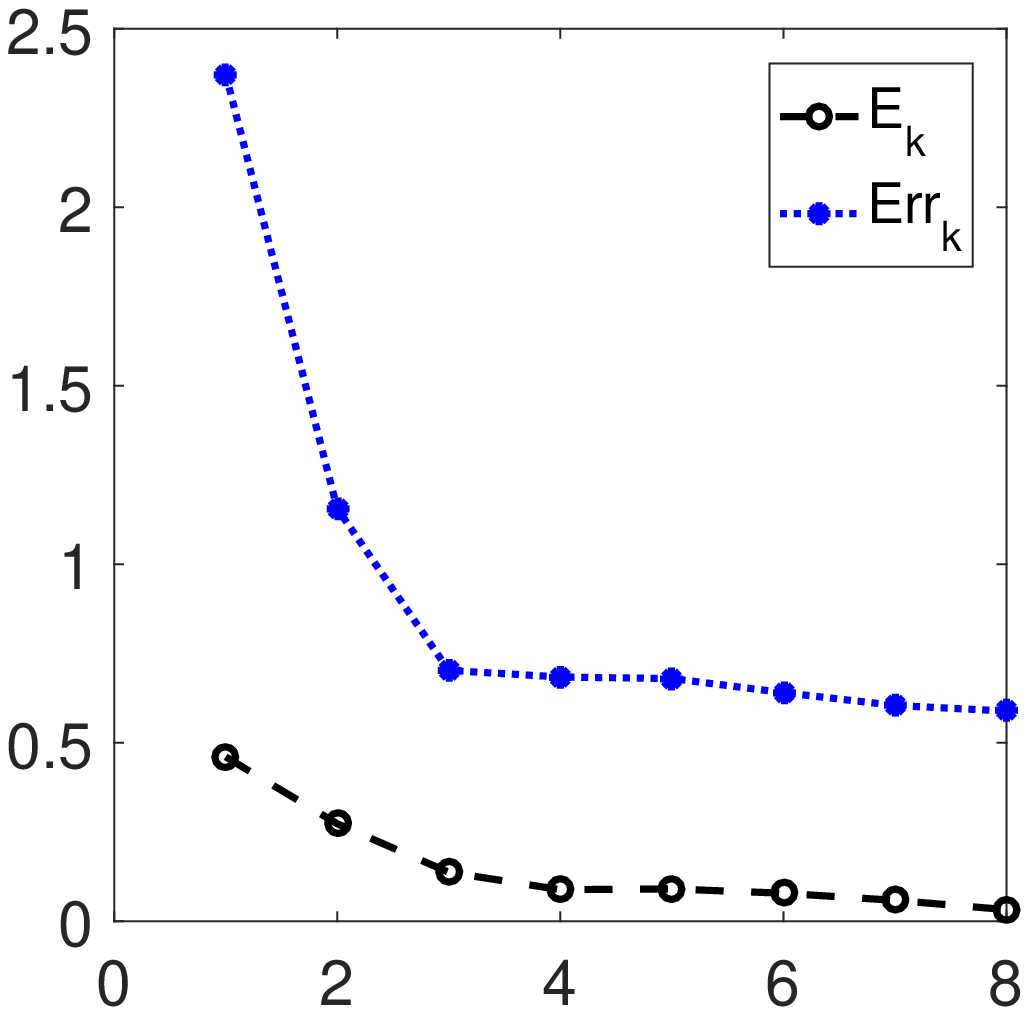}}
\caption{Reconstructions of a peanut-shaped obstacle with different
levels of noise by using phaseless data and a reference ball (see example 3).
The initial guess is given by 
$(c_1^{(0)},c_2^{(0)})=(0.75, -0.55), r^{(0)}=0.3$, the incident angle
$\theta=7\pi/6$, and the reference ball is $(b_1,b_2)=(9, 0),
R=0.5$.}\label{PhaselessIOSP-18}
\end{figure}

\begin{figure}
\centering
\includegraphics[width=0.4\textwidth]{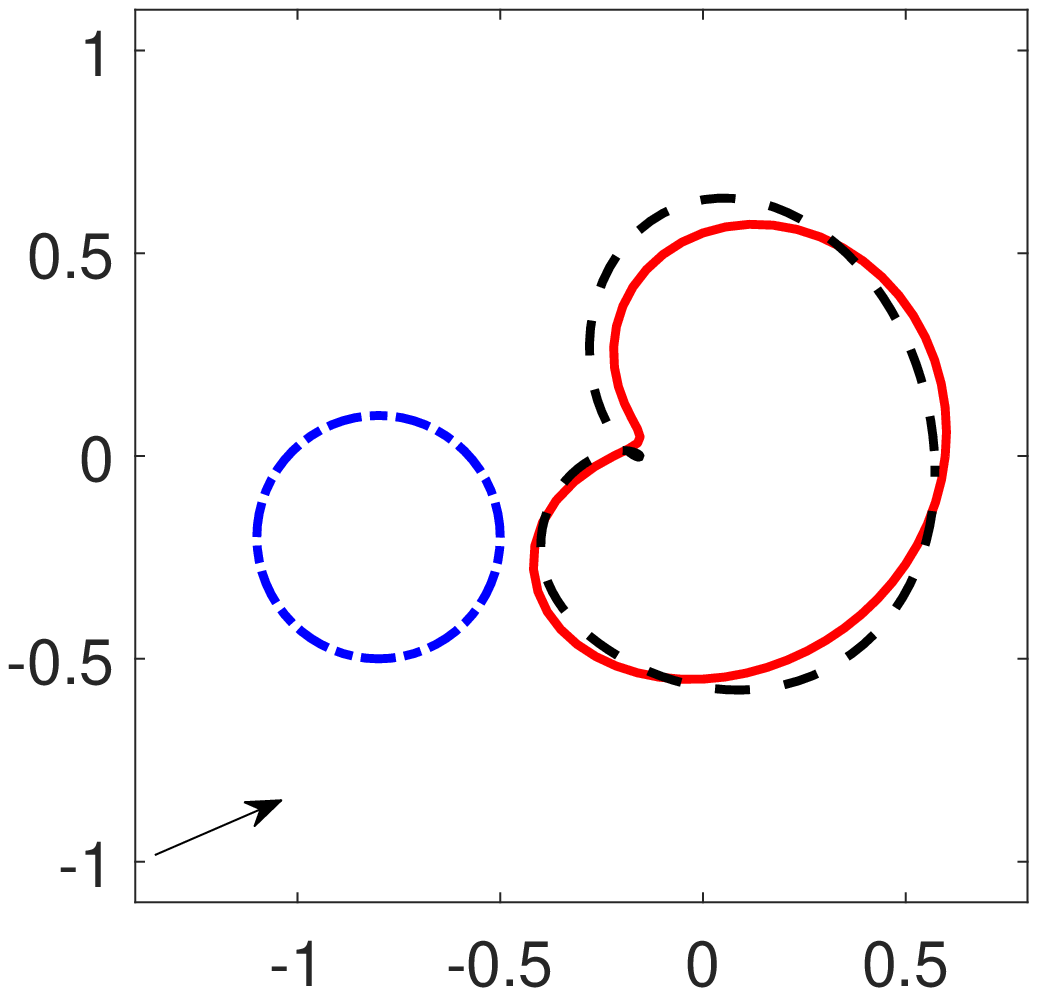}
\includegraphics[width=0.4\textwidth]{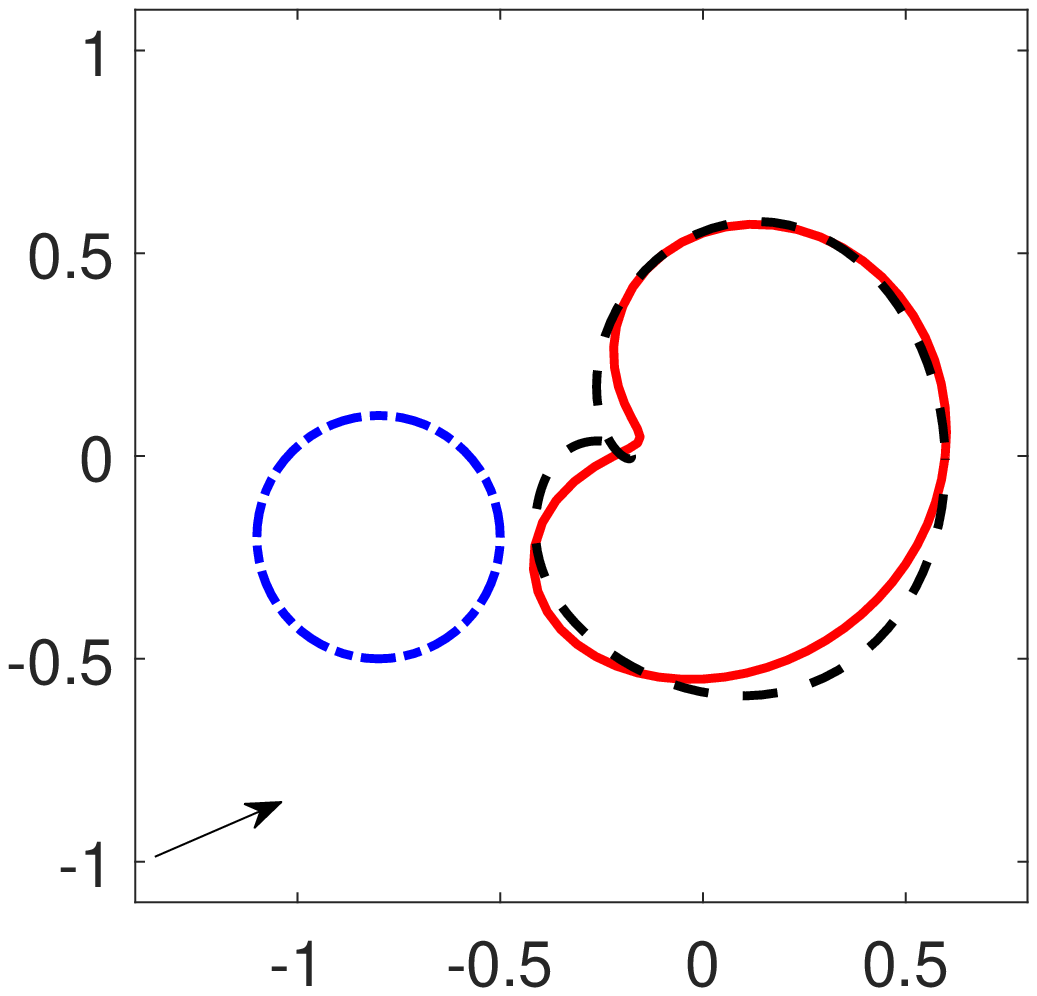}
\caption{Reconstructions of an apple-shaped obstacle with
different reference balls, where $1\%$ noise is added, the inciedent angle 
$\theta=\pi/6$, and the initial guess is given by $(c_1^{(0)},c_2^{(0)})=(-0.8,
0.2), r^{(0)}=0.3$ (see example 3). (left) $(b_1, b_2)=(5.5, 0), R=0.5$,
$\epsilon=0.01$;
(right) $(b_1, b_2)=(-8.5, 0), R=0.4$,
$\epsilon=0.015$.}\label{PhaselessIOSP-17}
\end{figure}

\begin{figure}
\centering 
\includegraphics[width=0.4\textwidth]{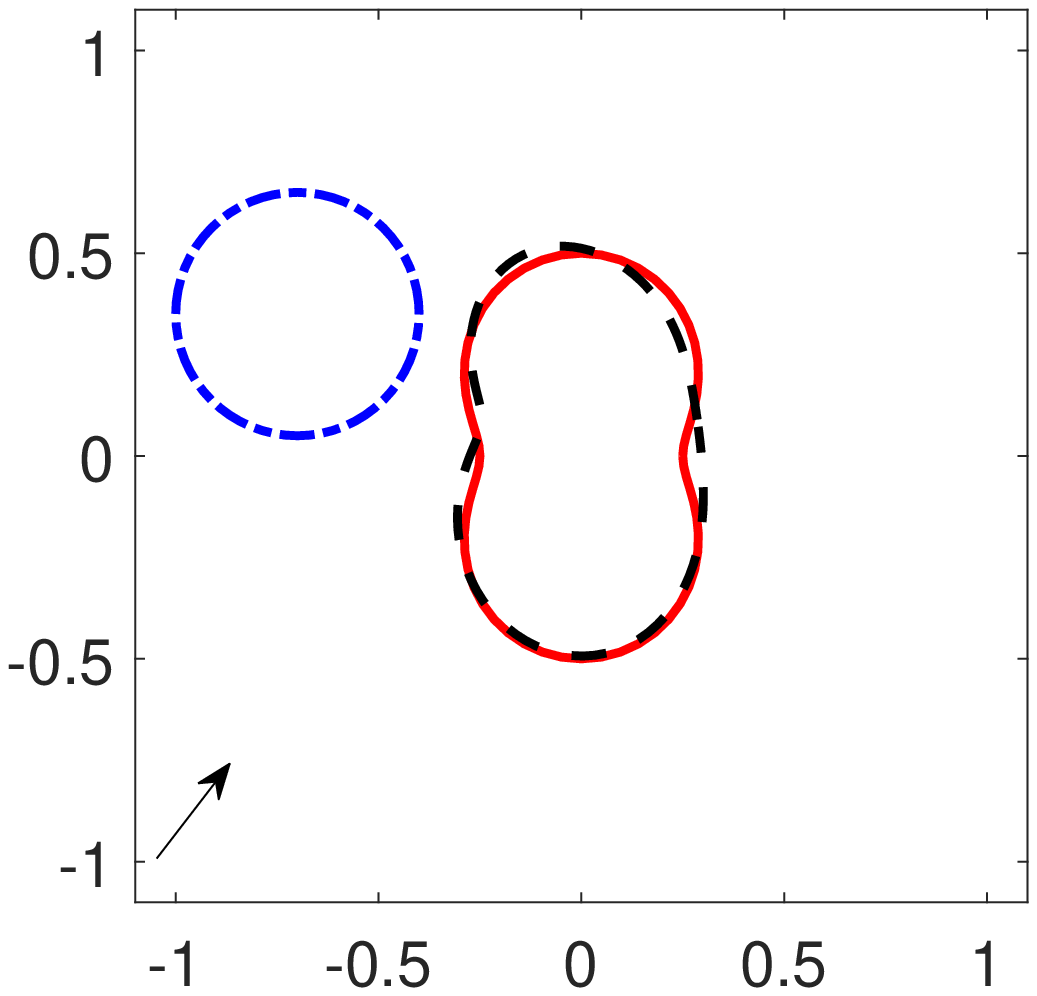}
\includegraphics[width=0.4\textwidth]{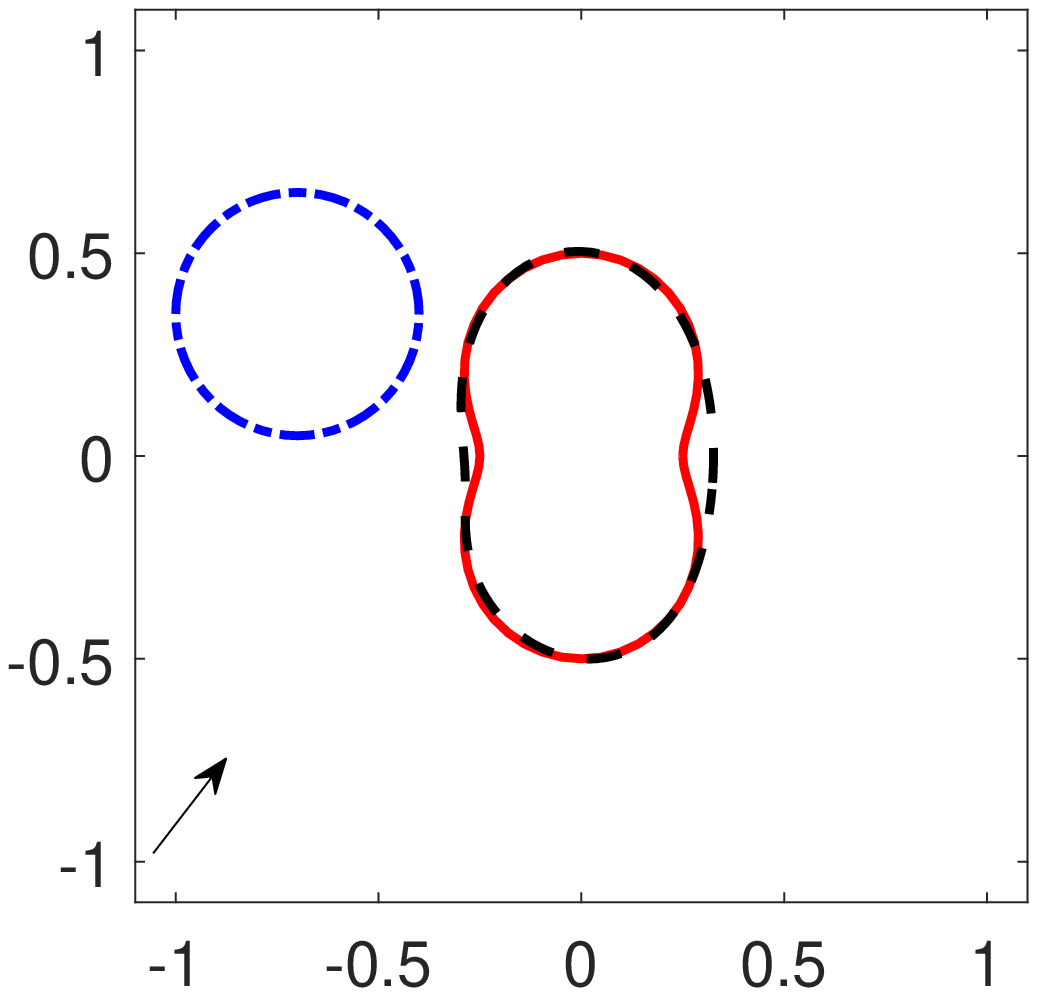}
\caption{Reconstructions of a peanut-shaped obstacle with different reference
balls, where $1\%$ noise is added, the inciedent angle $\theta=\pi/3$, and the
initial guess is given by $(c_1^{(0)},c_2^{(0)})=(-0.7, 0.35),
r^{(0)}=0.3$ (see example 3). (left) $(b_1, b_2)=(6, 0), R=0.6$,
$\epsilon=0.006$; (right)
$(b_1, b_2)=(7.5, 0), R=1.5$, $\epsilon=0.006$.}\label{PhaselessIOSP-20}
\end{figure}

\section{Numerical experiments}

In this section, we present some numerical examples to demonstrate the
feasibility of the proposed iterative reconstruction methods. In all the
examples, a
single shear plane wave is used to illuminate the obstacle. The synthetic
compressional far-field data are numerically
generated at 64 points, i.e., $\bar{n}=32$, by using another Nystr\"om method
based on Alpert quadrature to avoid the ``inverse crime"
\cite{LaiLi}. To test stability, the noisy data $u_{\infty,\delta}$ and
$|u_{\infty,\delta}|^2$ are constructed in the following way
\begin{align*}
u_{\infty,\delta}=u_{\infty}(1+\delta\breve\eta), \quad
|u_{\infty,\delta}|^2=|u^\infty|^2(1+\delta\eta),
\end{align*}
where $\breve\eta=\breve\eta_1 +\mathrm{i}\breve\eta_2$, $\breve\eta_1$,
$\breve\eta_2$ and $\eta$ are normally distributed random number ranging in
$[-1,1]$, $\delta>0$ is the relative noise level. In addition,  we denote the
$L^2$ relative error between the reconstructed and exact boundaries by
$$
{\rm Err}_k:=\frac{\|p_D^{(k)}-p_D\|_{L^2(\Omega)}}{\|p_D\|_{L^2(\Omega)}}.
$$

In the iteration, we obtain the update $\xi$ from a scaled Newton step by using 
the Tikhonov regularization and $H^2$ penalty term, i.e., 
$$
\xi=\rho\Big(\lambda\widetilde{I}+\Re(\widetilde{B}^*\widetilde{B})\Big)^{-1}
\Re(\widetilde{B}^*\widetilde{w}),
$$
where the scaling factor $\rho\geq0$ is fixed throughout the iterations.
Analogous to \cite{DZhG2018}, the regularization parameters $\lambda$ in
equation \eqref{EqualRLHuygens3} are chosen as
\[
\lambda_k:=\left\|a_{\rm p} \Big(\phi_\infty-S^\infty_{\kappa_{\rm p}}
(p^{(k-1)},\varphi_1^{(k-1)})\Big)+a_{\rm s}
\Big(\psi_\infty-S^\infty_{\kappa_{\rm s}}
(p^{(k-1)},\varphi_2^{(k-1)})\Big)\right\|_{L^2},\ k=1,2,\cdots.
\]

In all of the following figures, the exact boundary curves are displayed in
solid lines, the reconstructed boundary curves are depicted in dashed
lines $--$, and all the initial guesses are taken to be a circle with radius
$r^{(0)}=0.3$ which is indicated in the dash-dotted lines $\cdot-$. The incident
directions are denoted by directed line segments with arrows. Throughout all the
numerical examples, we take $\lambda=3.88, \mu=2.56$, the scaling factor
$\rho=0.9$, and the truncation $M=6$. The number of quadrature points is equal
to 128, i.e., $n=64$. In addition, we choose the angular frequency
$\omega=0.7\pi$ in Example 1 and $\omega=0.6\pi$ in Example 2 and Example 3. We
present the results for two commonly used examples: an apple-shaped obstacle and
a peanut-shaped
obstacle. The parametrization of the exact boundary curves for these two
obstacles are given in Table \ref{partable}.

\subsection{Example 1: the Phased IOSP}

We consider an inverse elastic scattering problem of reconstructing a rigid
obstacle from the phased far-field data by using Algorithm I. In
Figures \ref{IOSP-2} and \ref{IOSP-5}, the results are shown for the
apple-shaped obstacle and the peanut-shaped obstacles with $1\%$ and
$5\%$ noise, respectively. The relative $L^2$ error ${\rm Err}_k$
between the reconstructed obstacle and the exact obstacle and the relative
error estimator ${\rm E}_k$ defined in \eqref{relativeerror} are plotted
against the number of iterations. As can be seen from the error curves, the
relative error estimator ${\rm E}_k$ follows the actual relative error ${\rm
Err}_k$ very well and is a reasonable choice of the stopping criteria for the
iterations. For the fixed incident direction, Figures \ref{IOSP-3} and
\ref{IOSP-6} show the reconstructions of the apple-shaped obstacle and the
peanut-shaped obstacle by using different initial guesses; for the fixed initial
guess, Figures \ref{IOSP-4} and \ref{IOSP-7} show the reconstructions of the
apple-shaped obstacle and the peanut-shaped obstacle by using different
incident directions. As shown in these results, the reconstruction is not
sensitive to the initial guess or the incident direction. The location and shape
of the obstacle can be simultaneously and satisfactorily reconstructed for a
single incident plane wave.

\subsection{Example 2: the phased IOSP with a reference ball}

Now we investigate the inverse scattering problem of reconstructing a rigid
obstacle from the phased far-field data by introducing a reference ball. The
reconstructions with $1\%$ noise and $5\%$ noise for the apple-shaped and
peanut-shaped obstacles are shown in Figure \ref{IOSPreference-8} and Figure
\ref{IOSPreference-12}, respectively. The relative $L^2$ error ${\rm Err}_k$
and the relative error estimator ${\rm E}_k$ are also presented in the figures.
Tests are also done by using different initial guesses and different
incident directions. In addition, we test the influence by using reference balls
with different
radius and location. For the fixed initial guess and incident direction, Figures
\ref{IOSPreference-11} and \ref{IOSPreference-14} show the reconstructions of
the apple-shaped obstacle and the peanut-shaped obstacle by using different
reference balls. The reconstructed obstacles agree very well with exact ones. As
can be seen, the results by using the reference ball technique are
comparable with those without the reference ball in Example 1. The method
works very well to reconstruct the location and the shape.

\subsection{Example 3: the phaseless IOSP with a reference ball}

By adding a reference ball to the inverse scattering system, we consider the
inverse scattering problem of reconstructing a rigid obstacle from phaseless
far-field data by using Algorithm II. The
reconstructions with $1\%$ noise and $5\%$ noise are shown in Figures
\ref{PhaselessIOSP-15}--\ref{PhaselessIOSP-18}. Again, the relative $L^2$ error
${\rm Err}_k$ and the relative error estimator ${\rm E}_k$ are plotted in the
figures. Figures \ref{PhaselessIOSP-17} and \ref{PhaselessIOSP-20} show the
reconstructions of the apple-shaped obstacle and the peanut-shaped obstacle by
using different reference balls. From these figures, we observe that the
translation invariance property of the phaseless far-field pattern can be broken
by introducing a reference ball. Based on this technique, both of the location
and shape of the obstacle can be satisfactorily reconstructed from the phaseless
far-field data by using a single incident plane wave.

\section{Conclusions and future works}

In this paper, we have studied the two-dimensional inverse elastic scattering
problem by the phased and phaseless far-field data for a single incident plane
wave. Based on the Helmholtz decomposition, the elastic wave equation is
reformulated into a coupled boundary value problem of the Hemholtz equation. The
relationship between compressional or shear far-field pattern for the Navier
equation and the corresponding far-field pattern for the Helmholtz equation are
investigated. The translation invariance property of the phaseless
compressional and shear far-field pattern is proved. A nonlinear integral
equation method is developed for the inverse problem. For the phaseless data, we
introduce a reference ball technique to the inverse
scattering system in order to break the translation invariance. Numerical
examples are presented to demonstrate the effectiveness and stability of the
proposed method. Future work includes the different boundary conditions of the
obstacle and the three-dimensional problem. It is a challenging problem for the
uniqueness of the phaseless inverse elastic scattering problem with a reference
ball. We intend to investigate these issues and report the progress elsewhere.

\end{document}